\documentclass[11pt]{article}

\usepackage{latexsym}
\usepackage{amssymb}
\usepackage{amsthm}
\usepackage{amscd}
\usepackage{amsmath}
\usepackage{mathrsfs}
\usepackage{graphicx}
\usepackage[colorlinks=true, pdfstartview=FitV, linkcolor=blue, citecolor=blue, urlcolor=blue]{hyperref}
\usepackage{shuffle}
\usepackage[centertableaux]{ytableau}
\usepackage{stmaryrd}
\usepackage{tikz-cd}

\usepackage[colorinlistoftodos]{todonotes}
\usetikzlibrary{chains,scopes,decorations.markings}

\usepackage{mathdots}

\theoremstyle{definition}
\newtheorem* {theorem*}{Theorem}
\newtheorem* {conjecture*}{Conjecture}
\newtheorem{theorem}{Theorem}[section]

\theoremstyle{definition}

\newtheorem* {example*}{Example}

\newtheorem{lemma}[theorem]{Lemma}
\theoremstyle{definition}
\newtheorem{definition}[theorem]{Definition}
\theoremstyle{definition}

\newtheorem{proposition}[theorem]{Proposition}
\newtheorem{corollary}[theorem]{Corollary}

\newtheorem*{remark*}{Remark}
\theoremstyle{definition}
\newtheorem{remark}[theorem]{Remark}
\theoremstyle{definition}
\newtheorem {example}[theorem]{Example}
\theoremstyle{definition}

\theoremstyle{definition}

\theoremstyle{definition}
\newtheorem{problem}[theorem]{Problem}
\theoremstyle{definition}

\newcommand{\ytabc}[2]{
\ytableausetup{boxsize = #1,aligntableaux=center}
{\small\begin{ytableau}  #2  \end{ytableau}}
}

\def\modu{\ (\mathrm{mod}\ }

\def\({\left(}
\def\){\right)}

\newcommand{\CC}{\mathbb{C}}
\newcommand{\QQ}{\mathbb{Q}}
\newcommand{\cP}{\mathcal{P}}
\newcommand{\cQ}{\mathcal{Q}}
\newcommand{\cK}{\mathcal{K}}

\newcommand{\cR}{\mathcal{R}}

\newcommand{\cS}{\mathcal{S}}

\newcommand{\cC}{\mathcal{C}}

\def\cS{\mathcal{S}}

\def\NN{\mathbb{Z}_{\geq 0}}

\def\Hom{\mathrm{Hom}}

\def\CC{\mathbb{C}}
\def\RR{\mathbb{R}}
\def\ZZ{\mathbb{Z}}
\def\HH{\mathbb{H}}

\def\GL{\mathrm{GL}}

\def\ch{\mathsf{ch}}
\def\spanning{\textnormal{-span}}

\newcommand{\abs}[1]{\lvert #1 \rvert}

\newcommand{\Flag}{\mathcal{F}}  

\newcommand{\supp}{\mathrm{supp}}

\newcommand{\cL}{\mathcal{L}}
\newcommand{\g}{\mathfrak{g}}

\newcommand{\one}{{1\hspace{-.12cm} 1}}

\def\fk{\mathfrak}

\def\barr{\begin{array}}
\def\earr{\end{array}}
\def\ba{\begin{aligned}}
\def\ea{\end{aligned}}
\def\be{\begin{equation}}
\def\ee{\end{equation}}

\def\qquand{\qquad\text{and}\qquad}
\def\quand{\quad\text{and}\quad}

\newcommand{\gl}{\mathfrak{gl}}

\def\id{\mathrm{id}}
\def\PP{\mathbb{P}}

\def\fkS{\fk S}

\def\ben{\begin{enumerate}}
\def\een{\end{enumerate}}

\def\bei{\begin{itemize}}
\def\eei{\end{itemize}}

\def\cT{\mathcal{T}}
\def\cE{\mathcal E}

\def\x{\textbf{x}}
\def\y{\textbf{y}}

\def\e{\textbf{e}}

\newcommand{\xRightarrow}[2][]{\ext@arrow 0359\Rightarrowfill@{#1}{#2}}

\newcommand{\poly}{\operatorname{poly}}

\newcommand{\esf}{\mathsf{e}}  

\def\R{R}

\newcommand{\cA}{\mathcal{A}}
\newcommand{\cB}{\mathcal{B}}

\def\arcstart{\ \xy<0cm,-.06cm>\xymatrix@R=.1cm@C=.10cm }
\newcommand{\arcstartc}[1]{\ \xy<0cm,-.15cm>\xymatrix@R=.1cm@C=#1cm}

\newcommand{\bfa}{\mathbf{a}}
\newcommand{\ii}{\mathbf{i}}
\newcommand{\mm}{\mathbf{m}}

\newcommand{\bfN}{\mathbf{N}}

\def\sC{\mathscr{C}}

\definecolor{darkred}{rgb}{0.7,0,0} 
\newcommand{\defn}[1]{{\color{darkred}\emph{#1}}} 

\newcommand{\weight}{\operatorname{wt}}

\def\q{\mathfrak{q}}

 \def\PP{\ZZ_{>0}}

\def\row{\mathsf{row}}

\def\col{\mathsf{col}}

\def\D{\mathsf{D}}

\def\sqgln{\sqrt{\gl_n}}

\def\SetTab{\mathsf{SetTab}}
\def\SS{\mathbb{S}}

\numberwithin{equation}{section}
\allowdisplaybreaks[1]

\newcommand{\zero}{\boldsymbol{0}}
\newcommand{\iso}{\cong}

\newcommand{\bal}{\boldsymbol{\alpha}}

\newcommand{\powF}[1]{\mathtt{F}^{(#1)}}  
\newcommand{\defectF}[1]{\mathtt{F}_+} 
\newcommand{\restrictF}[1]{\mathtt{R}_{#1}} 

\newcommand{\defect}{\mathsf{def}}
\newcommand{\insmap}{\mathsf{insert}}

\newcommand{\sK}{\mathscr{K}}

\newcommand{\uj}[2]{\{#2\}_{#1}}
\newcommand{\elem}[2]{\langle #2 \rangle_{#1}}

\newcommand{\fkD}{\mathfrak{D}}

\usepackage{nicefrac}
\newcommand{\halfpower}{\nicefrac{1}{2}} 

\newcommand{\sqrtB}{\cB^{\halfpower}}
\newcommand{\sqrtE}{\cP^{\halfpower}}

\newcommand{\VCrys}[1]{\mathsf{Vec}^{\halfpower}_n(#1)}

\newcommand{\VCrysRN}[2]{\mathsf{Vec}^{\halfpower}_{#1,#2}}
\newcommand{\VCrysR}[1]{\VCrysRN{#1}{n}}  

\newcommand{\PCrysR}[1]{\mathsf{LP}^{\halfpower}_{#1,n}}  

\newcommand{\BZL}{\mathrm{BZL}}
\usepackage{fullpage}

\usepackage{comment}

\begin{document}
\title{Square root crystals and the square root of $\cB(\infty)$}
\author{
Eric Marberg \\ Department of Mathematics \\  Hong Kong University of Science and Technology \\ {\tt eric.marberg@gmail.com}
\and 
Travis Scrimshaw \\ Department of Mathematics \\ Hokkaido University \\ {\tt tcscrims@gmail.com}
}

\date{}

\makeatletter
\newcommand{\subjclass}[2][1991]{%
  \let\@oldtitle\@title%
  \gdef\@title{\@oldtitle\footnotetext{#1 \emph{Mathematics subject classification:} #2}}%
}
\newcommand{\keywords}[1]{%
  \let\@@oldtitle\@title%
  \gdef\@title{\@@oldtitle\footnotetext{\emph{Key words and phrases:} #1.}}%
}
\makeatother

\subjclass[2020]{05E05, 17B37, 05A19, 14M15}

\keywords{square root crystal, Grothendieck polynomial, set-valued tableau, direct limit crystal}

\maketitle

\begin{abstract}
We introduce a general monoidal category of $\bfN$-root crystals and then study the special case of square root $\gl_n$-crystals.
The latter objects include Yu's crystals on semistandard set-valued tableaux.
Prior work of the first author, Tong, and Yu showed that regular square root $\gl_n$-crystals can be a useful tool for proving Grothendieck positivity results.
The objects studied here go beyond the regular case and allow us to construct a square root analog of the direct limit crystal $\cB(\infty)$.
We give several descriptions of our square root of $\cB(\infty)$, using marginally large tableaux, the Lusztig or PBW parameterization, and the Nakashima--Zelevinsky polyhedral model.
We show that this crystal has a simple character formula, exhibits a nontrivial Demazure filtration, and recovers Yu's semistandard set-valued tableau crystals after taking appropriate tensor products.
We also investigate a number of differences between square root crystals and classical crystal constructions.
\end{abstract}

\tableofcontents

\section{Introduction}

The recent papers \cite{Yu23,MT2023,MTY} introduce and study \defn{square root crystals}, which are $K$-theoretic generalizations of Kashiwara crystals in type A.
Square root crystals have roughly the same relationship to set-valued tableaux and symmetric Grothendieck functions as their classical predecessors do to semistandard tableaux and Schur polynomials. 
 
The goals of this article are twofold.
First, we give a systematic treatment of a more general category of square root crystals, outside the regular case considered in \cite{MT2023,MTY}.
Second, we study several constructions of a ``square root'' analog of Kashiwara's $\cB(\infty)$-crystal.
The rest of this introduction presents our motivations and main results in more detail.

\subsection{Classical background}
The \defn{complete flag variety} $\Flag$ is the space of all strictly increasing chains of vector spaces
$
\{0\} = V_0 \subsetneq V_1 \subsetneq V_2 \subsetneq \cdots \subsetneq V_n = \CC^n
$. 
This space
can be realized as the quotient $\GL_n/ B$, where $\GL_n$ is defined over $\CC$ and $B$ is the Borel subgroup of upper triangular matrices.
The group $B$ acts on $\Flag$ and the orbits $X_w^{\circ}$ of this action are parametrized by permutations $w \in S_n$. These orbits are affine spaces, and their closures $X_w := \overline{X_w^{\circ}} = \bigsqcup_{v \leq w} X_v^{\circ}$ are the well-studied \defn{Schubert varieties}, which give rise to a nice filtration on $\Flag$ under Bruhat order.

Each $X_w$ determines a class $[X_w]$ in the cohomology ring $H^{\bullet}(\Flag; \ZZ)$, 
and the set of these classes provides a $\ZZ$-basis for the ring.
By a classical theorem of Borel, $H^{\bullet}(\Flag; \ZZ)$ is isomorphic to the coinvariant algebra $\ZZ[x_1, \dotsc, x_n] / \ZZ[x_1, \dotsc, x_n]_+^{S_n}$, where $\ZZ[x_1, \dotsc, x_n]_+^{S_n}$ is the ideal generated by all symmetric polynomials $f$ with $f(0,\dotsc,0) = 0$.
Lascoux and Sch\"utzenberger~\cite{LS82,LS83} introduced a remarkable family of \defn{Schubert polynomials} $\fkS_w$ that represent the classes $[X_{w_0w}]$ under the Borel isomorphism, where $w_0 = [n, \dotsc, 2, 1] \in S_n$ is the unique permutation with the most inversions.
When restricted to Grassmannian permutations, these polynomials recover the \defn{Schur functions} $s_{\lambda}(x)$.

Schubert calculus is the study of these objects, and it is intimately related with the representation theory of Lie groups and algebras.
A particularly important result in this area is the Borel--Weil(--Bott) theorem (see, e.g.,~\cite[\S23.3]{FH91}), which constructs the highest weight $\GL_n$-representation $V(\lambda) \cong H^0(\Flag,\cL_\lambda)^\ast$ for dominant integral weights $\lambda$ from the dual of the global sections of the line bundle $ \cL_\lambda = \GL_n\times_B \CC_{-\lambda}$, where  $\CC_{-\lambda}$ is a $1$-dimensional $B$-module.
On $\cL_{\lambda} \rvert_{X_w}$, the restriction of the line bundle to $X_w$, one naturally obtains a $B$-representation $V(\lambda)_w \iso H^0(X_w, \cL_{\lambda}\rvert_{X_w})^\ast$, which is known as a \defn{Demazure module}.

Since $\Flag = X_{w_0}$, we have a filtration of $V(\lambda)$ by Demazure modules.
Using properties of parabolic subgroups of $\GL_n$, Demazure (see~\cite{Andersen85} or \cite{LMS79}) gave a character formula for $V(\lambda)_w$ that was analogous to the construction of Schubert polynomials.
This \defn{Demazure character formula} refines the Weyl character formula for Schur functions $s_{\lambda}(x)$, which are the characters of the $\GL_n$-representations $V(\lambda)$. 

The categories of highest weight representations $V(\lambda)$ for $\GL_n$, its Lie algebra $\gl_n$, and the associated (Drinfel'd--Jimbo) quantum group $U_q(\gl_n)$ are essentially equivalent~\cite{Bourbaki98,Joseph95}.
By using the quantum group, Kashiwara's theory of crystal bases~\cite{Kashiwara90,Kashiwara91} (which is a shadow of Lusztig's theory of canonical bases~\cite{Lusztig90,Lusztig91,GL93}) constructs for $V(\lambda)$ a combinatorial realization $\cB(\lambda)$, which we refer to as a \defn{$\gl_n$-crystal}.
The object $\cB(\lambda)$ consists of a directed graph with edges labeled by integers $i\in\{1,2,\dots,n-1\}$ and vertices weighted by tuples in $\ZZ^n$.
Subsequent work~\cite{Kashiwara93,Littelmann95} showed that $\cB(\lambda)$ naturally restricts to describe $V(\lambda)_w$ 
and further abstractions of $\gl_n$-crystals led to important constructions such as the direct limit $\cB(\infty)$ of $\cB(\lambda)$.
These developments also yield crystals for Verma modules, which are the induced highest-weight modules in the Bernstein--Gelfand--Gelfand category $\mathcal{O}$ (see, e.g.,~\cite{Humphreys08} for more information).

\subsection{$K$-theoretic generalizations}

A recent topic of interest has been $K$-theoretic Schubert calculus, where cohomology is replaced by (connective) $K$-theory. 
A common theme in this subject is to find  combinatorial constructions that can recover the classical case while also extending to the $K$-theoretic setting.
For instance, the $K$-theoretic analogs of Schubert polynomials are the Grothendieck polynomials $\mathfrak{G}_w(x;\beta)$~\cite{LS82,LS83,LS90,FK1994}, which represent the connective $K$-theory classes of the varieties $X_{w_0w}$ and which restrict in the Grassmannian case to the symmetric Grothendieck polynomials $G_{\lambda}(x; \beta)$.

The functions $G_{\lambda}(x; \beta)$ have been well-studied, with two key results being their Schur positivity~\cite{Lenart00} and a \defn{semistandard set-valued tableau (SSVT)} generating function formula~\cite{Buch2002}.
There is a $K$-theoretic analog of Demazure characters for $G_{\lambda}(x; \beta)$ known as \defn{Lascoux polynomials}~\cite{Kirillov16,Lascoux01,Monical16,RY15}, which have combinatorial formulas given by~\cite{BSW20} and~\cite{SY23}.

A natural problem is to find a way to encode the Lascoux polynomials using a $\gl_n$-crystal structure (see~\cite[Prob.~7.13]{MPS21}).
A $\gl_n$-crystal on SSVT was constructed for this in~\cite{MPS21}. This crystal identifies the Schur decomposition of $G_{\lambda}(x; \beta)$,
and when $\lambda$ is a rectangular partition, it can be extended to encode $L_{w\lambda}(x; \beta)$~\cite{MPS21,PS22}. However, there is no natural way to obtain such an encoding for general $\lambda$; see~\cite[Prop.~7.11]{MPS21}.

Yu~\cite{Yu23} found a solution to this problem by using the same set of tableaux but relaxing the axioms of a $\gl_n$-crystal.
More specifically, Yu defined a digraph structure on SSVT, which is reminiscent of a $\gl_n$-crystal and was later given the name \defn{square root crystal ($\sqgln$-crystal)} in \cite{MT2023}.
He then showed that this structure recovers the $K$-theoretic analog of the Demazure filtration on $V(\lambda)$ for all partitions $\lambda$~\cite[Thm.~5.1]{Yu23}.

We denote the $\sqgln$-crystal on SSVT of shape $\lambda$ and max entry $n$ by $\SetTab_n(\lambda)$.
Yu's original construction depended on a new, more complicated, signature rule.
This suggested a tensor product rule, where $\SetTab_n(\lambda)$ is a certain connected subcrystal built from a tensor power of a crystal $\SS_n$ of nonempty subsets of $\{1,2,\dots,n\}$.
This approach was explored in~\cite{MT2023,MTY}, where it turns out that the tensor product is the usual tensor product rule for Kashiwara crystals (see Definition~\ref{def:tensor_product} and Remark~\ref{no-loss-rem}).

While $\SetTab_n(\lambda) \subseteq \SS_n^{\otimes \abs{\lambda}}$, the crystals obtained in general are much more intriguing. They can have multiple highest weights and their isomorphism classes are not determined by the highest weights (see  Examples~\ref{ex:not_monoidal} and~\ref{ex:weight_not_isomorphism}).
Moreover, the Demazure crystals obtained are distinct from those in~\cite{PS22} (see Example~\ref{ex:PS_nondem}).
Despite this, the characters of these crystals are still very well behaved~\cite{MTY} (see Theorem~\ref{thm:standard_characters}).

The goal of this paper is to develop a more general theory of square root crystals, mimicking that of abstract (Kashiwara) crystals from~\cite{Kashiwara93}.
We focus on the analog of the crystal $\cB(\infty)$ as it provides additional benefits through the Kashiwara embedding~\cite{Kashiwara93}.
This embedding is fully encoded in the polyhedral realization of Nakashima and Zelevinsky~\cite{Nakashima99,NZ97} as the elements are given by integral points in a rational cone.
It realizes $\cB(\infty)$ as a subcrystal of a tensor product of elementary crystals $\cE_i$ introduced in~\cite[Ex.~1.2.6]{Kashiwara93} that correspond to an infinite $\mathfrak{sl}_2$ string for the $i$-th simple root.
These elementary crystals satisfy the Coxeter braid relations for the corresponding Weyl group (see, e.g.,~\cite[\S12.1]{BumpSchilling}).

Our intent was to realize a similar construction for square root crystals as an alternative way to prove Yu's results and generalize to other types.
However, some subtleties arise in our way of  achieving this.
We will now describe our results in more detail and highlight some of the obstructions we encountered.

\subsection{Main results}

We begin in a very general setting by introducing a notion of $\bfN$-th roots of the crystal operators $e_i$ and $f_i$ (Definition~\ref{abstract-def}).
This depends on some positive integer parameters $\bfN=(N_i)_{i \in I}$ and leads to a definition of a monoidal category of
\defn{abstract $\bfN$-root crystals}.
Notably, the tensor product for this category (Definition~\ref{def:tensor_product}) has no dependence on the choice of parameters $\bfN$. 

Because of an important integrality condition required for tensor products (Proposition~\ref{prop:monoidal_cat}), we show that one can restrict to the case $N_i \in\{ 1,2\}$ with no loss of generality (Proposition~\ref{prop:trivial_large_root}).
Taking all parameters $N_i=1$ recovers the classical definition of abstract Kashiwara crystals, while the square root crystals mentioned in the previous section correspond to taking all $N_i = 2$.
Many general properties and definitions for abstract Kashiwara crystals lift to $\bfN$-root crystals, allowing us to build a fairly well-behaved category with a natural functor to Kashiwara crystals.

After these preliminaries, we focus on the special case of  $\sqgln$-crystals and its subcategory of \defn{polynomial $\sqgln$-crystals}, which were previously studied in~\cite{MT2023,MTY,Yu23} (there called \defn{normal $\sqgln$-crystals}). 
Each polynomial $\sqgln$-crystal is generated from a standard crystal $\SS_n$ by taking iterated tensor products, full subcrystals, and finite direct sums.
This lets us define a monoidal functor to $\gl_n$-crystals that recovers $\cB(\lambda)$ from $\SetTab_n(\lambda)$ (Proposition~\ref{defectF-prop}).

One of our main results is to construct a version of $\cB(\infty)$ in the category of $\sqgln$-crystals. 
We do this first by taking a certain direct limit of weight-shifted versions of the crystals $\SetTab_n(\lambda)$ from \cite{Yu23} (Theorem~\ref{thm:directed_system}).
We denote this object as $\SetTab_n(\infty)$.
The direct limit approach allows us to lift several interesting properties of $\SetTab_n(\lambda)$ to $\SetTab_n(\infty)$, such as a Demazure filtration (Theorem~\ref{dem-thm}).
We also provide $\SetTab_n(\infty)$ with a set-valued analog (Proposition~\ref{mala-lem}) of the marginally large tableaux model for $\cB(\infty)$ due to Cliff~\cite{Cliff98} (see also~\cite{HL08}).
Figure~\ref{fig:gl3mlt} shows the crystal graph of $\SetTab_n(\infty)$ using this model.

A curious feature of $\sqgln$-crystals is that $2$-cycles of the form $b \xrightarrow[\hspace{10pt}]{i} b' \xrightarrow[\hspace{10pt}]{i+1} b$  may occur in the corresponding crystal graphs. 
This motivates us to define a square root version of the infinite $i$-string crystals $\cE_i$ by adding analogous $2$-cycles.
We call this a \defn{looped path crystal} and denote it by $\sqrtE_i$ (Definition~\ref{def:looped_path}).
Two other main results are a square root analog of the polyhedral realization of $\cB(\infty)$ (Theorem~\ref{bzl-thm}) and the Kashiwara embedding (Theorem~\ref{thm:Kashiwara_embedding}).
We also explain how to recover the finite crystal $\SetTab_n(\lambda)$ by left tensoring $\SetTab_n(\infty)$ with a 1-element crystal $\cR_\lambda$ and restricting to one connected component (Corollary~\ref{cor:finite_recovery}).
This is completely analogous to the classical situation described in~\cite[Thm.~5]{Kashiwara91}.

In more detail, Theorem~\ref{bzl-thm} shows that $\SetTab_n(\infty)$ arises from the following alternate construction.
First take a reduced word $\ii=(i_1,i_2,\dots,i_N)$ for $w_0=[n,\dots,2,1] \in S_n$, then form the tensor product $\sqrtE_{i_1}\otimes \sqrtE_{i_2} \otimes \cdots \otimes \sqrtE_{i_N}$, and finally take the subcrystal \defn{lower generated} (see Section~\ref{basic-sect}) by the element $ 0\otimes 0 \otimes \cdots \otimes 0$.
When we make an appropriate choice for $\ii$, the resulting crystal turns out to be isomorphic to the direct limit $\SetTab_n(\infty)$.

The same recipe results in the $\gl_n$-crystal $\cB(\infty)$ if we tensor together the simple path crystals  $\cE_i$ instead of the loop-paths $\sqrtE_i$ \cite[\S12]{BumpSchilling}.
In fact, this classical construction works for any reduced word $\ii$ for $w_0$.
By contrast, we obtain $\SetTab_n(\infty)$ when $\ii$ is in the restricted commutation class of the so-called \defn{Berenstein--Zelevinsky--Littelmann (BZL) word} 
(named for~\cite{BZ01,Littelmann98})
\[
\BZL = (n-1, n-2, n-1,  \dotsc, 2,3,4,\dots,n-1,  1,2,3, \dotsc, n-1),
\]
which is the ``nicest'' possible reduced word for $w_0$. 
Our construction of $\SetTab_n(\infty)$ can fail outside this class, however.

Even there, our model for $\SetTab_n(\infty)$ is not given by integral points in a cone (Example~\ref{ex:not_polyhedral}).
Section~\ref{sec:nonBZL} describes explicitly some obstructions that prevent us from taking $\ii$ to be any reduced word for $w_0$; see  Proposition~\ref{prop:no_kashiwara_here} and Remark~\ref{rem:non_BZL_failure}.
We also observe a number of other differences between $\sqgln$- and $\gl_n$-crystals (e.g.,  Remark~\ref{rem:braid_rels}, Remark~\ref{rem:square_func_elem}, Example~\ref{ex:non_twisted_auto}, and Section~\ref{sec:dualities}).
These are shadows of an asymmetry in how simple roots are split in the definition of a $\sqgln$-crystal.
This asymmetry is what allows the $2$-cycles to appear, but also seems to be necessary to have a Demazure-type filtration on $\SetTab_n(\lambda)$.
This filtration is an essential motivation for Yu's original definition of $\SetTab_n(\lambda)$ in \cite{Yu23}.

On the way to Theorem~\ref{bzl-thm}, we prove an analog of the Lusztig parameterization (Theorem~\ref{BZL-cor}), which comes from the PBW-type bases of $U_q(\g)^-$.
We show this result by splitting the crystal up row-by-row, as opposed to the column-by-column splitting of~\cite{Yu23} (Lemma~\ref{lemma:admissible_reading}), and explicitly describing the crystal structure on the parameterization for each row (Proposition~\ref{vcrys-prop}).
One application of this  is to derive a surprising character formula (Corollary~\ref{cor:infchar}) for 
$\SetTab_n(\infty)$, namely:
\be
\ch\bigl( \SetTab_n(\infty) \bigr) = \prod_{1\leq i < j\leq n} \frac{1+\beta x_j}{1 - x_j x_i^{-1}}
= \ch\bigl(\cB(\infty)\bigr)  \prod_{i=1}^n (1 + \beta x_i)^{i-1}.
\ee
The factorization on the right is not at all obvious from the direct limit construction of $\SetTab_n(\infty)$.

\subsection{Outline}

This paper is organized as follows.
In Section~\ref{sec:abstract}, we give our general definition of a monoidal category of \defn{$\bfN$-root crystals} and describe some of its properties.
In Section~\ref{sec:sqrt_crystals}, we study a subcategory of \defn{polynomial $\sqgln$-crystals} and construct our square root version of $\cB(\infty)$ as a direct limit. 
Section~\ref{sec:parameterizations} gives a Lusztig-type parameterization of this crystal and contains our main results concerning a Kashiwara embedding and polyhedral-type parameterization. 
We also derive some results on the string data associated to our crystal.
Section~\ref{sec:oddities} presents some additional remarks about how $\sqgln$-crystals behave compared to their classical Kashiwara counterparts.

\subsection*{Acknowledgements}

This article is based upon work supported by the National Science Foundation under grant DMS-1929284 while both authors were in residence at the Institute for Computational and Experimental Research in Mathematics in Providence, RI, during the Categorification and Computation in Algebraic Combinatorics semester program in the fall of 2025.

T.S.\ thanks Hong Kong University of Science and Technology for its hospitality during his visit in February, 2026.

E.M.\ was partially supported by Hong Kong RGC grants 16304122 and 16304625. 
T.S.\ was partially supported by JSPS KAKENHI Grant Numbers JP23K12983 and JP26K06735.

\subsection*{Statement on AI Use}

We completed the first draft of this article without AI use.
GPT-6 Astra was used extensively for proofreading and verifying mathematical correctness when preparing the second version. This was particularly significant in helping us find errors in (and subsequent corrections to) our axioms for crystal morphisms in Definition~\ref{morphism-def} and the categorical properties in Proposition~\ref{prop:basic_cat_props}.

\section{Abstract crystals}
\label{sec:abstract}

We begin with a general definition of an abstract $\bfN$-root crystal, which generalizes the notion of an abstract Kashiwara crystal (that is, one not necessarily arising from the crystal basis of a representation).
Our main definitions here extend earlier constructions in \cite{MT2023,MTY,Yu23}.

\subsection{Basic definitions}\label{basic-sect}

Throughout this section, we fix a finite set $I$, and let $\bfN := (N_i \mid i \in I)$ be a sequence of positive integers indexed by $I$.

Let $V$ be a finite-dimensional real vector space with a positive definite symmetric bilinear form $\langle\cdot,\cdot\rangle \colon V\times V \to \RR$.
For any nonzero vector $0\neq \alpha \in V$, define $\alpha^{\vee} := \frac{2}{\langle \alpha,\alpha\rangle} \alpha \in V$.

\begin{definition}
An \defn{$\bfN$-root type} is a tuple $(\Lambda, \Phi, \bal)$ consisting of:
\begin{itemize}

\item a free $\ZZ$-module $\Lambda\subset V$ called the \defn{weight lattice}, whose elements we call \defn{weights};

\item a \defn{root system} $\Phi \subset \Lambda$ with a fixed choice of \defn{simple roots} $\{\alpha_i \mid i \in I\}\subset \Phi$ 
and a fixed choice of \defn{fundamental weights} $\{\Lambda_i \mid i \in I\} \subset \Lambda$ satisfying 
\[
\langle \lambda, \alpha_i^{\vee} \rangle \in \ZZ\quand \langle \Lambda_i, \alpha_j^{\vee} \rangle = \delta_{ij}\quad\text{for all $\lambda \in \Lambda$ and $i,j \in I$};
\]

\item a family of weights $\bal := (\alpha_i^{(m)} \in \Lambda \mid i \in I, \; m \in \ZZ/N_i\ZZ)$ such that if $i \in I$ then
\be \label{eq:kth_weight_cond}
\sum_{m \in \ZZ/N_i\ZZ} \alpha_i^{(m)} = \alpha_i.
\ee
\end{itemize}
\end{definition}

Both the choice of simple roots and the choice of fundamental weights are part of the data of an $\bfN$-root type.
As the root system $\Phi$ is not required to span $V$, the fundamental weights $\Lambda_i$ are not uniquely determined by the simple roots $\alpha_i$.

Now we fix an $\bfN$-root type.
We refer to the set $\{\alpha_i^{\vee} \mid i \in I\}$ as the \defn{simple coroots}.
The \defn{Cartan matrix} of the root system $\Phi$ is the matrix $[\langle \alpha_i, \alpha_j^{\vee} \rangle]_{i,j \in I}$, which has $\langle \alpha_i, \alpha_i^{\vee} \rangle = 2$ and $\langle \alpha_i, \alpha_j^{\vee} \rangle \leq 0$ for all $i \neq j$.
The Cartan matrix determines the type of the root system and is independent of the choices made (up to reordering).
For more on root systems, we refer the reader to~\cite[\S2.1]{BumpSchilling}.

For each $\alpha \in \Phi$, define the reflection $s_{\alpha} \colon V \to V$ by $s_{\alpha}(v) = v - \langle v,\alpha^{\vee} \rangle \alpha$, and notice that $s_{\alpha}^2(v) = v$.
The \defn{Weyl group} of $\Phi$ is $W = \langle s_\alpha \mid \alpha \in \Phi \rangle \subseteq \GL(V).$
This group is also generated by the \defn{simple reflections} $s_i := s_{\alpha_i} $ for $i \in I$.
The larger generating set of reflections $\{  s_\alpha \mid \alpha \in \Phi\}$ consists of all elements conjugate in $W$ to a simple reflection. 

For any $w \in W$, let $\ell(w) \in \NN$ denote the \defn{length} of $w$, which is the minimal number such that $w = s_{i_1} \cdots s_{i_{\ell(w)}}$.
\defn{Bruhat order} on $W$ is the partial order defined as the transitive closure of $w < w'$ if $\ell(w) < \ell(w')$ and there exists a reflection $s_\alpha$ for $\alpha \in \Phi$ such that $s_\alpha w = w'$.
If $W$ is finite, then there exists a unique element $w_0$ of longest length.
For more on Weyl groups, we refer the reader to~\cite{Humphreys90}.

When there is no danger of confusion, we will simply denote an $\bfN$-root type $(\Lambda, \Phi, \bal)$ by $\bal$.

\begin{definition}[Abstract crystal] 
\label{abstract-def}
Fix an $\bfN$-root type $\bal$, and suppose $\cB$ is a set with maps
\[
\weight \colon  \cB\to \Lambda,
\quad
e_i,f_i \colon  \cB \sqcup\{\zero\}\to \cB \sqcup \{\zero\},
\quand
\varepsilon_i, \varphi_i \colon \cB \to \tfrac{1}{N_i}\ZZ \sqcup \{-\infty\}
\]
for each $i \in I$, where $\zero \notin \cB$ is an auxiliary element.
We say that $\cB$ is an \defn{abstract $\bfN$-root crystal} of type $(\Lambda, \Phi, \bal)$ if the following axioms hold for all $i \in I$ and $b,c \in \cB$:
\begin{enumerate}
\item[(A0)] $e_i\zero = f_i\zero = \zero$;
\item[(A1)] $e_i b = c$ if and only if $b = f_i c$;
\item[(A2)] $\varphi_i(b) = \varepsilon_i(b) + \langle \weight(b), \alpha_i^{\vee} \rangle$
under the convention that $-\infty + m = -\infty$ for any $m \in \QQ$;
   \label{axiom:phi_ep_wt}
\item[(A3)] if $\varepsilon_i(b) = -\infty$ then $e_i b = f_i b = \zero$;
\item[(A4)] if $e_i b \neq \zero$ then $\varepsilon_i(e_i b) = \varepsilon_i(b) - \frac{1}{N_i}$ and $\varphi_i(e_i b) = \varphi_i(b) + \frac{1}{N_i}$;
\item[(A5)] if $e_i b \neq \zero$ then $\weight(e_i b) = \weight(b) + \alpha_i^{(m)}$ for $m = N_i \varepsilon_i(b) \in \ZZ/N_i\ZZ$.
\end{enumerate}
\end{definition}

\begin{remark}\label{abstract-def-rmk}
Since  the form $\langle\cdot,\cdot \rangle$  takes values in $\ZZ$ when pairing lattice weights with simple coroots, axiom (A2) implies that either 
\[
\text{$\varepsilon_i(b)=\varphi_i(b)=-\infty$}
\qquad\text{or}\qquad
 \text{$\varepsilon_i(b)$ and $\varphi_i(b)$ are both finite with $\varphi_i(b) - \varepsilon_i(b) \in \ZZ$.}
 \]
Additionally, axioms  (A1), (A4), and (A5) imply that if
 $e_i^{N_i} b \neq \zero$ 
then  $\weight(e_i^{N_i} b) = \weight(b) + \alpha_i$.
\end{remark}

We have given an ``asymmetric'' formulation of the axioms in Definition~\ref{abstract-def}, but we can also deduce the following dual properties.

\begin{proposition}\label{prime-prop}
Suppose $\cB$ is an abstract $\bfN$-root crystal.
Let $i \in I$ and $c \in \cB$.
 Then 
\begin{enumerate}
\item[(A3$'$)] if $\varphi_i(c) = -\infty$ then $e_i c = f_i c = \zero$;
\item[(A4$'$)] if $f_i c \neq \zero$ then $\varepsilon_i(f_i c) = \varepsilon_i(c) + \frac{1}{N_i}$ and $\varphi_i(f_i c) = \varphi_i(c) - \frac{1}{N_i}$;
\item[(A5$'$)] if $f_i c \neq \zero$ then $\weight(f_i c) = \weight(c) - \alpha_i^{(m)}$ for $m = N_i \varphi_i(c) +1 \in \ZZ/N_i\ZZ$.
\end{enumerate}
\end{proposition}

\begin{proof}
Property (A3$'$) follows from (A3) since we observed in Remark~\ref{abstract-def-rmk} that if $\varphi_i(c) = -\infty$ then $\varepsilon_i(c)=-\infty$.
Now suppose $ f_i c \neq \zero$, and let $b = f_ic \in\cB$. 
Then (A1) states $\zero \neq e_i b = c \in \cB$, and so $\varepsilon_i(b)$, $\varepsilon_i(c)$, $\varphi_i(b)$, and $\varphi_i(c)$ are all finite
and (A4$'$) is equivalent to (A4).

To deduce (A5$'$), again let $b = f_i c \in \cB$, and set $m = \varepsilon_i(b)N_i$.
By (A5), we have
\[
\weight(e_i b) = \weight(b) + \alpha_i^{(m)} = \weight(f_i c) + \alpha_i^{(m)} = \weight(c),
\]
where the last equality is from $ e_i b = c$, and so $\weight(f_i c) = \weight(c) - \alpha_i^{(m)}$.
This is equivalent to (A5$'$) since  (A4$'$) implies that
\[
\varepsilon_i(b)=\varepsilon_i(f_ic) = \varepsilon_i(c) + \tfrac{1}{N_i} = \varphi_i(c) + \tfrac{1}{N_i} - \underbrace{(\varphi_i(c) - \varepsilon_i(c))}_{\in\ZZ\text{ by Remark~\ref{abstract-def-rmk}}}
\]
so $N_i\varphi_i(c) + 1 \equiv N_i\varepsilon_i(b) \modu N_i)$, which is $m$ as claimed.
\end{proof}

For convenience, we sometimes refer to an abstract $\bfN$-root crystal as simply an abstract crystal.
We call $e_i$ and $f_i$ the (\defn{raising} and \defn{lowering}, respectively) \defn{crystal operators} of $\cB$.
An element $b$ of an abstract crystal $\cB$ is \defn{highest weight} (respectively, \defn{lowest weight}) if $e_ib=\zero$ (respectively, $f_ib=\zero$) for all $i \in I$.
The \defn{crystal graph} of $\cB$ is the directed graph with vertex set $\cB$ that has labeled edges $b \xrightarrow{i} c$ for each $b,c \in \cB$ and $i \in I$ such that $f_i b = c$.

An abstract crystal $\cB$ is \defn{connected} if there is an element $b \in \cB$ such that for  each $b' \in \cB$, there exists a sequence $(g_j \in \{e_i, f_i \mid i \in I\})_{j=1}^{\ell}$ such that $b' = g_{\ell} \cdots g_2g_1 b$.
Equivalently, $\cB$ is connected if its crystal graph is weakly connected as a directed graph.
An \defn{$i$-string} of $\cB$ is a connected subcrystal $\cS$ such that for any $b, c \in \cS$,
there exists an integer $k\in\NN$ such that either $e_i^k b = c$ or $f_i^k b = c$. 
Notice that it is sufficient to check this  condition for a single fixed element $b \in \cS$.
We say that $\cB$ is \defn{generated} by a set $X \subseteq \cB$
if every weakly connected component of the crystal graph of $\cB$ contains at least one element of $X$.
We say that $\cB$ is \defn{lower generated} by $X \subseteq \cB$ if for any $b \in \cB$, there exists a sequence $(i_j \in I)_{j=1}^{\ell}$ and an element $x \in X$ such that $b = f_{i_{\ell}} \cdots f_{i_2} f_{i_1} x$. 

Let $\ZZ\llbracket\Lambda\rrbracket$ be the $\ZZ$-module of infinite $\ZZ$-linear combinations of the symbols $ x^{\lambda} $ for $ \lambda \in \Lambda $.
The \defn{character} of an abstract crystal $\cB$ whose weight map has $\abs{\weight^{-1}(\lambda)} < \infty$ for all $\lambda \in \Lambda$ is 
\[
\ch(\cB) = \sum_{b \in \cB} x^{\weight(b)} \in \ZZ\llbracket\Lambda\rrbracket.
\]
The module $\ZZ\llbracket\Lambda\rrbracket$ 
is too large to have a well-defined ring structure,
but some of its submodules do form rings.
Let $\ZZ\llbracket\Lambda\rrbracket_<$ denote the $\ZZ$-submodule consisting of finite $\ZZ$-linear combinations of vectors of the form $x^{\lambda}  f$  for some $\lambda \in \Lambda$ and $f\in\ZZ\llbracket Q^- \rrbracket$, where
$
Q^- := \bigoplus_{i \in I} -\NN \alpha_i
$
is the negative root cone.
This is a ring with multiplication extending the natural product $x^{\lambda} x^{\mu} = x^{\lambda + \mu}$.
Let $\ZZ[\Lambda] \subseteq \ZZ\llbracket\Lambda\rrbracket_<$ be the subring of finite linear combinations of $x^{\lambda}$ for $\lambda \in \Lambda$.

\begin{remark}
\label{rem:special_roots}
We write $\bfN=1$ to indicate that $N_i=1$ for all $i \in I$.
When this property holds, our definition of an abstract $\bfN$-root crystal reduces to the classical notion of an (abstract) \defn{Kashiwara crystal}~\cite{Kashiwara93}. 
This follows as if $N_i=1$ then we must have $\alpha_i^{(m)} = \alpha_i$ for all $m \in \ZZ/N_i\ZZ$.

Kashiwara crystals are  a combinatorial abstraction for the Kashiwara crystal bases of representations of the (Drinfel'd--Jimbo) quantum group $U_q(\g)$ associated to a (symmetrizable) Kac--Moody Lie algebra $\g$ (with root system $\Phi$)~\cite{Kashiwara90,Kashiwara91}.
We do not currently have such an interpretation of abstract $\bfN$-root crystals in general.
\end{remark}

\begin{definition}
An abstract $\bfN$-root crystal $\cB$ is \defn{upper regular}
 if for each $i \in I$ and $b \in \cB$ one has
\be\label{upper-reg-eq}
 \varepsilon_i(b) =\tfrac{1}{N_i} \max\{ k \in \NN \mid e_i^k b \neq \zero\} \in \tfrac{1}{N_i}\NN.
\ee
Similarly, an abstract $\bfN$-root crystal 
$\cB$ is \defn{lower regular}
 if for each $i \in I$ and $b \in \cB$ one has
\be\label{lower-reg-eq}
 \varphi_i(b) = \tfrac{1}{N_i} \max\{ k \in \NN \mid f_i^k b \neq \zero\}\in \tfrac{1}{N_i}\NN.
\ee
We call $\cB$ \defn{regular} if it is both upper regular and lower regular.
\end{definition}

We have an alternative way to characterize when a crystal is upper or lower regular.

\begin{proposition}\label{regular-prop}
Suppose $\cB$ is an abstract $\bfN$-root crystal.
\ben
\item[(a)] The crystal $\cB$ is upper regular if and only if the following holds:
\begin{equation}
\label{eq:URalt}
\tag{U}
\text{$\varepsilon_i(b) \geq 0$ for all $i \in I$ and $b \in \cB$, with equality if and only if $e_ib=\zero$.}
\end{equation}
\item[(b)] The crystal $\cB$ is lower regular if and only if the following holds:
\begin{equation}
\tag{L}
\text{$\varphi_i(b) \geq 0$
for all $i \in I$ and $b \in \cB$,
 with equality if and only if $f_ib=\zero$.}
\end{equation}
\een
\end{proposition}

\begin{proof}
We will only show (a) as the proof of (b) is similar.
Clearly if $\cB$ is upper regular then~\eqref{eq:URalt} holds.
Now assume that \eqref{eq:URalt} holds.
Choose any $i \in I$.
If $b \in \cB$ and $k \in \NN$ are such that $e_i^kb\neq \zero$, then axiom (A4) in Definition~\ref{abstract-def} implies that $\varepsilon_i(e_i^kb) = \varepsilon_i(b) - \frac{k}{N_i}$.
As \eqref{eq:URalt} asserts that $\varepsilon_i$ takes only nonnegative values on $\cB$, we conclude that   $e_i^kb= \zero$ for all integers $k > \varepsilon_i(b) N_i$.
Hence, the formula $\widetilde\varepsilon_i(b) := \tfrac{1}{N_i} \max\{ k \in \NN \mid e_i^k b \neq \zero\}$ gives a well-defined element of $\frac{1}{N_i}\NN$ for all $b\in \cB$.
To deduce that $\cB$ is upper regular, we must check that $\varepsilon_i(b)=\widetilde\varepsilon_i(b)$ for all $b\in \cB$.

We prove this by induction on $\widetilde\varepsilon_i(b)$.
If $\widetilde\varepsilon_i(b)=0$ then $e_i b = \zero$ so by assumption $\varepsilon_i(b) = 0 = \widetilde\varepsilon_i(b)$.
If $\widetilde\varepsilon_i(b)>0$ then clearly $e_ib \neq \zero$ and  $\widetilde\varepsilon_i(b) = \tfrac{1}{N_i} +\widetilde\varepsilon_i(e_ib)$, so $\varepsilon_i(b) = \tfrac{1}{N_i} +\varepsilon_i(e_ib)$ by axiom (A4) in Definition~\ref{abstract-def}.
In this case we have $\varepsilon_i(e_ib)=\widetilde\varepsilon_i(e_ib)$ by induction, so $\varepsilon_i(b)=\widetilde\varepsilon_i(b)$ again follows.
\end{proof}

Below are some examples of abstract $\bfN$-root crystals 
  for an arbitrary   $\bfN$-root type.

\begin{example}
\label{ex:Tla_crystal}
For each weight $\lambda \in \Lambda$, there is a one-element abstract $\bfN$-root crystal
$
\cT_{\lambda} := \{t_{\lambda}\}
$
with crystal structure  defined  by
$
e_i t_{\lambda} = f_i t_{\lambda} = \zero
$,
$
\varepsilon_i(t_{\lambda}) = \varphi_i(t_{\lambda}) = -\infty
$,
and
$\weight(t_{\lambda}) = \lambda.
$
Note that $\cT_{\lambda}$ is neither upper regular nor lower regular when $I$ is nonempty.
\end{example}

\begin{example}
\label{ex:Rla_crystal}
For each $\lambda \in \Lambda$, we have another one-element abstract $\bfN$-root crystal
$
\cR_{\lambda} := \{r_{\lambda}\}
$
with crystal structure given by
$e_i r_{\lambda} = f_i r_{\lambda} = \zero$,
$\varepsilon_i(r_{\lambda}) = - \langle \lambda, \alpha_i^{\vee} \rangle$,
$\varphi_i(r_{\lambda}) = 0$,
and
$\weight(r_{\lambda}) = \lambda$.
Note that $\cR_{\lambda}$ is always lower regular but only upper regular when $ \langle \lambda, \alpha_i^{\vee} \rangle=0$ for all $i \in I$.
\end{example}

Our definition of $\bfN$-root types allows the parameters  $N_i$ for $i \in I$ to have any positive integer values. However, the following proposition shows that in an abstract $\bfN$-root crystal, taking $N_i>2$ forces the crystal operators $e_i$ and $f_i$ to act trivially.
This degeneracy is a consequence of the integrality of the form $\langle\cdot, \cdot\rangle$, which is an important structural property of abstract crystals. 

\begin{proposition}
\label{prop:trivial_large_root}
Suppose $\cB$ is an abstract $\bfN$-root crystal. 
For any $i \in I$ such that $N_i > 2$, we have $e_i b = f_i b = \zero$ for all $b \in \cB$.
\end{proposition}

\begin{proof}
Suppose   $i \in I$  and  $b \in \cB$ have $e_i b \neq \zero$. Then   (A4) in Definition~\ref{abstract-def} implies that
$
 \varphi_i(e_ib)-\varphi_i(b)  = \varepsilon_i(b) - \varepsilon_i(e_ib) = \tfrac{1}{N_i}
$
so by (A2) in  Definition~\ref{abstract-def}  we have 
\[
\tfrac{2}{N_i} =  \varphi_i(e_ib)-\varphi_i(b)  + \varepsilon_i(b) - \varepsilon_i(e_ib) = \langle \weight(e_ib)- \weight(b),\alpha_i^\vee\rangle \in \ZZ.
\]
As $N_i\in \PP$ we must have $N_i \in \{1,2\}$.
It follows similarly that if $f_i b \neq \zero$ then $N_i \in \{1,2\}$.
\end{proof}

In view of this proposition, we shall usually assume that $N_i \in \{1,2\}$ for all $i \in I$.
This essentially loses no generality, 
since if we ignore  the irrelevant values of $\epsilon_i$ and $\varphi_i$ when $N_i>2$,
then an abstract $\bfN$-root crystal has the same data 
as one with all $N_i \in \{1,2\}$ but with the root system $\Phi$ replaced by the parabolic subsystem whose set of simple roots is $\{\alpha_i \mid i \in I \text{ with }N_i \leq 2\}$.
 
\begin{example}
\label{ex:elem_crystal}
Fix some $i \in I$ such that $N_i \in \{1,2\}$, and assume that $\langle \alpha_i^{(0)}, \alpha_i^{\vee}\rangle=\langle \alpha_i^{(1)}, \alpha_i^{\vee}\rangle = 1$ when $N_i = 2$.
Then we may define the \defn{$i$-th elementary crystal} $\cE_i$ as the abstract $\bfN$-root crystal given by the set of symbols $\left\{ \elem{i}{m} \mid m \in \frac{1}{N_i}\ZZ\right\}$ with weight function
\[
\weight(\elem{i}{m}) = \lfloor m \rfloor \alpha_i + 2(m - \lfloor m \rfloor) \alpha_i^{(0)} \in \Lambda,
\]
and with crystal operators and statistics given by
\begin{subequations} \label{eq:elem_crystal_def}
\begin{align}
\varepsilon_j(\elem{i}{m}) &= \begin{cases} -m & \text{if } i=j, \\ -\infty & \text{otherwise}, \end{cases}
\qquad&
e_j \elem{i}{m} &= \begin{cases} \elem{i}{m+\frac{1}{N_i}} & \text{if } i = j, \\ \zero & \text{otherwise}, \end{cases}
\\
\varphi_j(\elem{i}{m}) &= \begin{cases} m & \text{if } i=j, \\ -\infty & \text{otherwise}, \end{cases}
\qquad&
f_j \elem{i}{m} &= \begin{cases} \elem{i}{m-\frac{1}{N_i}} & \text{if } i = j, \\ \zero & \text{otherwise}. \end{cases}
\end{align}
\end{subequations}
\end{example}

\subsection{Categories of crystals}\label{cat-sect}

We continue to fix an arbitrary $\bfN$-root type $(\Lambda, \Phi, \bal)$, which we abbreviate as simply $\bal$.
We now define a category $\sC_{\bal}$ whose objects consist of all abstract $\bfN$-root crystals of type $\bal$.
The morphisms in this category are defined below.

Unless otherwise stated, all abstract $\bfN$-root crystals and morphisms considered in this section will be in the category $\sC_{\bal}$.
For simplicity, when $N_i = 1$ for all $i \in I$ we denote this category by $\sC_1$.

\begin{definition}[Crystal morphism]
\label{morphism-def}
A \defn{morphism} of abstract crystals $\cB, \cC \in \sC_{\bal}$ is a map of sets \[\Psi \colon \cB \sqcup\{\zero\}\to \cC\sqcup\{\zero\}\] satisfying $\Psi(\zero) = \zero$ 
and the following conditions for all $b \in \cB$ with $\Psi(b)\neq\zero$:
\ben
\item[(M1)] $\weight\bigl(\Psi(b)\bigr) = \weight(b)$,
along with $\varepsilon_i\bigl(\Psi(b)\bigr)= \varepsilon_i(b) $ and $\varphi_i\bigl(\Psi(b)\bigr)=\varphi_i(b) $ for all $i \in I$;

\item[(M2)] if $i \in I$ is such that $\Psi(e_ib)\neq \zero$ then $e_i \Psi(b) = \Psi(e_i b)$;

\item[(M3)] if $i \in I$ is such that $\Psi(f_ib)\neq \zero$ then $f_i \Psi(b) = \Psi(f_i b)$;

\item[(M4)] if $i \in I$ is such that $N_i>1$ and $\Psi(e_i^2b) \neq \zero$ then $\Psi(e_ib)\neq \zero$;

\item[(M5)] if $i \in I$ is such that  $N_i>1$ and $\Psi(f_i^2b) \neq \zero$ then $\Psi(f_ib)\neq \zero$.

\een
A morphism is \defn{strict} if $e_i \circ \Psi = \Psi \circ e_i$ and $f_i\circ \Psi = \Psi \circ f_i$ for all $i \in I$.
A morphism is an \defn{embedding} if it is injective as a map of sets.
A morphism is an \defn{isomorphism} if it is strict and bijective as a map of sets, in which case the inverse map is also a morphism.
\end{definition}

\begin{remark}
\label{rem:extra_morphism_cond}
If we omit (M4) and (M5), these are the usual axioms for morphisms of Kashiwara crystals (cf. \cite[\S2.3]{BumpSchilling}).
The extra requirements imposed by (M4) and (M5) ensure the existence of a ``squaring'' functor $\sC_{\bal} \to \sC_1$, which we discuss below after Proposition~\ref{prop:basic_cat_props}.
By Proposition~\ref{prop:trivial_large_root}, conditions (M2)--(M5) are equivalent to requiring for all $i \in I$ and $b\in \cB$ that:
\ben
\item[(M2$'$)] if $1\leq j \leq k \leq N_i$ and $\Psi(b) \neq \zero\neq \Psi(e_i^kb)$ then $e_i^j \Psi(b)  =\Psi(e_i^j b)$; 
\item[(M3$'$)]  if $1\leq j \leq k \leq N_i$ and $\Psi(b) \neq \zero\neq \Psi(f_i^kb)$ then $f_i^j \Psi(b) =\Psi(f_i^j b)$.
\een
\end{remark}

Suppose $\cB$ and $\cC$ are abstract $\bfN$-root crystals.
The direct sum $\cB \oplus \cC$ is the   crystal formed from the disjoint union of the functions $e_i$, $f_i$, $\varepsilon_i$, $\varphi_i$, and $\weight$ (for all $i \in I$).
The crystal graph of $\cB \oplus \cC$ is the disjoint union of the crystal graphs of $\cB$ and $\cC$.
One has $\ch(\cB \oplus \cC) = \ch(\cB) + \ch(\cC)$ when $\ch(\cB)$ and $ \ch(\cC)$ are both defined.

The notion of an embedding allows us to define subobjects in $\sC_{\bal}$ in the usual way, which we call \defn{subcrystals}. More precisely,
we say $\cB$ is a subcrystal of $\cC$ if $\cB$ is a subset of $\cC$ and the natural set inclusion map is a crystal embedding.
To indicate that $\cB$ is a subcrystal of $\cC$, we write $\cB \subseteq \cC$.
If $\cB \subseteq \cC$ then we define the quotient crystal $\cC/\cB$
by taking the set difference $\cC \setminus \cB$, retaining the structure maps  
of $\cC$,  and truncating any crystal operators leaving the subset.

A \defn{full subcrystal} of an abstract crystal is a subcrystal such that the natural embedding is strict.
This is equivalent to being a disjoint union of connected components of the crystal graph.
A strict embedding $\Psi \colon \cB \to \cC$ must send any connected component (i.e., a connected full subcrystal) of $\cB$ to an isomorphic connected component of $\cC$.
Thus, if $\Psi$ is a strict embedding then $\Psi$ is an isomorphism from $\cB$ to some union of connected components of $\cC$.

Finally, a subset $\cS\subseteq \cB$ is \defn{midpoint-closed} if,
whenever  $i \in I$ is such that $N_i>1$ and $b\in\cS$ is such that $f_i^2b \in \cS$, one also has  $ f_i b \in \cS$.
Conditions (M4) and (M5) in Definition~\ref{morphism-def} imply that if $\Psi \colon \cB\to\cC$ is any morphism then its support $\supp(\Psi) = \{b \in \cB \mid \Psi(b) \neq \zero\}$ is midpoint-closed.

\begin{proposition}
\label{prop:basic_cat_props}
The category of abstract $\bfN$-root crystals $\sC_{\bal}$ has the following properties:
\begin{itemize}
\item[(a)] There exists a zero morphism $\zero_{\cB\cC} \in \Hom(\cB,\cC)$ for all $\cB, \cC \in \sC_{\bal}$.
\item[(b)] The empty crystal $\emptyset$ is the zero object in $\sC_{\bal}$.
\item[(c)] Direct sum of crystals gives coproducts in $\sC_{\bal}$.
\item[(d)] Any morphism $\Psi \colon \cB\to \cC$ has a kernel in $\sC_{\bal}$.
\item[(e)] A morphism $\Psi \colon \cB\to \cC$ has a cokernel in $\sC_{\bal}$
precisely when $\cC \setminus \mathrm{Image}(\Psi)$ is midpoint-closed.
\item[(f)] If $\cB\subseteq \cC$ then the projection $\cC\to \cC/\cB$ is a morphism precisely when $\cC\setminus\cB$ is midpoint-closed, and in this case it provides the  cokernel of the inclusion $\cB\hookrightarrow \cC$.
\end{itemize}
\end{proposition}

\begin{proof}
Let $\cB, \cC \in \sC_{\bal}$.
The zero morphism $\zero_{\cB\cC} \colon \cB \to \cC$ is given by $\zero_{\cB\cC}(b) = \zero$ for all $b \in \cB$.
It is easy to see that this satisfies the necessary commutative diagram.
There is a unique morphism from $\emptyset$ to any crystal and a unique morphism from any crystal to $\emptyset$, so $\emptyset$ is a zero object.
Next, for $\cB_1, \cB_2 \in \sC_{\bal}$, consider the inclusion maps $\Psi_i \colon \cB_i \to \cB_1 \oplus \cB_2$.
This gives a categorical coproduct since disjoint unions give categorical coproducts in the category of sets.

To show the existence of kernels, consider a morphism $\Psi \colon \cB \to \cC$.
Define the induced subcrystal $\cK := \{b \in \cB \mid \Psi(b) = \zero\} \subseteq \cB$, where the abstract $\bfN$-root crystal structure is given by $f_i b = b'$ if and only if $b,b' \in \cK$ and $b\xrightarrow{i}b'$ is an edge in the crystal graph of $\cB$, and similarly for $e_i$ for all $i \in I$.
Let $\iota \colon \cK \to \cB$ be the natural inclusion morphism, and clearly $\Psi \circ \iota = \zero_{\cK\cC}$.
Now suppose $\Phi \colon \cA \to \cB$ is a morphism with $\Psi\circ \Phi = \zero_{\cA\cC}$.
Then $\Phi(a) \in \cK\sqcup \{\zero\}$ for all $a \in \cA$, so restricting the codomain gives a map of sets $\widetilde{\Phi} \colon \cA\sqcup \{\zero\} \to \cK\sqcup\{\zero\}$ such that $\Phi = \iota \circ \widetilde{\Phi}$.
Checking that $\widetilde{\Phi}$ is a morphism 
is straightforward. 
%
Therefore, $\iota \colon \cK \to \cB$ is a categorical kernel for $\Psi$.

Continue to let $\Psi \colon \cB \to \cC$ be a morphism. 
Suppose $\chi \colon \cC \to \cQ$ is a cokernel for $\Psi$.
Fix some $c\in \cC$, and define a morphism $\Upsilon \colon \cC\to \cC$ by $\Upsilon(c) = c$ and $\Upsilon(c') = \zero$ for all $c' \in \cC \setminus \{c\}$.
If $c$ is not in the image of $\Psi$, then $\Upsilon\circ \Psi = \zero_{\cB\cC}$.
Thus $\chi$ being a cokernel means there exists a unique morphism $\tilde \Upsilon \colon \cQ\to \cC$ such that $\Upsilon = \widetilde\Upsilon \circ \chi$.
This implies $\chi(c) \neq \zero$ as $\widetilde\Upsilon(\chi(c)) = \Upsilon(c) = c \neq \zero$.
On the other hand, we must have $\chi(c') = \zero$ for all $c'$ in the image of $\Psi$ to satisfy $\chi \circ \Psi = \zero_{\cB\cQ}$.
Thus, the complement of the image of $\Psi$ in $\cC$ is the support of the morphism $\chi$, so this set is midpoint-closed.

Conversely, if the set $\cQ = \cC \setminus \{\Psi(b) \mid b\in \cB\}$ is midpoint-closed, then we can view it as an induced subcrystal of $\cC$, and the map $\chi \colon \cC\sqcup\{\zero\} \to \cQ\sqcup\{\zero\}$ given by $\chi(c) = c$ for $c \in \cQ$ and $\chi(c') = \chi(\zero) = \zero$ for $c' \in \cC\setminus \cQ$ is a categorical cokernel for $\Psi$.

Finally, assume $\cB\subseteq \cC$. 
Then the map $\cC \to \cC/\cB$ sending $b \mapsto \zero$ for $b \in \cB$ and $c \mapsto c$ for $c \in \cC\setminus \cB$ evidently defines a morphism if the set $\cC\setminus \cB$ is midpoint-closed, and provides the cokernel of $\cB \hookrightarrow\cC$ in view of the previous paragraph.
Conversely, this map can only give a morphism if the set $\cC\setminus \cB$ is midpoint-closed as it gives the morphism's support.
\end{proof}

We briefly discuss some additional features of the category $\sC_{\bal}$.
First, the direct sum does not give a categorical product, which can be seen by considering $\cT_0 \oplus \cT_0$.
We have
\[
\abs{\Hom(\cT_0, \cT_0 \oplus \cT_0)} = 3,
\]
which is not a perfect square, implying that $\Hom(\cT_0, \cT_0 \oplus \cT_0) \not\iso \Hom(\cT_0, \cT_0) \times \Hom(\cT_0, \cT_0)$, where such an isomorphism is a consequence of a categorical product.

Next, we note that not every epimorphism is a cokernel of some morphism (i.e., categorically normal).
Consider the map $\Psi \colon \cB \oplus \cB \to \cB$ such that $\Psi(b) = b$ for all $b \in \cB$ (coming from either summand) when $\cB$ is nonempty.
This is clearly an epimorphism.
However, for it to be the cokernel of a morphism, the latter morphism would need to be the zero morphism from its domain to $\cB\oplus\cB$, and then its cokernel would be the identity map $\cB \oplus \cB \to \cB \oplus \cB$.

We now define a ``squaring'' functor $\powF{\bfN} \colon \sC_{\bal} \to \sC_1$.
For each object $\cB \in \sC_{\bal}$ we define an abstract Kashiwara crystal $\powF{\bfN}(\cB)$ with the same set of elements and weight function as $\cB$, but with the crystal operators $\overline{e}_i, \overline{f}_i$ and statistics $\overline{\varepsilon}_i, \overline{\varphi}_i$ given by
\be\label{powF-eq}
\overline{e}_i := e_i^{N_i},
\qquad
\overline{f}_i := f_i^{N_i},
\qquad
\overline{\varepsilon}_i(b) := \lfloor \varepsilon_i(b) \rfloor,
\qquand
\overline{\varphi}_i(b) := \lfloor \varphi_i(b) \rfloor,
\ee
where $\lfloor -\infty\rfloor := -\infty$.
The functor is trivial on morphisms $\Psi$; that is,   $\powF{\bfN}(\Psi)$ is the  same function on the underlying sets,
which gives a morphism in $\sC_1$ by the discussion in Remark~\ref{rem:extra_morphism_cond}.

The Weyl group $W \subseteq \GL(V)$ acts on $ \ZZ\llbracket \Lambda\rrbracket $ by the formula $w \cdot x^{\lambda} = x^{w\lambda}$, extended by linearity.
It is well-known~\cite[Props.~2.36 and 2.37]{BumpSchilling} that if a regular Kashiwara crystal has a well-defined character, then this character is invariant under the action of the Weyl group.
This result extends to abstract $\bfN$-root crystals:

\begin{proposition}\label{sym-prop}
Suppose $\cB$ is an abstract $\bfN$-root crystal of root type $\bal$.
If $\cB$ is upper (respectively, lower) regular, 
then $\powF{\bfN}(\cB)$ is an upper (respectively, lower) regular Kashiwara crystal.
Consequently, if $\cB$ is regular with
 a well-defined character,
then $\ch(\cB)$ is $W$-invariant. 
\end{proposition}

\begin{proof}
The first claim is immediate from the characterization of regularity in Proposition~\ref{regular-prop}.
The observation about characters then follows as $\ch(\cB) = \ch(\powF{\bfN}(\cB))$.
\end{proof}

For upper/lower regular crystals, we do not need to check all of the conditions for a map of sets to be a crystal morphism.

\begin{proposition}
\label{prop:morphism_by_ops}
Suppose $\cB, \cC \in \sC_{\bal}$ are both upper regular or both lower regular.
Let $\Psi \colon \cB \sqcup \{\zero\} \to \cC \sqcup \{\zero\}$ be a map of sets such that $\Psi(\zero) = \zero$ and that commutes with all of the crystal operators, such that each connected component of $\cB$ contains at least one element $b$  with $\weight(b) = \weight\bigl(\Psi(b)\bigr)$.
Then $\Psi$ is a crystal morphism.
\end{proposition}

\begin{proof}
Assume $\cB, \cC \in \sC_{\bal}$ are both upper regular.
We have $\varepsilon_i(b) = \varepsilon_i\bigl(\Psi(b)\bigr)$ for all $b \in \cB$ by  the definition of upper regular as $\Psi$ commutes with the crystal operators.
Since each connected component has one element where the weights agree, 
it follows 
from axioms (A5) and (A5$'$) 
 that $\weight(b) = \weight\bigl(\Psi(b)\bigr)$ for all $b \in \cB$.
Finally, we have $\varphi_i(b) = \varphi_i\bigl(\Psi(b)\bigr)$ for all $b \in \cB$ by axiom (A2) in Definition~\ref{abstract-def}, so $\Psi$ is a crystal morphism.
The proof in the lower regular case is similar.
\end{proof}

We will also use the following technical variation of Proposition~\ref{prop:morphism_by_ops}.

\begin{lemma}
\label{lemma:lower_embed_check}
Suppose $\cB, \cC \in \sC_{\bal}$ are both upper regular and $\cB$ is lower generated by $X$.
Let $\Psi \colon \cB \sqcup \{\zero\} \to \cC \sqcup \{\zero\}$ be an embedding of sets such that $\Psi(\zero) = \zero$ and that commutes with all of the $e_i$ crystal operators and $\weight(x) = \weight\bigl(\Psi(x)\bigr)$ for all $x \in X$.
Then $\Psi$ is a crystal embedding.
\end{lemma}

\begin{proof}
Since $\Psi$ commutes with the $e_i$ crystal operators and is an embedding, it commutes with the $f_i$ crystal operators for all $ b \in \cB$ such that $f_i b \neq \zero$ by axiom (A1).
Therefore, $\Psi$ is a crystal morphism by an argument similar to the proof of Proposition~\ref{prop:morphism_by_ops}.
\end{proof}

\subsection{Tensor products}

We continue to abbreviate an arbitrary $\bfN$-root type $(\Lambda, \Phi, \bal)$ as $\bal$.
In this section, we define a tensor product that makes $\sC_{\bal}$ into a monoidal category.
An appealing feature of this tensor product is that its definition is essentially identical to the classical case~\cite[\S1.3]{Kashiwara93}.
Our convention for ordering tensor factors follows~\cite[\S2.3]{BumpSchilling}, which is opposite to the convention in \cite{Kashiwara93}.

\begin{definition}[Tensor product]
\label{def:tensor_product}
For $\cB, \cC \in \sC_{\bal}$, define $\cB \otimes \cC$ to be the set $\cB \times \cC$, where we write the element $(b,c)$ as the formal tensor $b\otimes c$, with the following crystal structure.
For each $i \in I$, define
\[
e_i,f_i \colon (\cB \otimes \cC) \sqcup \{\zero\} \to (\cB \otimes \cC) \sqcup \{\zero\}
\]
by setting $e_i \zero = f_i \zero = \zero$ along with 
\begin{subequations}
\label{ef-tensor-eq}
\begin{align}
e_i(b\otimes c) &= \begin{cases}
b \otimes (e_ic) &\text{if }\varepsilon_i(b) \leq \varphi_i(c), \\
(e_ib) \otimes c &\text{if }\varepsilon_i(b) > \varphi_i(c),
\end{cases}
\\
f_i(b\otimes c) &= \begin{cases}
b \otimes (f_ic) &\text{if }\varepsilon_i(b) < \varphi_i(c), \\
(f_ib) \otimes c &\text{if }\varepsilon_i(b) \geq \varphi_i(c),
\end{cases}
\end{align}
for each $b \otimes c \in \cB \otimes \cC$, 
with the convention that $b\otimes \zero= \zero\otimes c = \zero$.
Then set
\begin{align}
\varepsilon_i(b\otimes c) &= \max\Bigl\{\varepsilon_i(c), \varepsilon_i(b) - \langle \weight(c), \alpha_i^{\vee} \rangle \Bigr\}, \label{eq:ep_tensor}
\\
\varphi_i(b\otimes c) &= \max\Bigl\{\varphi_i(b), \varphi_i(c) + \langle \weight(b), \alpha_i^{\vee} \rangle \Bigr\},
\\
\weight(b\otimes c) & = \weight(b) + \weight(c).
\end{align}
\end{subequations}
\end{definition}

\begin{remark} 
Axiom (A2) in Definition~\ref{abstract-def} implies that $\varepsilon_i(c)$ and $ \varepsilon_i(b)$ are both finite precisely when $\varphi_i(c)$ and $ \varphi_i(b)$ are both finite.
When this occurs, we have the simpler formulas
\begin{subequations}
\label{eq:simpler_ephi}
\begin{align}
\varepsilon_i(b\otimes c) &=
 \varepsilon_i(c) + \max\Bigl\{0, \varepsilon_i(b) -\varphi_i(c)
\Bigr\},\\
 \varphi_i(b\otimes c) &= 
\varphi_i(b)+\max\Bigl\{0, \varphi_i(c)-\varepsilon_i(b)
 \Bigr\}.
\end{align}
\end{subequations}
\end{remark}

Additionally, define the \defn{tensor product} of two morphisms $\Psi_1 : \cB_1 \to \cC_1$ 
and $\Psi_2 : \cB_2 \to \cC_2$ in $\sC_{\bal}$
to be the morphism $\cB_1\otimes \cB_2\to \cC_1\otimes \cC_2$
with $b_1\otimes b_2 \mapsto \Psi_1(b_1)\otimes \Psi_2(b_2)$.

\begin{proposition}
\label{prop:tensor_containment}
The tensor product of any two objects (respectively, morphisms) in $ \sC_{\bal}$ is an object (respectively, a morphism) in $\sC_{\bal}$.
\end{proposition}

\begin{proof}
Suppose $\cB,\cC\in \sC_{\bal}$.
The proof that $\cB \otimes \cC$ satisfies (A0), (A1), (A3), and (A4) is the same as for (abstract) Kashiwara crystals (i.e., when $\bfN = 1$)~\cite[\S2.3]{BumpSchilling}.

It remains to show that $\cB\otimes \cC$ satisfies axioms (A2) and~(A5) in Definition~\ref{abstract-def}.
To prove this, notice that 
we have
\begin{align*}
\varphi_i(b \otimes c)
& = \max\Bigl\{\varphi_i(b), \varphi_i(c) + \langle \weight(b), \alpha_i^{\vee} \rangle \Bigr\}
\\ & = \max\Bigl\{\varepsilon_i(b) + \langle \weight(b), \alpha_i^{\vee} \rangle, \varepsilon_i(c) + \langle \weight(c), \alpha_i^{\vee} \rangle + \langle \weight(b), \alpha_i^{\vee} \rangle \Bigr\}
\\ & = \max\Bigl\{\varepsilon_i(b) - \langle \weight(c), \alpha_i^{\vee} \rangle, \varepsilon_i(c) \Bigr\} + \langle \weight(b), \alpha_i^{\vee} \rangle + \langle \weight(c), \alpha_i^{\vee} \rangle
\\ & = \varepsilon_i(b \otimes c) + \langle \weight(b \otimes c), \alpha_i^{\vee} \rangle
\end{align*}
since $\max(a+c, b+c) = \max(a,b) + c$.
Hence, axiom (A2) is satisfied by $\cB \otimes \cC$.

Next, suppose $i \in I$ and $b\otimes c \in \cB\otimes \cC$ have $e_i(b\otimes c)\neq0$. Then the following are equivalent:
\begin{enumerate}
\item[(1)] $\varepsilon_i(b) \leq \varphi_i(c) = \varepsilon_i(c) + \langle \weight(c), \alpha_i^{\vee} \rangle$;
\item[(2)] $\varepsilon_i(b \otimes c) = \varepsilon_i(c)$;
\item[(3)] $e_i(b \otimes c) = b \otimes (e_i c)$.
\end{enumerate}
Thus if $e_i(b \otimes c) = b \otimes (e_i c)$ then axiom (A5) holds.
Alternatively, if $e_i(b \otimes c) = (e_i b) \otimes c$ then
\[
N_i \varepsilon_i(b \otimes c) = N_i \varepsilon_i(b) - N_i \langle \weight(c), \alpha_i^{\vee} \rangle \equiv N_i \varepsilon_i(b) \modu N_i),
\]
and so axiom (A5) holds as well. We conclude that $\cB\otimes \cC \in \sC_{\bal}$.

To see that tensor product of two morphisms $\Psi_1,\Psi_2$ in $\sC_{\bal}$ is another crystal morphism, note that morphisms preserve statistics and weights
so condition (M1) in Definition~\ref{morphism-def} holds for $\Psi_1\otimes \Psi_2$.
The other morphism conditions 
hold for $\Psi_1\otimes \Psi_2$ since successive $i$-indexed crystal operators act either twice on one factor or once on each of two factors.
\end{proof}

In the following result, $\cT_0$ denotes the crystal $\cT_\lambda$ from Example~\ref{ex:Tla_crystal} with $\lambda =0 \in \Lambda$.

\begin{proposition}
\label{prop:monoidal_cat}
The category of $\bfN$-root crystals $\sC_{\bal}$ is a monoidal category under $\otimes$ with identity object $\cT_0$ and satisfying
\be\label{eq:tensor_sum_rel}
\cB \otimes (\cC \oplus \cC') \iso (\cB \otimes \cC) \oplus (\cB \otimes \cC')
\quand
(\cB \oplus \cB') \otimes \cC \iso (\cB \otimes \cC) \oplus (\cB' \otimes \cC).
\ee
\end{proposition}

\begin{proof}
Our argument that the tensor product is associative is analogous to the proof when $\bfN = 1$ given in~\cite[Prop.~2.32]{BumpSchilling} (see also~\cite[Prop.~1.3.1, Ex.~1.3.4]{Kashiwara93}).

More precisely, the associativity of the tensor product is a consequence of the equivalent conditions used in the proof of Proposition~\ref{prop:tensor_containment} above.
In more detail, consider $\cB = \cB_1 \otimes \cdots \otimes \cB_{\ell}$ (we only consider this as a set as we will show next this is unambiguous), and for a fixed element $b = b_1 \otimes \cdots \otimes b_{\ell} \in \cB$ define values
\[
a_{im} := \varepsilon_i(b_m) - \sum_{j=m+1}^{\ell} \langle \weight(b_j), \alpha_i^{\vee} \rangle
\quand
M = \max\{ a_{i1}, \dotsc, a_{i\ell} \}.
\]
We claim that $e_i$ acts on the $j$-th factor, where $j = \max \{m \mid a_{im} = M\}$, and $\varepsilon_i(b) = M$.
This can be explicitly checked for the case $\ell = 3$, and this shows that the natural map $\cB_1 \otimes (\cB_2 \otimes \cB_3) \to (\cB_1\otimes \cB_2)\otimes \cB_3$ commutes with the raising crystal operators $e_i$ while preserving $\varepsilon_i$. 

The proof of the analogous property for $f_i$ and $\varphi_i$ is similar, just using
\[
a'_{im} := \varphi_i(b_m) + \sum_{j=1}^{m-1} \langle \weight(b_j), \alpha_i^{\vee} \rangle
\quand
M = \max\{ a'_{i1}, \dotsc, a'_{i\ell} \}.
\]
However, in this case $f_i$ acts on the $j$-th factor with $j = \min \{m \mid a'_{im} = M\}$.
We conclude that
$
(\cB_1 \otimes \cB_2) \otimes \cB_3 \cong \cB_1 \otimes (\cB_2 \otimes \cB_3)
$
as the natural map between the two objects is clearly weight-preserving.

The claim that $\cT_0 \otimes \cB \iso \cB \otimes \cT_0 \iso \cB$ is an easy computation using the tensor product rule to show the obvious maps
$t_0 \otimes b \mapsto b$,
$b \otimes t_0 \mapsto b$,
and
$t_0 \otimes b \mapsto b \otimes t_0$
are isomorphisms (see also~\cite[Ex.~1.3.4]{Kashiwara93}).
Constructing the remaining isomorphisms in~\eqref{eq:tensor_sum_rel} is straightforward.
\end{proof}

From the definitions, we have
$
\ch(\cB \otimes \cC) = \ch(\cB)  \ch(\cC)
$
whenever the characters of $\cB$, $\cC$, and $\cB \otimes \cC$ are well-defined.
However, we will see later (Example~\ref{ex:not_monoidal} below) that the functor $\powF{\bfN} \colon \sC_{\bal} \to \sC_1$ is generally not monoidal, as we can have $
\powF{\bfN}(\cB \otimes \cC) \not\iso \powF{\bfN}(\cB) \otimes \powF{\bfN}(\cC)$
for objects $\cB, \cC \in \sC_{\bal}$.

The following result recovers~\cite[Prop.~2.29]{BumpSchilling} (first noted in~\cite[\S1.3]{Kashiwara93}) and~\cite[Thm.~4.26]{MT2023} as special cases.
 
\begin{proposition}
\label{prop:regular_tensor_closed}
The full subcategory of $\sC_{\bal}$ consisting of (upper/lower) regular crystals is closed under tensor product, direct sums, and taking full subcrystals.
That is, for (upper/lower) regular crystals $\cB, \cC \in \sC_{\bal}$, both $\cB\oplus \cC$ and $\cB \otimes \cC$ are (upper/lower) regular, as are all full subcrystals of $\cB$.
\end{proposition}
 
\begin{proof}
That this subcategory is closed under tensor products follows from the fact that $e_i$ (respectively, $f_i$) acts on a factor that obtains the maximum of $\{a_{i1}, \dotsc, a_{i\ell}\}$, which is also $\varepsilon_i$ (respectively, maximum of $\{a'_{i1}, \dotsc, a'_{i\ell}\}$, which is $\varphi_i$).
Alternatively, this could be seen from~\eqref{eq:simpler_ephi}.
That it is closed under direct sums and taking full subcrystals is immediate from the definitions.
\end{proof}

The $\bfN$-root crystal $\cT_0$ is not regular, but it turns out that a different unit object makes the subcategory of regular $\bfN$-root crystals into a monoidal category.

\begin{proposition}
\label{prop:trivial_crystal}
The category of upper (respectively, lower) regular $\bfN$-root crystals has a right (respectively, left) unit object provided by the \defn{trivial crystal} $\one = \{\varnothing\}$ whose structure maps satisfy
\[
e_i \varnothing = f_i \varnothing = \zero,
\qquad
\varepsilon_i(\varnothing) = \varphi_i(\varnothing) = 0,
\quand 
\weight(\varnothing) = 0.
\]
Moreover, the category of regular $\bfN$-root crystals is a monoidal category with unit object $\one$.
\end{proposition}

\begin{proof}
For an upper (respectively, lower) regular $\bfN$-root crystal $\cB$, it suffices to check that the map $b\mapsto  b\otimes \varnothing$ (respectively, $b\mapsto \varnothing \otimes b$) is an isomorphism $\cB\iso \cB \otimes \one$ (respectively, $\cB\iso \one \otimes \cB$), and this is straightforward.
\end{proof}

\begin{example}
\label{ex:R=CT}
For the crystals $\cT_{\lambda}$ and $\cR_{\lambda}$ in Examples~\ref{ex:Tla_crystal} and \ref{ex:Rla_crystal}, we have $\cR_{\lambda} \iso \one \otimes \cT_{\lambda}$.
Moreover, if $\lambda = 0$ then $\one \iso \cR_\lambda$. 
\end{example}

We note the following fact that we will utilize later.

\begin{lemma}
\label{lemma:T_upper_prod}
Let $\cB$ be an upper regular crystal, and let $\lambda \in \Lambda$.
Then
$
\cT_{\lambda} \otimes \cB
$
is upper regular.
\end{lemma}

\begin{proof}
This follows from the tensor product rule and the fact that $\varepsilon_i(t_{\lambda})  = \varphi_i(t_{\lambda}) = -\infty$.
\end{proof}

\section{Square root crystals}
\label{sec:sqrt_crystals}

In this section we study abstract $\bfN$-root crystals for one particular choice of $\bfN$-root data.
These objects generalize the crystals $\SetTab_n(\lambda)$ of \defn{set-valued tableaux}  considered in \cite{MT2023,MTY,Yu23}.
Our main application of this more general theory is to construct 
the direct limit crystal $\SetTab_n(\infty)$ (see Theorem~\ref{thm:directed_system})
and its Demazure filtration (see Theorem~\ref{dem-thm}).

\subsection{Conventions}
For nonnegative integers $k \leq m$, define $[k, m] := \{k, k+1, \dotsc, m\}$ and $[m] := [1, m] = \{1,2, \dotsc, m\}$.
For the remainder of this paper, unless otherwise noted, we adopt the following conventions.

First, fix some positive integer $n$.
Define the index set to be $I := [n-1]$ and set $\bfN := 2$ (meaning  $N_i = 2$ for all $i \in I$).
Let $V$ be the Euclidean space $\RR \oplus \RR^n$ with standard basis $\{\e_0, \e_1, \dotsc, \e_n\}$ and standard inner product $\langle \e_i, \e_j \rangle = \delta_{ij}$.
Define $\Phi \subset \Lambda$ to be the root system and weight lattice
\[
\Phi := \{ \e_i - \e_j \mid i,j \in [n]\text{ with }i\neq j\} \subset \Lambda := \ZZ\{\e_0, \e_1, \dotsc, \e_n\} = \ZZ \oplus \ZZ^n
\]
and set our fixed choice of simple roots and fundamental weights to be
\[
\{ \alpha_i := \e_i - \e_{i+1} \mid i \in I\}
\quand
\{ \Lambda_i := \e_1 + \e_2 + \cdots + \e_i \mid i \in I \},
\]
respectively.
For convenience, we also define $\Lambda_n := \e_1 + \e_2 + \cdots +  \e_n$, and we sometimes abuse terminology by referring to $\Lambda_n$ as a fundamental weight.

Notice that $\Phi$ is the usual root system of type $\gl_n$
but $\Lambda$ is slightly larger than the usual weight lattice for this type.
We consider the Weyl group of $\Phi$ to be the symmetric group $S_n$ of permutations of $[n]$, whose   longest element $w_0 = [n, \dotsc, 2, 1]$ is the reverse permutation of length $\ell(w_0) = \binom{n}{2}$ sending $i \mapsto n+1-i$.

Next, define  $\bal = (\alpha_i^{(m)} \in \Lambda \mid i \in I, \; m \in \ZZ/2\ZZ)$ where
\be\label{eq:sqrt_roots}
\alpha_i^{(0)} := -\e_0 + \e_i
\quand
\alpha_i^{(1)} := \e_0  - \e_{i+1}
\ee
for each $i \in I$.
Finally, fix the $\bfN$-root type $(\Lambda, \Phi, \bal)$.

We let  $\sqrt{\sC} := \sC_{\bal}$ denote the corresponding (monoidal) category of abstract $\bfN$-root crystals.
For the above choices of $\Phi$ and $\Lambda$, the category $\sC_1$ is the usual category of abstract Kashiwara $\gl_n$-crystals (but with a slightly enlarged weight lattice).
For this reason, we refer to abstract crystals in $\sqrt{\sC}$ as \defn{$\sqgln$-crystals} or \defn{square root crystals} following~\cite{MT2023,MTY} (which, more precisely, studied an important subcategory that we will define below).

For convenience, we 
include a self-contained definition of $\sqrt{\sC}$ that
is equivalent to Definition~\ref{abstract-def} and Proposition~\ref{prime-prop}
specialized to
 the $\bfN$-root type described above.

\begin{definition}[Square root crystal] 
\label{square-root-def}
A set $\cB$ with maps
\[
\weight \colon  \cB\to \ZZ \oplus \ZZ^n,
\quad
e_i,f_i \colon  \cB \sqcup\{\zero\}\to \cB \sqcup \{\zero\},
\quand
\varepsilon_i, \varphi_i \colon \cB \to \tfrac{1}{2}\ZZ \sqcup \{-\infty\}
\quad\text{for $i \in [n-1]$} 
\]
satisfying $e_i\zero=f_i\zero=\zero$
is a \defn{$\sqgln$-crystal}
if the following holds
for all $i \in [n-1]$ and $b,c \in \cB$:
\begin{enumerate}
\item[(1)] $\varepsilon_i(b) = -\infty$ if and only if $\varphi_i(b) = -\infty$ in which case $e_i b = f_i b = \zero$; 

\item[(2)] if $\varepsilon_i(b)$ and $\varphi_i(b) $ are finite then $\varphi_i(b) - \varepsilon_i(b) = \langle \weight(b), \e_i - \e_{i+1} \rangle \in \ZZ$;

\item[(3)] $e_i b = c$ if and only if $b = f_i c$ in which case $\varepsilon_i(b)$, $\varepsilon_i(c)$, $\varphi_i(b)$, $\varphi_i(c)$ are all finite with
\[
\varepsilon_i(b) - \varepsilon_i(c) = \varphi_i(c) - \varphi_i(b) = \tfrac{1}{2}
\quand
\weight(c) - \weight(b) = \begin{cases}
 \e_i -\e_0 &\text{if }\varepsilon_i(b) \in \ZZ \\
\e_0 -\e_{i+1} &\text{if }\varepsilon_i(b)\notin\ZZ.
\end{cases}
\]

\end{enumerate}
\end{definition}

To represent the characters of square root crystals, we identify $\ZZ\llbracket\Lambda\rrbracket$ with the set of formal Laurent series $\ZZ\llbracket\beta^{\pm1}, x_1^{\pm1}, \dotsc, x_n^{\pm 1}\rrbracket$ by setting $\beta := x^{-\e_0}$ and $x_i := x^{\e_i}$.
Note the nonstandard mapping of $\e_0$ to the \emph{inverse} of $\beta$, which is done to match the literature on set-valued tableaux and symmetric $\beta$-Grothendieck functions.
This means that if $\lambda= \sum_{i=0}^n \lambda_i \e_i \in \Lambda$ then
\[
\lambda \mapsto x^{\lambda} = \beta^{-\lambda_0} x_1^{\lambda_1} \cdots x_n^{\lambda_n} \in \ZZ\llbracket\Lambda\rrbracket.
\]
For use later on, let $\ZZ[\beta][x_1, \dotsc, x_n]^{S_n}$ denote the subring of polynomials invariant under the natural $S_n$-action on the $x$-variables. We refer to the elements of this subring as \defn{symmetric functions}.

\subsection{Polynomial crystals}

We now present some examples of $\sqgln$-crystals and introduce a subcategory whose objects are analogous to the Kashiwara crystals arising from highest weight representations of $U_q(\gl_n)$.

\begin{definition}[{\cite{MT2023,Yu23}}]
\label{def:standard_crystal}
The \defn{standard $\sqgln$-crystal} is the set
$
\SS_n := \{ S \mid \emptyset \neq S \subseteq [n]\}
$
with weight map $\weight(S) = (1 - \abs{S})\e_0 +\sum_{i \in S} \e_i \in \Lambda$, with crystal operators
\[ \ba e_i S &= \begin{cases}
S\sqcup \{i\}&\text{if }S\cap \{i,i+1\} = \{i+1\}, \\
S\setminus \{i+1\}&\text{if }S\cap \{i,i+1\} = \{i,i+1\}, \\
\zero&\text{otherwise},
\end{cases}
\allowdisplaybreaks \\
f_iS &= \begin{cases}
S\sqcup \{i+1\}&\text{if }S\cap \{i,i+1\} = \{i\}, \\
S\setminus \{i\}&\text{if }S\cap \{i,i+1\} = \{i,i+1\}, \\
\zero&\text{otherwise},
\end{cases}
\ea
\]
and with statistics
\[
\ba \varepsilon_i(S) &= \begin{cases}
1&\text{if }S\cap \{i,i+1\} = \{i+1\}, \\
\frac{1}{2}&\text{if }S\cap \{i,i+1\} = \{i,i+1\}, \\
0&\text{otherwise},
\end{cases}
& \quad
\varphi_i(S) &= \begin{cases}
1&\text{if }S\cap \{i,i+1\} = \{i\}, \\
\frac{1}{2}&\text{if }S\cap \{i,i+1\} = \{i,i+1\}, \\
0&\text{otherwise}.
\end{cases}
\ea
\]
\end{definition}

It is easy to see that the standard $\sqgln$-crystal $\SS_n$ is regular.

\begin{example}
The crystal graph of $\SS_3$ is 
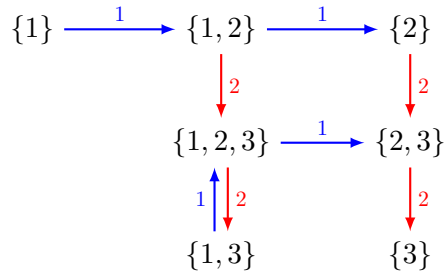
\[
    \begin{tikzpicture}[xscale=2.5, yscale=1.5,>=latex,baseline=(z.base)]
    \node at (0,0.0) (z) {};
      \node at (0,0) (T0) {${\{1\}}$};
      \node at (2,0) (T1) {${\{2\}}$};
      \node at (2,-2) (T2) {${\{3\}}$};
      \node at (1,0) (U0) {${\{1,2\}}$};
      \node at (2,-1) (U1) {${\{2,3\}}$};
      \node at (1,-1) (V0) {${\{1,2,3\}}$};
      \node at (1,-2) (V1) {${\{1,3\}}$};
      \draw[->,thick,color=blue]  (T0) -- (U0) node[midway,above,scale=0.75] {$1$};
      \draw[->,thick,color=blue]  (U0) -- (T1) node[midway,above,scale=0.75] {$1$};
      \draw[->,thick,color=red]  (T1) -- (U1) node[midway,right,scale=0.75] {$2$};
      \draw[->,thick,color=red]  (U1) -- (T2) node[midway,right,scale=0.75] {$2$};
      \draw[->,thick,color=red]  (U0) -- (V0) node[midway,right,scale=0.75] {$2$};
      \draw[->,thick,color=red]  ([xshift=1pt]V0.south) -- ([xshift=1pt]V1.north) node[midway,right,scale=0.75] {$2$};
      \draw[->,thick,color=blue]  ([xshift=-1pt]V1.north) -- ([xshift=-1pt]V0.south) node[midway,left,scale=0.75] {$1$};
      \draw[->,thick,color=blue]  (V0) -- (U1) node[midway,above,scale=0.75] {$1$};
     \end{tikzpicture}
\]
and its character is
$
\ch(\SS_3) = x_1 + x_2 + x_3 + \beta ( x_1 x_2 + x_1 x_3 + x_2 x_3) + \beta^2 x_1 x_2 x_3.
$
\end{example}

More generally,  the character of the standard $\sqgln$-crystal is
\[
\ch(\SS_n) = \esf_1(x) +\beta \esf_2(x) + \beta^2 \esf_3(x) + \cdots +\beta^{n-1} \esf_n(x),
\]
where $\esf_k(x)$ is the 
\defn{elementary symmetric polynomial}
$
\esf_k(x) = 
 \sum_{1\leq i_1<i_2<\dots<i_k\leq n} x_{i_1} x_{i_2} \cdots x_{i_k}.
$

\begin{definition}[Polynomial category]
Define the category of \defn{polynomial $\sqgln$-crystals} $\sqrt{\sC}_{\poly}$ as the smallest full subcategory of $\sqrt{\sC}$ that contains $\SS_n$ and that is closed under taking {full} subcrystals, finite direct sums, and finite tensor products, including $\cB^{\otimes 0} = \one$ for any $\cB \in \sqrt{\sC}_{\poly}$.
\end{definition}

For any connected $\cB \in \sqrt{\sC}_{\poly}$, there exists a $k \in \ZZ_{\geq 0}$ such that $\cB$ strictly embeds in $\SS_n^{\otimes k}$.
Therefore, all polynomial $\sqgln$-crystals $\cB$ are finite and regular by Proposition~\ref{prop:regular_tensor_closed}.
Thus, by Proposition~\ref{sym-prop} their characters are in the  ring  $\ZZ[\beta][x_1, \dotsc, x_n]^{S_n}$.

\begin{remark}\label{no-loss-rem}
Our polynomial $\sqgln$-crystals are effectively the same as the objects called \defn{normal $\sqgln$-crystals} in~\cite{MT2023,MTY}.
To be precise, each normal $\sqgln$-crystal is just a polynomial $\sqgln$-crystal whose weight map is modified by projecting $\e_0 \mapsto 0$ in $\Lambda$ and whose statistics $\varepsilon_i$ and $\varphi_i$ are rescaled by $2$.
On characters, these operations correspond to setting $\beta = 1$.
This substitution loses no information if we keep track of the quantity
\[
D := \langle \weight(b), \e_0 + \e_1 + \cdots + \e_n \rangle \in \ZZ,
\]
which
is constant for all $b \in \cB$ when
$\cB \in \sqrt{\sC}_{\poly}$ is connected,
as we have the identity 
\[
\beta^{D} \ch(\cB) \, \bigr\rvert_{x_i \mapsto \beta^{-1} x_i} = \ch(\cB) \, \bigr\rvert_{\beta = 1}.
\]
\end{remark}

Polynomial $\sqgln$-crystals are not as well-behaved as their classical analogs.
In the corresponding classical category of \defn{polynomial Kashiwara crystals} (called \defn{normal crystals} in \cite{BumpSchilling}), a connected component is uniquely determined (up to isomorphism) by its unique highest weight element of weight $\lambda$ or by its character.
Following \cite{Kashiwara90,Kashiwara91}, we denote an arbitrary crystal in this isomorphism class by $\cB(\lambda)$; then the standard $\gl_n$-crystal is $\cB(\Lambda_1)$.

Connected objects in $ \sqrt{\sC}_{\poly}$ may have multiple highest weight elements, and these do not uniquely determine the crystal's isomorphism class \cite{MT2023,MTY}.
Yet, the characters of polynomial $\sqgln$-crystals exhibit a striking positivity phenomenon, which we discuss in the next section.

\begin{example}
\label{ex:not_monoidal}
We now demonstrate that the functor $\powF{2} \colon \sqrt{\sC} \to \sC_1$ is not monoidal.
Consider the tensor product $\SS_3 \otimes \SS_3$. The following is one connected component of $\powF{2}(\SS_3 \otimes \SS_3)$:
\[
    \begin{tikzpicture}[xscale=3.5, yscale=1.5,>=latex,baseline=(z.base)]
      \node at (-3,1) (T0) {${\{1,2\} \otimes \{1,2\}}$};
      \node at (-3,0) (T1) {${\{1,2\} \otimes \{1,3\}}$};
      \node at (-2,0) (T2) {${\{1,3\} \otimes \{1,3\}}$};
      \node at (-3,-1) (T3) {${\{2\} \otimes \{1,2,3\}}$};
      \node at (-1,0) (T4) {${\{1,3\} \otimes \{2,3\}}$};
      \node at (-1,-1) (T5) {${\{2,3\} \otimes \{2,3\}}$};
      \node at (-2,-1) (T6) {${\{2,3\} \otimes \{1,3\}}$};
      \node at (-1,1) (T7) {${\{1,2,3\} \otimes \{2\}}$};
      \node at (-2,1) (T8) {${\{1,3\} \otimes \{1,2\}}$};
      \draw[->,thick,color=blue]  (T1) -- (T3) node[midway,right,scale=0.75] {$1$};
      \draw[->,thick,color=blue]  (T2) -- (T4) node[midway,above,scale=0.75] {$1$};
      \draw[->,thick,color=blue]  (T4) -- (T5) node[midway,right,scale=0.75] {$1$};
      \draw[->,thick,color=blue]  (T8) -- (T7) node[midway,above,scale=0.75] {$1$};
      \draw[->,thick,color=red]  (T0) -- (T1) node[midway,right,scale=0.75] {$2$};
      \draw[->,thick,color=red]  (T1) -- (T2) node[midway,above,scale=0.75] {$2$};
      \draw[->,thick,color=red]  (T3) -- (T6) node[midway,above,scale=0.75] {$2$};
      \draw[->,thick,color=red]  (T7) -- (T4) node[midway,right,scale=0.75] {$2$};
      \end{tikzpicture}
\]
This crystal has multiple highest (and lowest) weight elements.

On the other hand, $\powF{2}(\SS_3) $ is the direct sum of the connected crystals 
\[\cB(\Lambda_1) ,\quad \cT_{-\e_0}\otimes \cB(\Lambda_2), \quand \cT_{-2\e_0}\otimes \cB(\Lambda_3),\]
which have distinct highest weights, and so every connected component in its tensor square will have a unique highest weight element.
(This could also be verified by a direct computation.)
Hence, we must have $\powF{2}(\SS_3 \otimes \SS_3) \not\iso \powF{2}(\SS_3) \otimes \powF{2}(\SS_3)$.
\end{example}

Next, we define a monoidal functor on square root crystals that corresponds on characters to setting $\beta = 0$,
or more generally to taking the coefficient of $\beta^0$.
Such a functor cannot be defined on all objects in $\sqrt{\sC}$, so we instead will consider a subcategory.

Given $\cB \in \sqrt{\sC}$, define the \defn{defect} of any $b \in \cB$ by negating the coefficient of $\e_0$ in $\weight(b)$:
\begin{equation}
\label{eq:defect}
\defect(b) := \langle \weight(b), -\e_0 \rangle.
\end{equation}
Let $\sqrt{\sC}_+$ be the full subcategory of $\sqrt{\sC}$ consisting of all objects $\cB$ with $\defect(b) \geq 0$ for all $b \in \cB$.
Then define $\sK_+$ to be the full subcategory of $\sqrt{\sC}_+$ consisting of all objects $\cB$ such that if $b \in \cB$ has $\defect(b) = 0$ then $\varepsilon_i(b) \in \ZZ\sqcup \{-\infty\}$ for all $i \in I$.
The following is obvious from the definitions.

\begin{proposition}
\label{prop:plus_monoidal_cat}
Both $\sqrt{\sC}_+$ and $\sK_+$ are monoidal subcategories of $\sqrt{\sC}$ that are closed under taking (not necessarily full) subcrystals.
\end{proposition}

The following condition is sometimes useful for checking that a crystal is in $\sK_+$.

 \begin{proposition}\label{shortcut-prop}
 If $\cB \in\sqrt{\sC}_+$ is upper or lower regular then $\cB \in \sK_+$.
 \end{proposition}
 
 \begin{proof}
 Suppose $\cB \in \sqrt{\sC}_+$ is upper regular and $b \in \cB$
 has $\defect(b)=0$. If $\varepsilon_i(b) \in \frac{1}{2}+\ZZ$ then   $\varepsilon_i(b) \in \{\frac{1}{2}, \frac{3}{2}, \frac{5}{2},\dots\}$, and then the element $\zero \neq e_i b \in \cB$ has
 $\weight(e_ib)=\weight(b)+\alpha_i^{(1)}$ by axiom (A5) in Definition~\ref{abstract-def} so
  $\defect(e_ib) = \defect(b) - 1 =-1$. 
 Since all elements of $\cB$ must have nonnegative defect,   we must have $\varepsilon_i(b) \in \ZZ\sqcup\{-\infty\}$
so $\cB \in \sK^+$.

 If $\cB \in \sqrt{\sC}_+$ is lower regular and $b\in \cB$ has $\defect(b)=0$,
then it follows similarly that $\varphi_i(b) \notin \frac{1}{2}+\ZZ$,
in which case axiom (A2) implies that we again have  $\varepsilon_i(b) \in \ZZ\sqcup\{-\infty\}$
so $\cB \in \sK_+$.
 \end{proof}
 
Since polynomial $\sqgln$-crystals are regular and clearly $ \SS_n^{\otimes k} \in \sqrt{\sC}_+$, we obtain the following. 

\begin{corollary}
The category $\sqrt{\sC}_{\poly}$ is a full subcategory of $\sK_+$. 
\end{corollary}

Define the functor $\defectF{2} \colon \sK_+ \to \sC_1$ by applying $\powF{2}$ and then restricting to the defect $0$ elements.
In other words, $\defectF{2}(\cB)= \{b \in \cB \mid \defect(b) = 0 \}$ as a set, with the same weight map as $\cB$, but with the crystal operators and statistics of $\powF{2}(\cB)$ from \eqref{powF-eq} restricted to $\defectF{2}(\cB)$.
On morphisms, we define $\defectF{2}$ to act by restriction.

\begin{proposition}
\label{defectF-prop}
The functor $\defectF{2}$ is monoidal; that is, $\defectF{2}(\cB \otimes \cC) \iso \defectF{2}(\cB) \otimes \defectF{2}(\cC)$ for all $\cB,\cC\in \sK_+$.
\end{proposition}

\begin{proof}
Consider objects $\cB,\cC\in \sK_+$.
Since $\weight(b\otimes c) = \weight(b) + \weight(c)$ for all $b\otimes c \in \cB\otimes \cC$, the defect condition defining $\sqrt{\sC}_+$ implies that there is an equality of sets $ \defectF{2}(\cB) \otimes \defectF{2}(\cC) = \defectF{2}(\cB\otimes \cC)$.
We claim that this is actually an equality of $\gl_n$-crystals.

Suppose $b \in \cB$ and $c\in \cC$. If $\varepsilon_i(b)$ and $\varepsilon_i(c)$ are both contained in $\ZZ \sqcup\{-\infty\}$, then it is clear from the definitions that
$
\{\varepsilon_i(b),\ \varphi_i(b),\ \varepsilon_i(c),\ \varphi_i(c),\ \varepsilon_i(b\otimes c),\ \varphi_i(b\otimes c) \} \subset \ZZ \sqcup \{-\infty\}.
$ 
Furthermore, both $e_i^2$ and $f_i^2$ preserve the defect as they change the weight by $\pm \alpha_i$.
Hence, the statistics $\varepsilon_i$ and $\varphi_i$ agree on the two sets.
Therefore, if we can show that
\[
\ba
e_i^2(b\otimes c) &= \begin{cases}
b \otimes (e_i^2c) &\text{if } \varepsilon_i(b) \leq \varphi_i(c), \\
(e_i^2b) \otimes c &\text{if } \varepsilon_i(b) > \varphi_i(c),
\end{cases}
& \qquad
f_i^2(b\otimes c) &= \begin{cases}
b \otimes (f_i^2c) &\text{if } \varepsilon_i(b) < \varphi_i(c), \\
(f_i^2b) \otimes c &\text{if } \varepsilon_i(b) \geq \varphi_i(c),
\end{cases}
\ea
\]
then the claim will follow.

We show the desired identity for $e_i^2(b \otimes c)$ as the proof for $f_i^2(b \otimes c)$ is similar.
Suppose $\varepsilon_i(b) \leq \varphi_i(c)$. Then $e_i(b \otimes c) = b \otimes (e_i c)$.
If $e_ic=\zero$ then the desired identity is immediate.
Otherwise,  we have $e_i^2(b \otimes c) = b \otimes (e_i^2 c)$ as $\varepsilon_i(b) \leq \varphi_i(c) < \varphi_i(e_i c) = \varphi_i(c) + \frac{1}{2}$ by axiom~(A4) in Definition~\ref{abstract-def}.
Now suppose instead that $\varepsilon_i(b) > \varphi_i(c)$, so that $e_i(b \otimes c) = (e_i b) \otimes c$.
The desired identity is again immediate if $e_ib=\zero$ so assume otherwise.
Then $\varepsilon_i(b) \geq  \varphi_i(c) + 1$ since $\varepsilon_i(b)$ and $\varphi_i(c)$ are both in $\ZZ\sqcup\{-\infty\}$, and so $\varepsilon_i(e_i b) = \varepsilon_i(b) - \frac{1}{2} > \varphi_i(c)$.
Thus  $e_i^2(b \otimes c) = (e_i^2 b) \otimes c$.
\end{proof}

\begin{remark}
\label{rem:functor_genalization}
The preceding argument only requires that $\varepsilon_i(b) - \varphi_i(c) \in \ZZ$ for all $b \in \cB$ and $c \in \cC$ with $\varepsilon_i(b), \varphi_i(c)\neq -\infty$.
This gives a sufficient condition for the functor $\powF{2}$ to be monoidal.
In another direction, one could also replace the defect condition $\defect(b)\geq 0$ with the opposite inequality $\defect(b)\leq 0$; then our functor would correspond on characters to specializing $\beta^{-1}=0$.
Using these observations, one could define several variants of the subcategory $\sK_+$ and the functor $\defectF{2} \colon \sK_+ \to \sC_1$.
However, our applications only involve the version presented above.
\end{remark}

We also require another functor that is similar in nature to $\defectF{2}$.

\begin{definition}
For $J \subseteq [n]$, define the \defn{restriction functor} $\restrictF{J} \colon  \sqrt{\sC} \to \sqrt{\sC}$ on objects by setting 
\[
\restrictF{J}(\cB) = \{b \in \cB \mid \langle \weight(b), \e_k \rangle = 0\text{ for all $k \in [n] \setminus J$}\}.
\]
We view this set as a $\sqgln$-crystal in the following way.
The definitions of $e_i$, $f_i$, $\varepsilon_i$, $\varphi_i$ when\[ i \in  [n-1]\text{ has }\{i,i+1\} \subseteq J
\] are the same as for $\cB$ restricted to $\restrictF{J}(\cB)$. For all other indices $i$ we set $\varepsilon_i(b)= \varphi_i(b) = -\infty$ and $e_i b = f_i b = \zero$ for all $b \in \restrictF{J}(\cB)$.
On morphisms, $\restrictF{J}$ acts as restriction.
\end{definition}

We will primarily use $\restrictF{J}$ for sets of the form $J = [k,n]$.
In terms of the character of $\cB$, the restriction functor corresponds to deleting all terms involving $x_k$ for $k \in [n] \setminus J$
(or setting $x_k=0$ when no negative powers of these variables occur).
The functor $\restrictF{J}$ preserves tensor products in $\sqrt{\sC}_{\poly}$ since the elements of any crystal in this category satisfy $\langle \weight(b),\e_k\rangle\geq 0$ for all $k \in [n]$.

\subsection{Set-valued tableaux}

Let $P^+ := \bigoplus_{i \in [n]} \ZZ_{\geq 0} \Lambda_i$ be the set of vectors
$
\lambda := \lambda_1 \e_1 +  \cdots + \lambda_n \e_n \in \Lambda
$
with $\lambda_1\geq  \dots \geq \lambda_n \geq0$.
We refer to the elements $\lambda \in P^+$ as \defn{partitions}.
The  \defn{Young diagram} of  $\lambda \in P^+$ is the set
\[
\D_{\lambda} = \{ (i,j) \in \ZZ_{>0} \times \ZZ_{>0} \mid j\leq\lambda_i\},
\]
which we draw in English convention.
If we express $\lambda = \sum_{i=1}^n m'_i \Lambda_i$ then $m'_i$ counts the number of columns of height $i$ in $\D_\lambda$ (that is, the multiplicity of $i$ in the conjugate partition $\lambda'$ whose Young diagram is the transpose of $\D_\lambda$).

Given a partition $\lambda=\sum_{i=1}^n \lambda_i \e_i= \sum_{i=1}^n m'_i \Lambda_i \in P^+$, we define its \defn{size} by $\abs{\lambda} = \sum_{i=1}^n \lambda_i$ and its \defn{length} by $\ell(\lambda) := \max \{ i \mid m'_i \neq 0\} = \max \{i \mid \lambda_i > 0\}$.
When $\lambda=0$ we interpret the second two formulas to mean $ \ell(0) =0$.

Let $\D$ be a subset of $\ZZ_{>0} \times \ZZ_{>0}$.
A \defn{set-valued tableau} of shape $\D$ is a map $T \colon \D \to \SS_n$.
If $T$ is such a tableau then we write $T_{ij}$ for the set in $\SS_n$ that $T$ assigns to each position $(i,j) \in \D$.
We define $T$ to be \defn{semistandard} if $\max T_{i,j} \leq \min T_{i,j+1}$ (respectively, $\max T_{i,j} < \min T_{i+1,j}$) whenever $(i, j) \in \D$ and $(i,j+1) \in \D$ (respectively, $(i+1,j) \in \D$).
Pictorially, we have the local rules
\[
\ytableaushort{AB,C}
\qquad \Longleftrightarrow \qquad
\begin{array}{c@{\;\;}c@{\;\;}c}
\max A & \leq & \min B \\[-8pt] \rotatebox{-90}{$<$} \\[5pt] \min C
\end{array}
\]
A \defn{semistandard Young tableau}  is a semistandard set-valued tableau $T$ whose  entries all have size $1$; that is, with $\abs{T_{ij}}=1$ for all $(i,j) \in \D$. 

For each $\lambda \in P^+$
let $\SetTab_n(\lambda)$ be the set of semistandard set-valued tableaux of shape $\D_\lambda$.
Denote by $u_{\lambda}$ the unique set-valued tableau with all entries in row $i$ filled with $\{i\}$; that is, with $T_{ij} = \{i\}$ for all $(i,j) \in \D_{\lambda}$.
When there is no risk of confusion (mainly when $n < 10$), we   write sets $\{a_1, a_2,\dotsc, a_k\} \subseteq [n]$ as words $a_1a_2 \cdots a_k$.
For example, if $\lambda=4\e_1+ 2\e_2 + \e_3 \in P^+$ then
\[
u_\lambda = 
\ytabc{0.5cm}{1 & 1  & 1 & 1 \\
2& 2  \\
3 }
\qquand
\ytabc{0.5cm}{23 & 34  & 5 & 57 \\
4& 57  \\
7 } \in \SetTab_n(\lambda)\text{ if }n\geq 7.
\]

The \defn{column reading word} $R_{\col} \colon \SetTab_n(\lambda) \to \SS_n^{\otimes \abs{\lambda}}$ is the embedding defined by reading the columns of a tableau from bottom-to-top, while iterating over the columns from left-to-right.
By convention, the reading word of the zero partition is the unique map $\SetTab_n(0) \to \SS_n^{\otimes 0} = \one$. By abuse of notation, we identify the empty tableau with $\varnothing \in \one$.

\begin{example}
\label{ex:column_word}
For set-valued tableaux of shape $\lambda = (2,2,1)$, the column reading word sends
\[
\ytableaushort{AB,CD,E}
\quad \mapsto \quad
E \otimes C \otimes A \otimes D \otimes B.
\]
\end{example}

\begin{figure}[t]
\[
    \begin{tikzpicture}[xscale=3.5, yscale=1.4,>=latex,baseline=(z.base)]
      \node at (0,4) (T0) {${\{1\} \otimes \{1\}}$};
      \node at (0,3) (T1) {${\{1\} \otimes \{1,2\}}$};
      \node at (0,2) (T2a) {${\{1\} \otimes \{2\}}$};
      \node at (1,3) (T2b) {${\{1\} \otimes \{1,2,3\}}$};
      \node at (0,1) (T3a) {${\{1,2\} \otimes \{2\}}$};
      \node at (1,2) (T3b) {${\{1\} \otimes \{2,3\}}$};
      \node at (2,3) (T3c) {${\{1\} \otimes \{1,3\}}$};
      \node at (0,0) (T4a) {${\{2\} \otimes \{2\}}$};
      \node at (1,1) (T4b) {${\{1,2\} \otimes \{2,3\}}$};
      \node at (2,2) (T4c) {${\{1\} \otimes \{3\}}$};
      \node at (1,0) (T5a) {${\{2\} \otimes \{2,3\}}$};
      \node at (2,1) (T5b) {${\{1,2\} \otimes \{3\}}$};
      \node at (2,0) (T6a) {${\{2\} \otimes \{3\}}$};
      \node at (3,1) (T6b) {${\{1,2,3\} \otimes \{3\}}$};
      \node at (3,0) (T7a) {${\{2,3\} \otimes \{3\}}$};
      \node at (4,1) (T7b) {${\{1,3\} \otimes \{3\}}$};
      \node at (4,0) (T8) {${\{3\} \otimes \{3\}}$};
      \draw[->,thick,color=blue]  (T0) -- (T1) node[midway,right,scale=0.75] {$1$};
      \draw[->,thick,color=blue]  (T1) -- (T2a) node[midway,right,scale=0.75] {$1$};
      \draw[->,thick,color=blue]  (T2a) -- (T3a) node[midway,right,scale=0.75] {$1$};
      \draw[->,thick,color=blue]  (T2b) -- (T3b) node[midway,right,scale=0.75] {$1$};
      \draw[->,thick,color=blue]  (T3a) -- (T4a) node[midway,right,scale=0.75] {$1$};
      \draw[->,thick,color=blue]  (T3b) -- (T4b) node[midway,right,scale=0.75] {$1$};
      \draw[->,thick,color=blue]  (T4b) -- (T5a) node[midway,right,scale=0.75] {$1$};
      \draw[->,thick,color=blue]  (T4c) -- (T5b) node[midway,right,scale=0.75] {$1$};
      \draw[->,thick,color=blue]  (T5b) -- (T6a) node[midway,right,scale=0.75] {$1$};
      \draw[->,thick,color=blue]  (T6b) -- (T7a) node[midway,right,scale=0.75] {$1$};
      \draw[->,thick,color=red]  (T1) -- (T2b) node[midway,above,scale=0.75] {$2$};
      \draw[->,thick,color=red]  (T2a) -- (T3b) node[midway,above,scale=0.75] {$2$};
      \draw[->,thick,color=red]  (T3b) -- (T4c) node[midway,above,scale=0.75] {$2$};
      \draw[->,thick,color=red]  (T3a) -- (T4b) node[midway,above,scale=0.75] {$2$};
      \draw[->,thick,color=red]  (T4b) -- (T5b) node[midway,above,scale=0.75] {$2$};
      \draw[->,thick,color=red]  (T5b) -- (T6b) node[midway,above,scale=0.75] {$2$};
      \draw[->,thick,color=red]  (T4a) -- (T5a) node[midway,above,scale=0.75] {$2$};
      \draw[->,thick,color=red]  (T5a) -- (T6a) node[midway,above,scale=0.75] {$2$};
      \draw[->,thick,color=red]  (T6a) -- (T7a) node[midway,above,scale=0.75] {$2$};
      \draw[->,thick,color=red]  (T7a) -- (T8) node[midway,above,scale=0.75] {$2$};
      \draw[->,thick,color=red]  ([yshift=1pt]T2b.east) -- ([yshift=1pt]T3c.west) node[midway,above,scale=0.75] {$2$};
      \draw[->,thick,color=blue]  ([yshift=-1pt]T3c.west) -- ([yshift=-1pt]T2b.east) node[midway,below,scale=0.75] {$1$};
      \draw[->,thick,color=red]  ([yshift=1pt]T6b.east) -- ([yshift=1pt]T7b.west) node[midway,above,scale=0.75] {$2$};
      \draw[->,thick,color=blue]  ([yshift=-1pt]T7b.west) -- ([yshift=-1pt]T6b.east) node[midway,below,scale=0.75] {$1$};
     \end{tikzpicture}
\]
\caption{A connected component of $\SS_3 \otimes \SS_3$ isomorphic to $\SetTab_3(2\Lambda_1)$.}
\label{fig:S3square_cc}
\end{figure}

Any semistandard set-valued tableau $T \in \SetTab_n(\lambda)$ can be recovered from its column reading word $R_{\col} (T)$.
On the other hand, results in \cite{Yu23,MT2023} show that the image of $\SetTab_n(\lambda)$ under $R_{\col} $ is a full subcrystal of $\SS_n^{\otimes \abs{\lambda}}$.
We may therefore consider $\SetTab_n(\lambda)$ as a $\sqgln$-crystal with the structure induced by the column reading word embedding.

\begin{remark}
\label{rem:sigrule}
Yu~\cite{Yu23} defined this  crystal 
structure on  $\SetTab_n(\lambda)$
by introducing an explicit \defn{signature rule}  to compute $e_i$ and $f_i$.
There are also combinatorial formulas for the   statistics $\varepsilon_i$ and $\varphi_i$; see \cite[Cor.~2.14]{MTY}.
The weight map for  $\SetTab_n(\lambda)$ is
\be
\label{eq:explicit_tableau_weight}
\weight(T) = (\abs{\lambda}-\abs{T})\e_0 + \sum_{(i,j) \in \D_\lambda} \sum_{k \in T_{ij}} \e_k  \in \Lambda,
\qquad\text{where } \abs{T} := \sum_{(i,j) \in \D_\lambda} \abs{T_{ij}}.
\ee
\end{remark}

To make our next statement about $\SetTab_n(\lambda)$, we need to introduce a few additional definitions.
Write $s_i=(i\; i+1) \in S_n$ for the simple reflections generating the symmetric group, and recall that $w_0 =[n,\dots,3,2,1]\in S_n$ is the longest permutation.
The \defn{Demazure product} on $S_n$ is the unique associative operation $\circ \colon S_n \times S_n \to S_n$ with
\[
w \circ s_i = \begin{cases} w & \text{if } w(i) > w(i+1) \\ w s_i & \text{if }w(i) < w(i+1) \end{cases}
\qquad\text{for all $w \in S_n$ and $i \in [n-1]$.}
\]
The monoid $(S_n, \circ)$ is called the \defn{$0$-Hecke monoid} of $S_n$ (see, e.g.,~\cite[\S29.3 Lemma A]{Humphreys75}).
A \defn{Hecke word} for $w \in S_n$ is a $0$-Hecke monoid representative of $w$, or more precisely, a finite sequence $(i_1, i_2, \dotsc, i_N)$ with each $i_j \in [n-1]$ such that $w = s_{i_1} \circ s_{i_2} \circ \cdots \circ s_{i_N}$.

There is an $S_n$-action on polynomials given by restricting the $S_n$-action on $\Lambda$.
Relative to this action, if  $f \in \ZZ[\beta][x_1,\dotsc,x_n]$ then $s_i f$ is the polynomial obtained by swapping $x_i$ and $x_{i+1}$.
The \defn{Lascoux operators} $\varpi_i \colon \ZZ[\beta][x_1,\dotsc,x_n]\to \ZZ[\beta][x_1,\dotsc,x_n]$, 
for $i \in [n-1]$ are defined by the formula
\be\label{eq:lascoux_ops}
\varpi_i f := \frac{ (x_i + \beta x_i x_{i+1}) f - (x_{i+1} + \beta x_i x_{i+1}) \, s_i   f }{x_i - x_{i+1}}.
\ee
The Lascoux operators satisfy the braid relations for the symmetric group and have $\varpi_i^2 =  \varpi_i$
\cite[Thm.~5.1]{Monical16},
 so for any reduced expression $w = s_{i_1} \cdots s_{i_{\ell}} \in S_n$ the operator $\varpi_w := \varpi_{i_1} \cdots \varpi_{i_{\ell}}$ is well-defined.
Given $w \in S_n$ and $\lambda \in P^+$, the corresponding \defn{Lascoux polynomial} is 
\be  L_{w\lambda}(x; \beta) := \varpi_w x^{\lambda} \in   \ZZ[\beta][x_1,\dotsc,x_n].\ee It is well-known 
that $L_{w\lambda}(x; \beta) = L_{w'\mu}(x;\beta)$ if and only if $w \lambda = w' \mu$.

Now, following \cite[\S13.2]{BumpSchilling}, for each $i \in I$ and $X \subseteq \cB \in \sqrt{\sC}$ we define the \defn{crystal Demazure operator} $\fkD_i$ by the formula
\begin{subequations}
\begin{align}
\fkD_i X & := \{ b \in \cB \mid e_i^k b\in X\text{ for some }k \geq 0\}
\\ & = \{ f_i^k x \mid k\geq 0 \text{ and } x \in X\} \setminus\{\zero\}.
\end{align}
\end{subequations}
We always have $\fkD_i \fkD_i X = \fkD_i X$, but the
$\fkD_i$ operators do not  generally satisfy any length-three braid relations 
when applied to general subsets of crystals $\cB \in \sqrt{\sC}_{\poly}$ (consider Example~\ref{ex:not_monoidal} with $X = \{ \{1,2\} \otimes \{1,2\}\}$).
However, there is an action of the $0$-Hecke monoid on $\SetTab_n(\lambda)$ when we start with the highest weight element.

\begin{theorem}[{\cite[Thm.~5.1]{Yu23}; see \cite{MT2023}}]
\label{thm:crystal_SVT}
Let $\lambda \in P^+$.
Then $\SetTab_n(\lambda)$ is a connected polynomial $\sqgln$-crystal. This crystal has a filtration indexed by permutations $w \in S_n$ with respect to Bruhat order given by the \defn{Demazure crystals}
defined by
$
\SetTab_n(\lambda)_w := \fkD_{i_1} \fkD_{i_2} \cdots \fkD_{i_{\ell}} \{ u_{\lambda} \}
$
where $(i_1,i_2,\dots,i_\ell)$ is any Hecke word  for $w$.
Moreover, the 
character of each Demazure crystal is the Lascoux polynomial
$
\ch \bigl( \SetTab_n(\lambda)_w \bigr) = L_{w\lambda}(x; \beta) .
$
\end{theorem}

\begin{remark}
A central part of the preceding theorem is the assertion that
$\SetTab_n(\lambda)_w$ depends only on $w$ and not on the choice of Hecke word.
While~\cite[Thm.~5.1]{Yu23} is stated in terms of reduced expressions for $w$, it is relatively easy to upgrade this to  Theorem~\ref{thm:crystal_SVT}.

Indeed, by basic properties of Coxeter groups~\cite[\S5.8]{Humphreys90}, for any $w \in S_n$ with $s_i \circ w = w$, there exists a reduced expression of the form $w = s_i s_{a_1} \cdots s_{a_{\ell}}  $.
Since $\fkD_i^2 = \fkD_i$, we have 
\[\fkD_i \SetTab_n(\lambda)_w = \fkD_i \fkD_i\fkD_{a_1}\cdots \fkD_{a_\ell} \{u_\lambda\} = \fkD_i\fkD_{a_1}\cdots \fkD_{a_\ell} \{u_\lambda\} = \SetTab_n(\lambda)_w,\]
using \cite[Thm.~5.1]{Yu23} for the first and last equalities.
Thus  $\fkD_i \circ \SetTab_n(\lambda)_w = \SetTab_n(\lambda)_{s_i\circ w}$
for any $w \in S_n$,
so for any Hecke word $(i_1, i_2,\dotsc, i_N)$ for $w$, it follows by induction on $N$ that
\[
\SetTab_n(\lambda)_w=
\fkD_{i_1} \fkD_{i_2} \cdots \fkD_{i_N}  \SetTab_n(\lambda)_1 = \fkD_{i_1}\fkD_{i_2}  \cdots \fkD_{i_N}  \{u_\lambda\}.
\]
\end{remark}

By construction, we have $\SetTab_n(\lambda) \in \sqrt{\cC}_{\poly} \subseteq \sK_+$.
Applying the functor $\defectF{2}$ to $\SetTab_n(\lambda)$ recovers the usual Kashiwara crystal structure on semistandard Young tableaux~\cite{KN94}. 

The  \defn{symmetric Grothendieck polynomial} $G_{\lambda}(x; \beta)$
of $\lambda \in P^+$ is the character 
\be\label{G-def1-eq}
G_{\lambda}(x; \beta) := \ch \bigl( \SetTab_n(\lambda) \bigr) = \sum_{T \in \SetTab_n(\lambda)} x^{\weight(T)} = L_{w_0 \lambda}(x; \beta).
\ee
Symmetric Grothendieck polynomials first appeared in work of Lascoux--Sch\"utzen\-berger~\cite{LS82,LS83} (with $\beta = -1$) and Fomin--Kirillov~\cite{FK1994} (with generic $\beta$).
The specializations $G_{\lambda}(x;-1)$ are representatives for the $K$-theory classes of Schubert varieties in the complex Grassmannian~\cite[Thm.~8.1]{Buch2002}, whereas the more familiar \defn{Schur functions} are given by $s_{\lambda}(x) = G_{\lambda}(x; 0)$.

Because $\SetTab_n(\lambda) \in \sqrt{\cC}_{\poly}$, 
we know that $G_{\lambda}(x; \beta) \in \ZZ[\beta][x_1,\dots,x_n]^{S_n}$.
Furthermore, 
the symmetric Grothendieck polynomials are a $\ZZ[\beta]$-basis of $\ZZ[\beta][x_1, \dotsc, x_n]^{S_n}$~\cite{Lenart00}.

Slightly generalizing \eqref{G-def1-eq}, for any weight of the form
\be\label{gen-weight}
k\e_0 + \lambda \in \Lambda \quad \text{with $ k\in \ZZ$ and $\lambda \in P^+$,}
\ee
we define
\be
G_{k\e_0+\lambda}(x; \beta) := \beta^{-k} G_{\lambda}(x; \beta).
\ee

\begin{example}
Fix a positive integer $k$, and let $\lambda = k \e_1 =  k\Lambda_1 \in P^+$ be the partition whose Young diagram is a single row of length $k$.
Then from the discussion in~\cite[\S6]{Buch2002} (see also~\cite[Rem.~3.7]{MPS21} and~\cite{Lenart00}), we have
$
\ch\bigl( \SetTab_n(\lambda) \bigr) = G_{\lambda}(x; \beta) = \sum_{m=0}^{n-1} \beta^m s_{k1^m}(x),
$
where $k1^m$ means the hook shaped partition $(k-1)\Lambda_1 + \Lambda_{m+1} = k\e_1 + \e_2 + \e_3 + \cdots + \e_{m+1}.$ 
\end{example}

\begin{example}
The tensor product $\SS_3 \otimes \SS_3$ has 3 connected components.
One connected component is $\SetTab_3(2\Lambda_1)$ (which is given by Figure~\ref{fig:S3square_cc}), and another is $\SetTab_3(\Lambda_2)$ (see Example~\ref{ex:column_not_subsets} below).
The remaining connected component is generated by $\{1,2\} \otimes \{1\}$, which is not isomorphic to $\cT_{-\e_0} \otimes \SetTab_3(\Lambda_1+\Lambda_2)$, even though they have the same character.
This decomposition reflects the expansion
$
G_{1}(x; \beta) \cdot G_{1}(x; \beta) 
= G_{2}(x; \beta) + G_{11}(x; \beta) + \beta G_{21}(x;\beta). 
$
\end{example}

\begin{example}
\label{ex:PS_nondem}
Another filtration of $\SetTab_n(\lambda)$ is considered in \cite{PS22}. 
The filtration of the $\sqgln$-crystals for rectangular shapes does not in general agree with the decomposition given in~\cite{PS22}, which can be seen at the level of $i$-strings (called \defn{$i$-K-strings} in~\cite{PS22}).
Although these decompositions do agree when  $\lambda = 2\Lambda_2$ and $n = 3$, as can be seen by comparing~\cite[Fig.~2]{PS22} and Figure~\ref{settab-fig},  they diverge when $n = 4$ since the $2$-string
\[
\begin{tikzpicture}[xscale=1.5,baseline=0]
\node (1) at (0,0) {$\ytableaushort{{1}{12},{2}{4}}$};
\node (2) at (2,0) {$\ytableaushort{{1}{\scalebox{.8}{123}},{2}{4}}$};
\node (3) at (4,0) {$\ytableaushort{{1}{13},{2}{4}}$};
\node (4) at (6,0) {$\ytableaushort{{1}{13},{23}{4}}$};
\node (5) at (8,0) {$\ytableaushort{{1}{13},{3}{4}}$};
\draw[->,thick,red] (1) -- node[midway,scale=0.75,above] {$2$} (2);
\draw[->,thick,red] (2) -- node[midway,scale=0.75,above] {$2$} (3);
\draw[->,thick,red] (3) -- node[midway,scale=0.75,above] {$2$} (4);
\draw[->,thick,red] (4) -- node[midway,scale=0.75,above] {$2$} (5);
\end{tikzpicture},
\]
does not match the $i$-K-string in~\cite{PS22}.
\end{example}

\begin{figure}[t]
\[
    \begin{tikzpicture}[xscale=2.35, yscale=1.75,>=latex,baseline=(z.base)]
    \node at (5,0.0) (z) {};
	\node at (3,2) (A1) {$\ytabc{0.45cm}{1 & 1 \\ 2 & 2}$};
	\node at (3,1) (A2) {$\ytabc{0.45cm}{1 & 1 \\ 2 & 23}$};
	\node at (3,0) (A3) {$\ytabc{0.45cm}{1 & 1 \\ 2 & 3}$};
	\node at (3,-1) (A4) {$\ytabc{0.45cm}{1 & 1 \\ 23 & 3}$};
	\node at (4,0) (B4) {$\ytabc{0.45cm}{1 & 12 \\ 2 & 3}$};
	\node at (3,-2) (A5) {$\ytabc{0.45cm}{1 & 1 \\ 3 & 3}$};
	\node at (4,-1) (B5) {$\ytabc{0.45cm}{1 & 12 \\ 23 & 3}$};
	\node at (5,0) (C5) {$\ytabc{0.45cm}{1 & 2 \\ 2 & 3}$};
	\node at (4,-2) (A6) {$\ytabc{0.45cm}{1 & 12 \\ 3 & 3}$};
	\node at (5,-1) (B6) {$\ytabc{0.45cm}{1 & 2 \\ 23 & 3}$};
	\node at (5,-2) (A7) {$\ytabc{0.45cm}{1 & 2 \\ 3 & 3}$};
	\node at (6,-2) (A8) {$\ytabc{0.45cm}{12 & 2 \\ 3 & 3}$};
	\node at (7,-2) (A9) {$\ytabc{0.45cm}{2 & 2 \\ 3 & 3}$};
      \draw[->,thick,red]  (A1) -- (A2) node[midway,right,scale=0.75] {$2$};
      \draw[->,thick,red]  (A2) -- (A3) node[midway,right,scale=0.75] {$2$};
      \draw[->,thick,red]  (A3) -- (A4) node[midway,right,scale=0.75] {$2$};
      \draw[->,thick,red]  (A4) -- (A5) node[midway,right,scale=0.75] {$2$};
      \draw[->,thick,red]  (B4) -- (B5) node[midway,right,scale=0.75] {$2$};
      \draw[->,thick,red]  (B5) -- (A6) node[midway,right,scale=0.75] {$2$};
      \draw[->,thick,red]  (C5) -- (B6) node[midway,right,scale=0.75] {$2$};
      \draw[->,thick,red]  (B6) -- (A7) node[midway,right,scale=0.75] {$2$};
      \draw[->,thick,blue]  (A3) -- (B4) node[midway,above,scale=0.75] {$1$};
      \draw[->,thick,blue]  (B4) -- (C5) node[midway,above,scale=0.75] {$1$};
      \draw[->,thick,blue]  (A4) -- (B5) node[midway,above,scale=0.75] {$1$};
      \draw[->,thick,blue]  (B5) -- (B6) node[midway,above,scale=0.75] {$1$};
      \draw[->,thick,blue]  (A5) -- (A6) node[midway,above,scale=0.75] {$1$};
      \draw[->,thick,blue]  (A6) -- (A7) node[midway,above,scale=0.75] {$1$};
      \draw[->,thick,blue]  (A7) -- (A8) node[midway,above,scale=0.75] {$1$};
      \draw[->,thick,blue]  (A8) -- (A9) node[midway,above,scale=0.75] {$1$};
     \end{tikzpicture}
\]
\caption{The crystal $\SetTab_3(2\Lambda_2)$.}
\label{settab-fig}
\end{figure}

In spite of all this, a result from \cite{MTY} shows the characters of objects in  $\sqrt{\sC}_{\poly}$ are still well behaved.
As our conventions are slightly different than in \cite{MTY}, the statement below does not exactly match the cited reference, but it is equivalent in view of the observations in Remark~\ref{no-loss-rem}.

\begin{theorem}[{\cite[Thm.~1.1]{MTY}}]
\label{thm:standard_characters}
Let $\cB \in \sqrt{\sC}_{\poly}$ be a polynomial $\sqgln$-crystal with highest weight elements of weight $\lambda^{(1)}, \dotsc, \lambda^{(m)}$.
Then each $\lambda^{(i)}$ has the form \eqref{gen-weight} with $k\leq 0$ and
\[
\ch(\cB) =  G_{\lambda^{(1)}}(x; \beta) + \cdots + G_{\lambda^{(m)}}(x; \beta) \in \NN[\beta]\spanning\Bigl\{ G_{\mu}(x; \beta)  \mid \mu \in P^+\Bigr\}.
\]
\end{theorem}

Theorem~\ref{thm:standard_characters} generalizes a similar property of the characters of polynomial Kashiwara crystals, which are known to be sums of Schur functions in type $\gl_n$~\cite{KN94}.
One important difference between these two statements is that in the Kashiwara case, one can deduce Schur positivity of characters from quite rigid properties of the relevant crystals (like connected objects having unique highest weight elements), which no longer hold for polynomial $\sqgln$-crystals.
The examples below illustrate a few of these pathologies.

\begin{example}
\label{ex:weight_not_isomorphism}
Consider the highest weight element $b = \{1\} \otimes \{2\} \otimes \{1\} \otimes \{1\} \in \SS_3^{\otimes 4}$.
Then $b' = f_1^2 f_2^6 f_1^4 b$ is a lowest weight element, but
$
b'' = e_2^2 e_1^4 e_2^4 e_1 b' = \{1,2\} \otimes \{2\} \otimes \{1\} \otimes \{1\}
$
is another highest weight element.
One can check that $b$ and $b''$ have the respective weights 
\[
3\e_1 + \e_2 = \Lambda_2 + 2\Lambda_1
\qquand
-\e_0 + 3\e_1 + 2\e_2=-\e_0+ 2\Lambda_2 + \Lambda_1.
\]
On the other hand, a direct computation shows that the character of the connected component containing $b$ is $G_{\Lambda_2 + 2\Lambda_1}(x; \beta) + \beta G_{2\Lambda_2 + \Lambda_1}(x; \beta)$, as claimed by Theorem~\ref{thm:standard_characters}.
\end{example}

\begin{example}
Assume $n\geq 3$. Then  
$
b = \{2\} \otimes \{1\} \otimes \{1\} \otimes \{1\}
$
and
$
b' = \{1\} \otimes \{1\} \otimes \{2\} \otimes \{1\}
$
are the unique highest elements of two distinct connected components of $\SS_n^{\otimes 4}$.
A direct computation shows that these components both have the character of $G_{\Lambda_2+2\Lambda_1}(x; \beta)$ but are not isomorphic crystals.
One way to see this non-isomorphism is to check that
\begin{gather*}
f_2 f_1 b = \{2\} \otimes \{1\} \otimes \{1\} \otimes \{1,2,3\}
\neq 
f_1 f_2 b = \{2,3\} \otimes \{1\} \otimes \{1\} \otimes \{1,2\},
\end{gather*}
whereas
$
f_1 f_2 b' = f_2 f_1 b' = \{1\} \otimes \{1,2\} \otimes \{2,3\} \otimes \{1\}.
$
\end{example}

Recall that we have identified $\Lambda_i$ with a column of height $i$.
In view of the definition of the column reading word and Theorem~\ref{thm:crystal_SVT}, we can consider a semistandard set-valued tableau as a tensor product of single-column tableau crystals of the form $\SetTab_n(\Lambda_i)$.

\begin{corollary}
\label{cor:column_description}
Let $\lambda = \sum_{i=1}^n m'_i \Lambda_i \in P^+$.
Then $R_{\col}$ is the unique crystal embedding
\[
\SetTab_n(\lambda) \hookrightarrow \SetTab_n(\Lambda_n)^{\otimes m'_n} \otimes \SetTab_n(\Lambda_{n-1})^{\otimes m'_{n-1}} \otimes \cdots \otimes \SetTab_n(\Lambda_1)^{\otimes m'_1}
\]
that maps $u_{\lambda} \mapsto u_{\Lambda_n}^{\otimes m'_n} \otimes \cdots \otimes u_{\Lambda_1}^{\otimes m'_1}$.
\end{corollary}

We can be more explicit about the embedding in Corollary~\ref{cor:column_description}: it is given by splitting the columns in a semistandard set-valued tableau $T$.
For example, we have
\[
\ytableaushort{ABCD,EFG,H}
\quad \mapsto \quad
\ytableaushort{A,E,H} \otimes \ytableaushort{B,F} \otimes \ytableaushort{CG} \otimes \ytableaushort{D}.
\]

Semistandard Young tableaux whose shape is a single column are in bijection with finite nonempty subsets of positive integers, and therefore they are uniquely identified by their weights.
These properties no longer hold in the set-valued case.
More precisely, we cannot identify elements in $\SetTab_n(\Lambda_i)$ for $i > 1$ with   subsets of $[n]$ of size at least $i$,  as indicated by the following example.

\begin{example}
\label{ex:column_not_subsets}
The crystal $\SetTab_3(\Lambda_2)$
\[
\ytableausetup{boxsize=1.3em}
\begin{tikzpicture}[>=latex,xscale=2,baseline=0]
\node (12) at (0,0) {$\ytableaushort{1,2}$};
\node (13) at (2,0) {$\ytableaushort{1,3}$};
\node (23) at (4,0) {$\ytableaushort{2,3}$};
\node (123b) at (3,0) {$\ytableaushort{{12},3}$};
\node (123a) at (1,0) {$\ytableaushort{1,{23}}$};
\draw[->,thick,color=red] (12) -- node[midway,above,scale=.7] {2} (123a);
\draw[->,thick,color=red]  (123a) -- node[midway,above,scale=.7] {2} (13);
\draw[->,thick,color=blue] (13) -- node[midway,above,scale=.7] {1} (123b);
\draw[->,thick,color=blue](123b) -- node[midway,above,scale=.7] {1} (23);
\end{tikzpicture}
\]
has two elements with weight $-\e_0+\e_1+\e_2+\e_3$ that would correspond to the set $\{1,2,3\}$.
\end{example}

The preceding result has an analog using the \defn{row reading word} $\R_{\row} \colon \SetTab_n(\lambda) \to \SS_n^{\otimes \abs{\lambda}}$, which is defined by reading 
the rows of a set-valued tableau
 left-to-right, but iterating over the rows bottom-to-top.
For the tableau in Example~\ref{ex:column_word}, the row reading word maps to $E \otimes C \otimes D \otimes A \otimes B$.
We will show that the crystal structure on set-valued tableaux induced from the row reading word is exactly the same as from the column reading word.
More generally, we define:

\begin{definition}
Let $\D \subset \ZZ_{>0} \times \ZZ_{>0}$ be a finite set with size $k := \abs{\D} < \infty$.
An \defn{admissible reading order} is a bijection $\psi \colon \D \to [k]$ such that $\psi(i, j) \leq \psi(i', j')$ whenever $i \geq i'$ and $j \leq j'$.
Let $\cS$ be some subset of set-valued tableaux of shape $\D$.
The \defn{admissible reading} of $\cS$ associated to the admissible reading order $\psi$ is the map $R_{\psi} \colon \cS \to \SS_n^{\otimes k}$  given by
\[
R_{\psi}(T) = T\bigl( \psi^{-1}(1) \bigr) \otimes T\bigl( \psi^{-1}(2) \bigr) \otimes \cdots \otimes T\bigl( \psi^{-1}(k) \bigr).
\]
\end{definition}

Equivalently, an admissible reading of a set-valued tableau of shape $\D$ is such that for any two entries $A$ and $B$ with $A$ weakly southwest of $B$ in $\D$, it must hold that $A$ comes before $B$ in the reading word (i.e., we have $\cdots \otimes A \otimes \cdots \otimes B \otimes \cdots$).
It is easy to see that the row and column reading words are both admissible readings.

\begin{lemma}
\label{lemma:admissible_reading}
Every admissible reading is an embedding of $\sqgln$-crystals $\SetTab_n(\lambda)\to \SS_n^{\otimes \abs{\lambda}}$.
\end{lemma}

\begin{proof}
The result follows by the same proof as~\cite[Thm.~7.3.6]{HK02} or~\cite[Thm.~4.4]{BKK00}.
\end{proof}

Lemma~\ref{lemma:admissible_reading} means that we could replace the column reading word in our definition of the crystal structure for $\SetTab_n(\lambda)$ with any admissible reading and we would obtain the \emph{same} $\sqgln$-crystal.
The following analog of Corollary~\ref{cor:column_description} 
is a special case of this general fact.

\begin{corollary}
\label{cor:row_description}
Let $\lambda =\sum_{i=1}^n \lambda_i\e_i \in P^+$.
Then $R_{\row} $ is the unique embedding
\[
\SetTab_n(\lambda) \hookrightarrow \SetTab_n(\lambda_n \Lambda_1) \otimes \SetTab_n(\lambda_{n-1} \Lambda_1) \otimes \cdots \otimes \SetTab_n(\lambda_1 \Lambda_1)
\]
that maps
$
u_{\lambda} \mapsto (f_{n-1}^{2\lambda_n} \cdots f_1^{2\lambda_n} u_{\lambda_n \Lambda_1}) \otimes \cdots \otimes (f_1^{2\lambda_2} u_{\lambda_2 \Lambda_1}) \otimes u_{\lambda_1 \Lambda_1}.
$
\end{corollary}

Note that $u_k^{(i)} := f_{i-1}^{2k} \cdots f_1^{2k} u_{k\Lambda_1}$ is a single row of length $k$ with all entries $\{i\}$, which is easy to see by induction.
Another way to realize this is to note that the $i$-th row is equal to the image of the restriction functor $\restrictF{[i,n]}\bigl(\SetTab_n(\lambda_i \Lambda_1)\bigr)$, with $u_{\lambda_i}^{(i)}$ being its unique highest weight element.
In fact, we have a stronger statement.

\begin{corollary}
\label{cor:row_restriction}
The embedding from Corollary~\ref{cor:row_description} is the embedding
\[
\SetTab_n(\lambda) \hookrightarrow \restrictF{[n,n]}\bigl( \SetTab_n(\lambda_n \Lambda_1) \bigr) \otimes \cdots \otimes \restrictF{[2,n]}\bigl( \SetTab_n(\lambda_2 \Lambda_1) \bigr) \otimes \restrictF{[1,n]}\bigl( \SetTab_n(\lambda_1 \Lambda_1) \bigr).
\]
\end{corollary}

Subsequently, we can also define the mapping in Corollary~\ref{cor:row_description} explicitly by splitting any semistandard set-valued tableau along rows.
For example, we have
\[
\ytableaushort{ABCD,EFG,H}
\quad \mapsto \quad
\ytableaushort{H} \otimes \ytableaushort{EFG} \otimes \ytableaushort{ABCD}.
\]

\subsection{Direct limits}

Our goal here is to construct a {directed system} of $\sqgln$-crystal embeddings
\be\label{Psi-eq}
\{\Psi_{\lambda,\mu} \colon \cT_{-\lambda} \otimes \SetTab_n(\lambda) \hookrightarrow \cT_{-\lambda-\mu} \otimes \SetTab_n(\lambda+\mu) \mid \mu, \lambda \in P^+\}
\ee
such that the diagram
\be
\label{eq:directed_system_diagram}
\begin{tikzpicture}[>=latex,baseline=-30]
\node (1) at (0, 0) {$\cT_{-\lambda} \otimes \SetTab_n(\lambda)$};
\node (2) at (6, 0) {$\cT_{-\lambda-\mu} \otimes \SetTab_n(\lambda+\mu)$};
\node (3) at (6, -2) {$\cT_{-\lambda-\mu-\nu} \otimes \SetTab_n(\lambda+\mu+\nu)$};
\draw[->] (1) -- node[midway,above] {$\Psi_{\lambda,\mu}$} (2);
\draw[->] (2) -- node[midway,right] {$\Psi_{\lambda+\mu,\nu}$} (3);
\draw[->] (1) -- node[midway,below left] {$\Psi_{\lambda,\mu+\nu}$} (3);
\end{tikzpicture}
\ee
commutes for all partitions $\lambda,\mu,\nu \in P^+$.
The direct limit of this system will yield a  $\sqgln$-crystal
that we denote as $\SetTab_n(\infty)$.

Since each crystal $\cT_{-\lambda} \otimes \SetTab_n(\lambda)$ is lower generated by a unique highest weight element and upper regular, any candidate embedding $\Psi_{\lambda,\mu}$ must send $t_{-\lambda} \otimes u_{\lambda} \mapsto t_{-\lambda-\mu} \otimes u_{\lambda+\mu}$,
and 
to show $\Psi_{\lambda,\mu}$ is a crystal embedding, it will be sufficient to show it commutes with $e_i$. We give an explicit formula for the map $\Psi_{\lambda,\mu}$ in~\eqref{eq:insertion_map} below.
We will show our result by building up from the case when $\lambda$ and $\mu$ are both single rows (i.e., length $1$).

Suppose $\lambda,\mu \in P^+$ are partitions.
Given a set-valued tableau $T$ of shape $\D_\lambda$, let $\insmap_{\mu}(T)$ denote the set-valued tableau of shape $\D_{\lambda+\mu}$ with value $\{i\}$ in each box $(i,j) \in \D_\mu$ and value $T_{i,j-\mu_i}$ in each box $(i,j) \in \D_{\lambda+\mu}\setminus\D_\mu$.
For example, if $\mu = 3\e_1 + 3\e_2 + 2\e_3$ then 
\[ \insmap_\mu\(\ \ytabc{0.55cm}{1 & 125&5 & 5 & 57 \\
25& 6 & 67 \\
6 & 8 \\
78 & 9 }\ \)
= \ytableaushort{1111{125}55{57},222{25}6{67},3368,{78}9}*[*(blue!30)]{3,3,2}.
\]
We omit the proof of the following lemma, which can be derived as a straightforward exercise.

\begin{lemma} \label{iota-lem}
Let $\lambda,\mu\in P^+$.
Then $\insmap_{\mu} $ is an injective map of sets $\SetTab_n(\lambda) \to \SetTab_n(\lambda+\mu)$, which is a bijection if and only if $\mu$ is a scalar multiple of $\Lambda_n=\e_1+\e_2+\cdots+\e_n$.
\end{lemma}

We now claim that the maps $\Psi_{\lambda,\mu}$ can be explicitly described by
\be\label{eq:insertion_map}
\Psi_{\lambda,\mu}(t_{-\lambda} \otimes T)  := t_{-\lambda-\mu} \otimes \insmap_\mu(T)
\qquad\text{for }T \in \SetTab_n(\lambda).
\ee
From Lemma~\ref{iota-lem}, we know that $\Psi_{\lambda,\mu}$ is an injective map of sets, but it remains to show this map is a crystal morphism.

In view of Theorem~\ref{thm:crystal_SVT} and Lemma~\ref{lemma:T_upper_prod},
  $\cT_{-\lambda} \otimes \SetTab_n(\lambda)$ is a connected upper regular crystal that is lower generated by $\{t_{-\lambda}\otimes u_{\lambda}\}$. Thus, by Lemma~\ref{lemma:lower_embed_check}, it is sufficient to show $\Psi_{\lambda,\mu}$ commutes with all crystal operators $e_i$.
We first check this directly when $\lambda$ and $\mu$ have length at most one.

\begin{lemma}
\label{lemma:single_row_embedding}
Fix $k, m \in \NN$. 
The map
\[
\Psi_{k\Lambda_1,m\Lambda_1} \colon \cT_{-k\Lambda_1} \otimes  \SetTab_n(k\Lambda_1)  \to \cT_{-(k+m)\Lambda_1} \otimes \SetTab_n((k+m) \Lambda_1)
\]
defined by~\eqref{eq:insertion_map} is a $\sqgln$-crystal embedding.
\end{lemma}

\begin{proof}
As noted above,
the crystal $\cT_{-k\Lambda_1} \otimes \SetTab_n(k\Lambda_1)$ is upper regular and lower generated by $\{t_{-k\Lambda_1} \otimes u_{k\Lambda_1}\}$.
We know that $\Psi_{k\Lambda_1,m\Lambda_1}$ commutes with the $e_i$ operators from the tensor product rule.
Indeed, we can identify $\insmap_{m\Lambda_1}(T) = u_{m\Lambda_1} \otimes T$ by the embedding into $\SS_n^{\otimes (k+m)}$, and then
\[
e_i(u_{m\Lambda_1} \otimes T) = u_{m\Lambda_1} \otimes (e_i T)
\qquad\text{for all $i \in [n-1]$ and $T \in \SetTab_n(k\Lambda_1)$}
\]
as $\varepsilon_i(u_{m\Lambda_1}) = 0 \leq \varphi_i(T)$ because $\SetTab_n(k\Lambda_1)$ is a lower regular crystal.
Therefore, the claim follows from Lemma~\ref{lemma:lower_embed_check}.
\end{proof}

\begin{theorem}
\label{thm:directed_system}
The maps $\{\Psi_{\lambda,\mu} \mid \lambda,\mu \in P^+\}$ defined by~\eqref{eq:insertion_map} are $\sqgln$-crystal embeddings and form a directed system; that is, the diagram \eqref{eq:directed_system_diagram} commutes.
\end{theorem}

\begin{proof}
It is clear that the directed system diagram~\eqref{eq:directed_system_diagram} commutes as functions on sets, and so it remains to show that $\Psi_{\lambda,\mu}$ is a crystal morphism.
By Corollary~\ref{cor:row_restriction}, we can consider any semistandard set-valued tableau as a tensor product of rows with the $j$-th row given by $\restrictF{[j,n]}\bigl(\SetTab_n(\lambda_j \Lambda_1) \bigr)$.
By the $j$-shifted version of Lemma~\ref{lemma:single_row_embedding}, the crystal operators $e_i$ commute with inserting any $\{j\}$ into row $j$,
as these insertions leave unchanged the comparisons determining which row $e_i$ acts on in the tensor product. This observation yields the claim.
\end{proof}

\begin{remark}
We could also prove this theorem using the signature rule mentioned in Remark~\ref{rem:sigrule}.
In the terminology of \cite{Yu23}, the $i$-signature of $\insmap_\mu(T)$ only differs from that of $T$ by prepending certain \defn{null forms} and \defn{right forms}.
This observation makes it possible to deduce (e.g., using the summary of \cite{Yu23} in~\cite[\S2.2]{MTY}) that $\Psi_{\lambda,\mu}$ is a crystal morphism as no cancellations are introduced.
\end{remark}

Denote the direct limit of the system in Theorem~\ref{thm:directed_system} by
\begin{equation}
\label{eq:direct_limit_def}
\SetTab_n(\infty) := \varinjlim_{\lambda \in P^+} \Bigl(\cT_{-\lambda} \otimes \SetTab_n(\lambda)\Bigr).
\end{equation}
This is a $\sqgln$-crystal whose elements are equivalence classes on
\[
\bigsqcup_{\lambda \in P^+} \cT_{-\lambda} \otimes \SetTab_n(\lambda)
\]
given by the equivalence relation $t_{-\lambda} \otimes T \sim t_{-\mu} \otimes T'$ if and only if there exists a $\nu \in P^+$ such that 
\[
\nu-\lambda \in P^+,\qquad \nu-\mu \in P^+,\qquand
\Psi_{\lambda,\nu-\lambda}(t_{-\lambda} \otimes T) = \Psi_{\mu,\nu-\mu}(t_{-\mu} \otimes T').\] For our directed system it is sufficient to take $\nu = \lambda + \mu$.
Therefore, by our construction, this holds if and only if $\insmap_{\mu}(T) = \insmap_{\lambda}(T')$.
For simplicity, we write equivalence classes as $[T] := [t_{-\lambda} \otimes T]$ as $\lambda$ is the shape of $T$.

In terms of this notation, the $\sqgln$-crystal structure on $\SetTab_n(\infty)$ has the following explicit description.
Assume $T \in\SetTab_n(\lambda)$.
Then the weight map is
\be
\weight([T]) = \weight(T)-\lambda \in \ZZ\oplus \ZZ^{n}
\ee
and for each $i\in [n-1]$, we have
\begin{subequations}
\label{raising-limit-eq0}
\begin{align}
e_i[T] &= \begin{cases} 
[e_iU]&\text{if there exists some $t_{-\mu}\otimes U \sim t_{-\lambda}\otimes T$ with $e_i U\neq \zero$}, \\
\zero&\text{otherwise},
\end{cases}
\\
f_i[T] &= \begin{cases} 
[f_iU]&\text{if there exists some 
$t_{-\mu}\otimes U \sim t_{-\lambda}\otimes T$ with $f_i U\neq \zero$}, \\
\zero&\text{otherwise}.
\end{cases}
\end{align}
\end{subequations}
Since the crystals are upper regular, the crystal operator formulas just given simplify to
\begin{subequations}
\label{raising-limit-eq}
\begin{align}
e_i[T] &= [e_i T] \quad\text{for all $T \in \SetTab_n(\lambda)$, setting } [\zero] = \zero,
\\
f_i[T] &= [f_i T] \quad \text{for those $T \in \SetTab_n(\lambda)$ with } f_i T \neq \zero.
\end{align}
\end{subequations}
Finally, we have the statistics
\be\label{string-Tcol-eq}
\varepsilon_i([T]) = \varepsilon_i(T)
\qquand
\varphi_i([T]) = \varphi_i(T) +\lambda_{i+1}-\lambda_i.
\ee

The $\sqgln$-crystal $\SetTab_n(\infty)$ is an analog of the $\gl_n$-crystal $\cB(\infty)$ corresponding to the lower half of the quantum group $U_q^-(\gl_n)$~\cite{Kashiwara90,Kashiwara91}.
More directly, for type $\gl_n$, we may define
\be\label{cB-SetTab-eq}
\cB(\infty) = \defectF{2}\bigl(\SetTab_n(\infty)\bigr),
\ee
where $\defectF{2}$ is the functor from Proposition~\ref{defectF-prop}.
It is easy to see that $\defectF{2}\bigl(\SetTab_n(\infty)\bigr)$ is isomorphic to the direct limit of the system of $\gl_n$-crystals obtained by applying $\defectF{2}$ to \eqref{Psi-eq}, and this limit matches the construction of $\cB(\infty)$
 in \cite{Kashiwara90,Kashiwara91}
 by results of Cliff~\cite{Cliff98} and Hong--Lee~\cite{HL08} (cf.~\cite[Exs.~12.1 and 12.4]{BumpSchilling}).
The weight shifted crystals $\cT_{\lambda} \otimes \SetTab_n(\infty)$ are therefore variants of the $\gl_n$-crystal bases encoding Verma modules of highest weight $\lambda$.

The unique highest weight element of $\SetTab_n(\infty)$, which has weight $0$, is given by the class
\be\label{unifty-eq}
u_{\infty} := \{ t_{-\lambda} \otimes u_{\lambda} \mid \lambda \in P^+\}
\ee
(under the relation $\sim$).
Indeed, $u_\infty = [u_\lambda]$ for all $\lambda \in P^+$
so we have $\varepsilon_i(u_\infty) = \varphi_i(u_\infty) = 0$ for all $ i \in [n-1]$.
Furthermore, $\SetTab_n(\infty)$ is lower generated by $\{u_{\infty}\}$.
In fact, we can extend the Demazure filtration of Theorem~\ref{thm:crystal_SVT} to the case $\lambda = \infty$.

\begin{theorem}\label{dem-thm}
The Demazure crystal defined by 
\[
\SetTab_n(\infty)_w := \fkD_{i_1} \cdots \fkD_{i_N} \{ u_{\infty} \}
\qquad \text{ for any } w = s_{i_1} \circ \cdots \circ s_{i_N} \in S_n
\]
does not depend on the choice of Hecke word for $w \in S_n$. 
These subcrystals
give a filtration of $\SetTab_n(\infty)$ indexed by $S_n$ with respect to Bruhat order.
Consequently, $\SetTab_n(\infty) = \SetTab_n(\infty)_{w_0}$ is connected with a unique highest weight element $u_{\infty}$.
\end{theorem}

We will discuss the characters of $\SetTab_n(\lambda)$ and $\SetTab_n(\lambda)_w$  in Section~\ref{vec-param-sect}.

\begin{proof}
This follows from Theorem~\ref{thm:crystal_SVT} and the direct limit construction.
In more detail, we have
\be\label{dem-union-eq}
  \fkD_{i_1}\fkD_{i_2}\cdots \fkD_{i_N} \{ u_{\infty}\}
=\bigcup_{\lambda \in P^+} \Bigl\{ [T] \mid T \in \SetTab_n(\lambda)_w = \fkD_{i_1}\fkD_{i_2}\cdots \fkD_{i_N} \{ u_{\lambda}\}\Bigr\}.
\ee
(Note that in the union, by minor abuse of notation, we are taking $\fkD_i$ relative to $\SetTab_n(\lambda)$.)
The sets in the union do not depend on the choice of Hecke word for $w$, since they are each equal to $\{ [T] \mid T \in \SetTab_n(\lambda)_w\}$.
Therefore neither does $\fkD_{i_1}\fkD_{i_2}\cdots \fkD_{i_N} \{ u_{\infty}\}$, as claimed.
Furthermore, the fact $\SetTab_n(\lambda)_{w_0}= \SetTab_n(\lambda)$ for all $\lambda \in P^+$ yields $\SetTab_n(\infty)_{w_0} =  \SetTab_n(\infty)$ by~\eqref{dem-union-eq}.
\end{proof}

Since $\SetTab_n(\lambda)$ is a regular crystal, the following is immediate from \eqref{raising-limit-eq} and \eqref{string-Tcol-eq}.

\begin{corollary}
\label{cor:inf_upper_reg}
The crystal $\SetTab_n(\infty)$ is upper regular.
\end{corollary}

One can recover $\SetTab_n(\lambda)$ from $\SetTab_n(\infty)$ by tensoring with $\cR_{\lambda}$ from Example~\ref{ex:Rla_crystal}.
An analogous property holds
 for classical $\gl_n$-crystals using $\cB(\infty)$~\cite[Thm.~5]{Kashiwara91} (see also~\cite{Joseph95}).

\begin{corollary}
\label{cor:finite_recovery}
Fix $\lambda \in P^+$.
Then the map   $T \mapsto r_{\lambda} \otimes [T]$ for $T \in \SetTab_n(\lambda)$ is a $\sqgln$-crystal isomorphism
from $\SetTab_n(\lambda)$ to 
the connected component of $r_{\lambda} \otimes u_{\infty}$
in $\cR_{\lambda} \otimes \SetTab_n(\infty)$.
\end{corollary}

\begin{proof}
One can check directly that the given map is an isomorphism using \eqref{raising-limit-eq} and \eqref{string-Tcol-eq}.
We omit the details as the argument is the same as in the proof of~\cite[Cor.~5.3.13]{Joseph95}.
\end{proof}

We also extend~\cite[Lemma~4.19]{Yu23} (the $\sqgln$-analog of~\cite[Prop.~3.3.5]{Kashiwara93}) 
to the case when $\lambda = \infty$ (and hence the $\sqgln$-analog of~\cite[Prop.~3.3.4]{Kashiwara93}). 
This will be useful in our character computations below.

\begin{corollary}
\label{cor:string_decomp}
Let $\lambda \in P^+ \sqcup \{\infty\}$ and $i \in [n-1]$.
Fix some $w \in S_n$.
Suppose $\cS \subseteq \SetTab_n(\lambda)$ is an $i$-string.
Then the intersection $\cS \cap \SetTab_n(\lambda)_w$ is either empty, all of $\cS$, or just the highest weight element of $\cS$.
\end{corollary}

\begin{proof}
As mentioned above, this is~\cite[Lemma~4.19]{Yu23} when $\lambda \neq \infty$.
The case when $\lambda = \infty$ follows from the direct limit construction.
\end{proof}

To conclude this section, we introduce a set-valued variation of the \defn{marginally large tableau} model for $\cB(\infty)$ (in type $\mathfrak{gl}_n$) from~\cite{Cliff98,HL08}.
This gives another way to describe the elements of $\SetTab_n(\infty)$.

\begin{definition}[Marginally large tableau]
\label{def:marginally_large}
Fix a partition $\lambda = \sum_{i=1}^n \lambda_i \e_i \in P^+$.
A semistandard set-valued tableau $T$ of shape $\lambda$ is \defn{marginally large} if $\lambda_n = 0$ and we have
\[
\abs{ \{j \mid  T_{ij} = \{i\} \} } = \lambda_{i+1} + 1
\qquad \text{ for each $i \in [n-1]$}.
\]
\end{definition}

In other words, $T$ is marginally large if it has exactly $n-1$ rows and the number of boxes whose entry is $\{i\}$ in row $i$ is exactly $\lambda_{i+1} + 1$ for each $i \in [n-1]$ with $\lambda_n = 0$.
In any such tableau, the first column must be the sets $\{1\}, \{2\}, \dotsc, \{n-1\}$, in that order from top to bottom.

\begin{example}
When $n=8$,  the semistandard set-valued tableau
\[
\ytableausetup{boxsize=0.5cm}
\ytableaushort{111111111111111{\scalebox{.8}{124}}55{57},222222222335{\scalebox{.8}{678}}8,33333333,44444{48}8,5555,666,78} *[*(blue!30)]{15,9,8,5,4,3,1}
\]
is marginally large; here we have shaded all boxes with entry $\{i\}$ in row $i$.
\end{example}

We omit the proof of the following result, whose first part is a straightforward combinatorial exercise and whose second claim is easy to deduce from the tensor product rule or Yu's signature rules (see \cite[\S2.2]{MTY}).

\begin{proposition}\label{mala-lem}
Each element of $\SetTab_n(\infty)$ may be expressed as $[T]$ for a unique marginally large semistandard set-valued tableau $T$, and for this tableau, it holds that $f_iT \neq \zero$ for all $i \in I$.
\end{proposition}

For example, we have $u_{\infty} = [T]$, where $T$ is the marginally large tableau of staircase shape $\lambda = \Lambda_1 + \Lambda_2 + \cdots + \Lambda_{n-1} = \sum_{i=1}^{n-1} (n-i) \e_i$ with all entries in row $i$ given by $\{i\}$ (that is, it corresponds to the highest weight element in $\SetTab_n(\lambda)$).

In view of Proposition~\ref{mala-lem}, we can specify the crystal structure on $\SetTab_n(\infty)$ by describing how the crystal operators act directly  on marginally large tableaux.
In this model,  we first have $e_i$ and $ f_i$ act by the usual formulas in \eqref{raising-limit-eq},
and then we change the result to the unique marginally large tableau in the relevant equivalence class.
(In effect, this last step simply adds or removes at most one column of some height $h$, consisting of entries $\{1\}, \dotsc, \{h\}$,  in an appropriate location.)
The exact formulation is analogous to the $\gl_n$ case (see, e.g.,~\cite{HL08}) as when $T$ is marginally large we have  $f_i T \neq \zero$ for all $i \in [n-1]$.
See Figure~\ref{fig:gl3mlt} for an example.

\begin{figure}[t]
\[
\ytableausetup{boxsize=0.5cm}
    \begin{tikzpicture}[xscale=2.1, yscale=2.3,>=latex,every node/.style={scale=0.7},baseline=0]
\node at (0,0) (A1) {$\ytableaushort{11,2} *[*(blue!30)]{2,1}$};
\node at (-1,-1) (A2) {$\ytableaushort{111,2{23}} *[*(blue!30)]{3,1}$};
\node at (0,-1) (B2) {$\ytableaushort{11{13},2} *[*(blue!30)]{2,1}$};
\node at (1,-1) (C2) {$\ytableaushort{11{12},2} *[*(blue!30)]{2,1}$};
\node at (-2,-2) (A3)  {$\ytableaushort{111,23} *[*(blue!30)]{3,1}$};
\node at (-1,-2) (B3) {$\ytableaushort{111{12},2{23}} *[*(blue!30)]{3,1}$};
\node at (0,-2) (C3) {$\ytableaushort{111{13},2{23}} *[*(blue!30)]{3,1}$};
\node at (1,-2) (D3) {$\ytableaushort{11{\scalebox{.8}{123}},2} *[*(blue!30)]{2,1}$};
\node at (2,-2) (E3) {$\ytableaushort{112,2} *[*(blue!30)]{2,1}$};
\node at (-2,-3) (A4) {$\ytableaushort{111{12},23} *[*(blue!30)]{3,1}$};
\node at (-3,-3) (B4) {$\ytableaushort{1111,2{23}3} *[*(blue!30)]{4,1}$};
\node at (-1,-3) (C4) {$\ytableaushort{1112,2{23}} *[*(blue!30)]{3,1}$};
\node at (0,-3) (D4) {$\ytableaushort{111{\scalebox{.8}{123}},2{23}} *[*(blue!30)]{3,1}$};
\node at (1,-3) (H4) {$\ytableaushort{111{13},23} *[*(blue!30)]{3,1}$};  
\node at (2,-3) (E4) {$\ytableaushort{11{23},2} *[*(blue!30)]{2,1}$};
\node at (3,-3) (F4) {$\ytableaushort{11{12}2,2} *[*(blue!30)]{2,1}$};
\node at (-3.8,-3.8) (C5) {$\vdots$};
\node at (-3,-3.8) (B5) {$\vdots$};
\node at (-2,-3.8) (A5) {$\vdots$};
\node at (-1.5,-3.8) (D5) {$\vdots$};
\node at (-0.5,-3.8) (E5) {$\vdots$};
\node at (0.5,-3.8) (F5) {$\vdots$};
\node at (1.5,-3.8) (J5) {$\vdots$};  
\node at (2,-3.8) (G5) {$\vdots$};
\node at (3,-3.8) (H5) {$\vdots$};
\node at (3.8,-3.8) (I5) {$\vdots$};
\draw[->,thick,red]  (A1) -- (A2) node[midway,above left,scale=0.75] {$2$};
\draw[->,thick,blue]  (A1) -- (C2) node[midway,above right,scale=0.75] {$1$};
\draw[->,thick,red]  (A2) -- (A3) node[midway,above left,scale=0.75] {$2$};
\draw[->,thick,blue]  (A2) -- (B3) node[midway,left,scale=0.75] {$1$};
\draw[->,thick,red]  (B2) -- (C3) node[midway,right,scale=0.75] {$2$};
\draw[->,thick,blue]  ([xshift=4pt]B2.south) -- ([xshift=-8pt]D3.north) node[midway,below left,scale=0.75] {$1$};
\draw[<-,thick,red]  ([xshift=8pt]B2.south) -- ([xshift=-4pt]D3.north) node[midway,above right,scale=0.75] {$2$};
\draw[->,thick,red]  (C2) -- (D3) node[midway,right,scale=0.75] {$2$};
\draw[->,thick,blue]  (C2) -- (E3) node[midway,above right,scale=0.75] {$1$};
\draw[->,thick,blue]  (A3) -- (A4) node[midway,above left,scale=0.75] {$1$};
\draw[->,thick,red]  (A3) -- (B4) node[midway,above left,scale=0.75] {$2$};
\draw[->,thick,blue]  (B3) -- (C4) node[midway,left,scale=0.75] {$1$};
\draw[->,thick,red]  (B3) -- (D4) node[midway,above right,scale=0.75] {$2$};
\draw[->,thick,blue]  (C3) -- (D4) node[midway,left,scale=0.75] {$1$};
\draw[->,thick,blue]  (D3) -- (E4) node[midway,above right,scale=0.75] {$1$};
\draw[->,thick,red]  (E3) -- (E4) node[midway,above left,scale=0.75] {$2$};
\draw[->,thick,blue]  (E3) -- (F4) node[midway,above right,scale=0.75] {$1$};
\draw[->,thick,red]  (C3) -- (H4) node[midway,above right,scale=0.75] {$2$};  
\draw[->,thick,blue]  (A4) -- (A5) node[midway,above left,scale=0.75] {$1$};
\draw[->,thick,blue]  (B4) -- (B5) node[midway,above left,scale=0.75] {$1$};
\draw[->,thick,blue]  (C4) -- (D5) node[midway,above left,scale=0.75] {$1$};
\draw[->,thick,blue]  (D4) -- (E5) node[midway,above left,scale=0.75] {$1$};
\draw[->,thick,blue]  (E4) -- (H5) node[midway,above right,scale=0.75] {$1$};
\draw[->,thick,blue]  (F4) -- (I5) node[midway,above right,scale=0.75] {$1$};
\draw[->,thick,blue]  (H4) -- (F5) node[midway,above left,scale=0.75] {$1$};  
\draw[->,thick,red]  (A4) -- (B5) node[midway,above left,scale=0.75] {$2$};
\draw[->,thick,red]  (B4) -- (C5) node[midway,above left,scale=0.75] {$2$};
\draw[->,thick,red]  (C4) -- (E5) node[midway,above right,scale=0.75] {$2$};
\draw[->,thick,red]  (D4) -- (F5) node[midway,above right,scale=0.75] {$2$};
\draw[->,thick,red]  (E4) -- (G5) node[midway,right,scale=0.75] {$2$};
\draw[->,thick,red]  (F4) -- (H5) node[midway,right,scale=0.75] {$2$};
\draw[->,thick,red]  (H4) -- (J5) node[midway,right,scale=0.75] {$2$};  
     \end{tikzpicture}
\]
\caption{The $\sqrt{\gl_3}$-crystal $\SetTab_3(\infty)$ up to depth at least 3 with the elements written as marginally large set-valued tableaux.}
\label{fig:gl3mlt}
\end{figure}

\begin{remark}
\label{rem:sub_direct_limits}
Write  $\lambda \preceq \mu$ if $\lambda,\mu \in P^+$ have $\mu - \lambda \in P^+$.
Under the order $\preceq$, any infinite directed set of weights $\overline{P}^+ \subseteq P^+$ (in which every finite set has an upper bound)   gives rise to a direct limit crystal, namely by replacing $P^+$ in \eqref{eq:direct_limit_def} by $\overline{P}^+$.
These limits are not necessarily isomorphic,
although they are always upper regular.
Particularly important examples include the sets
\begin{equation}
\label{eq:parabolic_directed_weights}
P_J^+ := \{\lambda \in P^+ \mid \langle \lambda, \alpha_i^{\vee} \rangle = 0 \text{ for all } i \in I \setminus J \}
\quad\text{for each $J \subseteq I = [n-1]$.}
\end{equation}
\end{remark}

In the classical case, the direct limits for these directed sets correspond to crystals describing parabolic Verma modules of highest weight $0$, and they naturally have marginally large tableau descriptions  (just restricted to columns whose heights are $J$).
For our applications, we will only need to consider the case $J = \{1\}$, which corresponds to a single row,  as it can be used to describe the general case seen in the proof of Theorem~\ref{thm:directed_system} (and in the marginally large tableau model).

\section{Lattice point models}
\label{sec:parameterizations}

In this section, we explore how different parameterizations of Kashiwara crystals in type $\gl_n$ carry over to $\sqgln$-crystals.
More specifically, the goal of this section is to construct versions of the following three parameterizations of the classical Kashiwara crystals $\cB(\lambda)$ for $\lambda \in P^+ \sqcup \{\infty\}$:
\begin{enumerate}
\item[(1)] The Lusztig parameterization~\cite{Kamnitzer10,BZ01,Lusztig90,Lusztig93}, which encodes the PBW basis of the negative half of the quantum group $U_q^-(\gl_n)$.
\item[(2)] The polyhedral model of Nakashima and Zelevinsky~\cite{Nakashima99,NZ97} that comes from the \defn{Kashiwara embedding} $\Upsilon_i \colon \cB(\infty) \hookrightarrow \cE_i \otimes \cB(\infty)$~\cite[Thm.~2.2.1]{Kashiwara93} (with the model implicitly defined in~\cite{Kashiwara93}).
\item[(3)] The string data parameterization, which was studied by Littelmann~\cite{Littelmann98} and can be considered as a particular ``dual'' version of the polyhedral model.
\end{enumerate}
The Kashiwara embedding is important for $\gl_n$ as it can be used to characterize $\cB(\infty)$~\cite[Prop.~3.2.3]{KS97}, identify $\cB(\lambda)$ inside of $\cB(\infty)$~\cite[Prop.~8.2]{Kashiwara95}, and establish Demazure crystal properties~\cite{Kashiwara93}.

We are able to achieve a Lusztig-type parameterization (Theorem~\ref{BZL-cor}) and a polyhedral-type model (Theorem~\ref{bzl-thm}), but only for one specific reduced expression for the longest element $w_0$. Namely, our constructions require the use of the \defn{BZL word} (named after Berenstein, Zelevinsky, and Littelmann~\cite{BZ01,Littelmann98}) given by:
\begin{subequations}
\label{bzl-eq}
\begin{align}
w_0 & = s_{n-1} (s_{n-2} s_{n-1}) \cdots (s_2s_3s_4 \cdots s_{n-1}) (s_1 s_2 s_3\cdots s_{n-1}) \in S_n,
\\
\BZL & := (n-1, n-2, n-1,  \dotsc, 2,3,4,\dots,n-1,  1,2,3, \dotsc, n-1).
\end{align}
\end{subequations}
It is interesting to note that the \defn{rectification} operator for polynomial $\sqgln$-crystals defined in \cite[Thm.~2.21]{MTY} also requires this specific reduced word (or more precisely, its reversed dual version).
We discuss what goes wrong for non-BZL words in Section~\ref{sec:nonBZL}.

\subsection{Vector parameterization}\label{vec-param-sect}

Fix a positive integer $k$ and consider a one-row semistandard set-valued tableau $T \in \SetTab_n(k \Lambda_1)$.
Such a tableau is uniquely determined by the following data:
\begin{itemize}
\item the number of boxes containing $i$, which we denote by $m_i=m_i(T) \in \ZZ_{\geq 0}$, and
\item whether or not the leftmost entry $A$ containing an $i$ has $\min A \neq i$, which we denote by 
\[
b_i  = b_i(T):= 
\begin{cases} 1 & \text{if such an entry $A$ exists and } \min A \neq i, \\ 0 & \text{otherwise.} \end{cases}
\]
\end{itemize}
For these parameters, we have $k = \sum_{i=1}^n (m_i - b_i)$ and it always holds that $b_i = 0$ when $m_i=0$ or when $i \in [n]$ is the minimal index with $m_i>0$.
This means we can define 
\[
m'_i= m'_i(T):= m_i - \tfrac{1}{2} b_i \in \HH_{\geq 0} := \tfrac{1}{2}\NN.
\]
Notice that we can recover $m_i = \lceil m_i' \rceil$ and $b_i = 2 (m'_i - \lfloor m'_i \rfloor)= \lceil m_i' \rceil - \lfloor m'_i \rfloor$.
Moreover,   
\[ 
\lfloor m'_1 \rfloor + \dots  + \lfloor m'_n \rfloor 
=k
\qquand
m'_i \in \ZZ\text{ if $i \in [n]$ is minimal with $m'_i\neq 0$.}
\]
Thus, we can parametrize $\SetTab_n(k \Lambda_1)$ by the set of vectors
\[
\VCrys{k} := \left\{ (a_1, a_2,\dotsc, a_n) \in \HH_{\geq 0}^n\ \middle\vert\ 
\ba
 &\lfloor a_1 \rfloor +\lfloor a_2 \rfloor + \cdots + \lfloor a_n \rfloor = k \text{ and} 
 \\
 &a_i \in \ZZ\text{ if $i \in [n]$ is minimal with $a_i\neq 0$} 
 \ea\right\}
\]
under the bijection $\Theta_k \colon \SetTab_n(k \Lambda_1) \to \VCrys{k}$ defined by
\[
\Theta_k(T) = (m'_1, \dotsc, m'_n).
\]
This parameterization makes sense even when $k = 0$ since $\SetTab_n(0)$ contains just the empty tableau and $\VCrys{0}$ contains just the zero vector in $\ZZ^n$.

\begin{example}
\label{VCrys-ex1}
As in Example~\ref{ex:column_not_subsets}, we cannot characterize a tableau $T \in \SetTab_n(k \Lambda_1)$ by just the sequence $ \mm = (m_1, \dotsc, m_n)$, unlike for $\gl_n$-crystals.
Consider the following elements of $\SetTab_3(2\Lambda_1)$ and their corresponding data $\bfa$ from $\Theta_2$:
\[
\ytableaushort{{12}3} \mapsto (1,\tfrac{1}{2},1),
\qquad
\ytableaushort{1{23}} \mapsto (1,1,\tfrac{1}{2}),
\qquad
\ytableaushort{{12}{23}} \mapsto (1,\tfrac{3}{2},\tfrac{1}{2}).
\]
The first two elements both have $\mm = (1,1,1)$ and weight $-\e_0+\e_1 + \e_2 + \e_3$.  
\end{example}

The map $\Theta_k$ induces a regular $\sqgln$-crystal structure on $\VCrys{k}$, which we now describe explicitly.
Note that we use the following minor abuse of notation: For a vector of the form $\bfa = (a_k, \dotsc, a_n)$ for some $k > 0$, we define $\bfa + C\e_i := (a_k,\dotsc, a_i + C,\dotsc,a_n)$ for $i \in [k, n]$ and it is undefined otherwise.

\begin{definition}
\label{def:vec_crystal} 
Define a $\sqgln$-crystal structure on $\VCrys{k}$ by
\begin{subequations}
\begin{align}
e_i \bfa & = \begin{cases}
	\bfa + \frac{1}{2} \e_i & \text{if } a_{i+1} \notin \ZZ \text{ and } a_i = 0, \\
	\bfa + \e_i - \frac{1}{2} \e_{i+1} & \text{if } a_{i+1} \in \ZZ_{>0}, \\
	\bfa - \frac{1}{2} \e_{i+1} & \text{if } a_{i+1} \notin \ZZ \text{ and } a_i > 0, \\
	\zero & \text{if } a_{i+1} = 0,
\end{cases}  \label{eq:vec_e}
\allowdisplaybreaks \\
f_i \bfa & = \begin{cases}
	\bfa - \frac{1}{2} \e_i & \text{if } a_{i+1} \notin \ZZ \text{ and } a_i = \frac{1}{2}, \\
	\bfa - \e_i + \frac{1}{2} \e_{i+1} & \text{if } a_{i+1} \notin \ZZ \text{ and } a_i > \frac{1}{2}, \\
	\bfa + \frac{1}{2} \e_{i+1} & \text{if } a_{i+1} \in \ZZ \text{ and } a_i > 0, \\
	\zero & \text{if } a_i = 0,
\end{cases}
\allowdisplaybreaks \\
\varepsilon_i(\bfa) & = \begin{cases} a_{i+1} & \text{if } a_{i+1} \notin \ZZ \text{ and } a_i > 0, \\ \lceil a_{i+1}\rceil & \text{otherwise}, \end{cases}  \label{eq:vec_ep}
\allowdisplaybreaks \\
\varphi_i(\bfa) & = \begin{cases} \lceil a_i \rceil - \frac{1}{2} & \text{if } a_{i+1} \notin \ZZ \text{ and } a_i > 0, \\ \lceil a_i \rceil & \text{otherwise}, \end{cases}
\allowdisplaybreaks \\
\weight(\bfa) & = k \e_0 + \sum_{i=1}^n \lceil a_i \rceil (\e_i - \e_0) = \sum_{i=1}^n\Bigl( \lceil a_i \rceil \e_i - (\lceil a_i \rceil - \lfloor a_i \rfloor ) \e_0\Bigr).
\end{align}
\end{subequations}
\end{definition}

In terms of tableaux, the cases for the $f_i$ crystal operators are respectively described by
\[
\ytableausetup{boxsize = .55cm,aligntableaux=center}
\ytableaushort{{hij}} \xrightarrow[\hspace{20pt}]{i} \ytableaushort{{hj}}\,,
\qquad\quad
\ytableaushort{h{ij}} \xrightarrow[\hspace{20pt}]{i} \ytableaushort{hj}\,,
\qquad\quad
\ytableaushort{hi} \xrightarrow[\hspace{20pt}]{i} \ytableaushort{h{ij}}\,,
\]
where $h \leq i < j := i+1$ (with $h < i$ in the first case).

\begin{proposition}
\label{prop:single_row_crystal} 
Fix $k \in \NN$.
Then the map $\Theta_k \colon \SetTab_n(k\Lambda_1) \to \VCrys{k}$ is a $\sqgln$-crystal isomorphism.
Moreover, $\VCrys{k}$ is a regular $\sqgln$-crystal.
\end{proposition}

Compare Figure~\ref{fig:n3_k2_vec} with Figure~\ref{fig:S3square_cc} for an example of the isomorphism $\Theta_k$.

\begin{proof}
Inspecting the weight of a semistandard set-valued tableau~\eqref{eq:explicit_tableau_weight} alongside the remarks before Example~\ref{VCrys-ex1} shows that $\Theta_k$ is a weight-preserving bijection.
One can check directly that $\varepsilon_i$ and $\varphi_i$ have the formulas \eqref{upper-reg-eq} and \eqref{lower-reg-eq} required of a regular crystal.
Since we know that $\SetTab_n(k \Lambda_1)$ is a regular $\sqgln$-crystal, it remains only to check that $\Theta_k$ commutes with the crystal operators.
This is straightforward from the definition of the tensor product for $\sqgln$-crystals (or from the signature rule in \cite{Yu23}) since we can identify $T \in \SetTab_n(k \Lambda_1)$ with $T_{11} \otimes T_{12} \otimes \cdots \otimes T_{1k} \in \SS_n^{\otimes k}$. 
\end{proof}

\begin{figure}[t]
\[
\newcommand{\hf}[1]{\frac{#1}{2}}
    \begin{tikzpicture}[xscale=3.2, yscale=1.5,>=latex,baseline=0] 
      \node at (0,4) (T0) {$(2,0,0)$};
      \node at (0,3) (T1) {$(2,\hf1,0)$};
      \node at (0,2) (T2a) {$(1,1,0)$};
      \node at (1,3) (T2b) {$(2,\hf1,\hf1)$};
      \node at (0,1) (T3a) {$(1,\hf3,0)$};
      \node at (1,2) (T3b) {$(1,1,\hf1)$};
      \node at (2,3) (T3c) {$(2,0,\hf1)$};
      \node at (0,0) (T4a) {$(0,2,0)$};
      \node at (1,1) (T4b) {$(1,\hf3,\hf1)$};
      \node at (2,2) (T4c) {$(1,0,1)$};
      \node at (1,0) (T5a) {$(0,2,\hf1)$};
      \node at (2,1) (T5b) {$(1,\hf1,1)$};
      \node at (2,0) (T6a) {$(0,1,1)$};
      \node at (3,1) (T6b) {$(1,\hf1,\hf3)$};
      \node at (3,0) (T7a) {$(0,1,\hf3)$};
      \node at (4,1) (T7b) {$(1,0,\hf3)$};
      \node at (4,0) (T8) {$(0,0,2)$};
      \draw[->,thick,color=blue]  (T0) -- (T1) node[midway,right,scale=0.75] {$1$};
      \draw[->,thick,color=blue]  (T1) -- (T2a) node[midway,right,scale=0.75] {$1$};
      \draw[->,thick,color=blue]  (T2a) -- (T3a) node[midway,right,scale=0.75] {$1$};
      \draw[->,thick,color=blue]  (T2b) -- (T3b) node[midway,right,scale=0.75] {$1$};
      \draw[->,thick,color=blue]  (T3a) -- (T4a) node[midway,right,scale=0.75] {$1$};
      \draw[->,thick,color=blue]  (T3b) -- (T4b) node[midway,right,scale=0.75] {$1$};
      \draw[->,thick,color=blue]  (T4b) -- (T5a) node[midway,right,scale=0.75] {$1$};
      \draw[->,thick,color=blue]  (T4c) -- (T5b) node[midway,right,scale=0.75] {$1$};
      \draw[->,thick,color=blue]  (T5b) -- (T6a) node[midway,right,scale=0.75] {$1$};
      \draw[->,thick,color=blue]  (T6b) -- (T7a) node[midway,right,scale=0.75] {$1$};
      \draw[->,thick,color=red]  (T1) -- (T2b) node[midway,above,scale=0.75] {$2$};
      \draw[->,thick,color=red]  (T2a) -- (T3b) node[midway,above,scale=0.75] {$2$};
      \draw[->,thick,color=red]  (T3b) -- (T4c) node[midway,above,scale=0.75] {$2$};
      \draw[->,thick,color=red]  (T3a) -- (T4b) node[midway,above,scale=0.75] {$2$};
      \draw[->,thick,color=red]  (T4b) -- (T5b) node[midway,above,scale=0.75] {$2$};
      \draw[->,thick,color=red]  (T5b) -- (T6b) node[midway,above,scale=0.75] {$2$};
      \draw[->,thick,color=red]  (T4a) -- (T5a) node[midway,above,scale=0.75] {$2$};
      \draw[->,thick,color=red]  (T5a) -- (T6a) node[midway,above,scale=0.75] {$2$};
      \draw[->,thick,color=red]  (T6a) -- (T7a) node[midway,above,scale=0.75] {$2$};
      \draw[->,thick,color=red]  (T7a) -- (T8) node[midway,above,scale=0.75] {$2$};
      \draw[->,thick,color=red]  ([yshift=1pt]T2b.east) -- ([yshift=1pt]T3c.west) node[midway,above,scale=0.75] {$2$};
      \draw[->,thick,color=blue]  ([yshift=-1pt]T3c.west) -- ([yshift=-1pt]T2b.east) node[midway,below,scale=0.75] {$1$};
      \draw[->,thick,color=red]  ([yshift=1pt]T6b.east) -- ([yshift=1pt]T7b.west) node[midway,above,scale=0.75] {$2$};
      \draw[->,thick,color=blue]  ([yshift=-1pt]T7b.west) -- ([yshift=-1pt]T6b.east) node[midway,below,scale=0.75] {$1$};
     \end{tikzpicture}
\]
\caption{The $\sqgln$-crystal $\VCrys{2}$ for $n = 3$.}
\label{fig:n3_k2_vec}
\end{figure}

Next, we want to describe equivalence classes in the direct limit
\[
\varinjlim_{k\to\infty} \cT_{-k\Lambda_1} \otimes \SetTab_n(k\Lambda_1)
\]
using this parameterization (which exists by Remark~\ref{rem:sub_direct_limits}).
To be explicit, for each $k,m \in \NN$, the directed system is given by the $\sqgln$-crystal embeddings
\[
\Psi^{(1)}_{k,m} \colon \cT_{-k\Lambda_1}\otimes\VCrys{k}\to \cT_{-(k+m)\Lambda_1} \otimes \VCrys{k+m}
\]
defined by 
\be\label{lc-psi-eq}
\Psi^{(1)}_{k,m}(t_{-k\Lambda_1}\otimes \bfa) := t_{-(k+m)\Lambda_1} \otimes (a_1+m, a_2,a_3,\dotsc,a_n).
\ee
This can be checked directly or by observing that in terms of the maps $\Psi_{\lambda,\mu}$ from \eqref{eq:insertion_map}, we have
\[
\Psi^{(1)}_{k,\ell} = \( \id \otimes \Theta_{k+\ell}\) \circ \Psi_{k\Lambda_1,\ell\Lambda_1} \circ \( \id\otimes\Theta_k^{-1}\).
\]
Since each $\bfa \in \VCrys{k}$ must have $a_1 \in \ZZ$, we see that an equivalence class in the direct limit corresponds to a set of vectors having the same (fixed) components $(a_2, \dotsc, a_n)$.
Hence, we represent each class as a single vector $(a_2, \dotsc, a_n)$.
Consequently, we define the crystal operators acting on  
\[
\VCrys{\infty} := \{ (a_2, \dotsc, a_n) \mid  a_2, \dotsc, a_n \in \HH_{\geq 0} \}
\]
with the crystal operators ignoring all conditions on $a_1$ (as we always consider it to be ``large'' and the first case for $f_1$ cannot occur as $a_1 \in \ZZ$) and applying the projection $\e_1 \mapsto 0$.
For use later on, we instead give a more general construction of the precise direct limit
\[
\VCrysR{j} := \{ \overline{\bfa}=(a_{j+1}, \dotsc, a_n) \mid a_{j+1}, \dotsc, a_n \in \HH_{\geq 0} \}.
\]
Note that $\VCrys{\infty} = \VCrysR{1}$ and that $\VCrysR{j}$ (as a set) is just the Cartesian product of $n-j$ copies of $\HH_{\geq 0}$.
The elements of $\VCrysR{j}$ are vectors indexed by $\{j+1,\dots,n\}$ rather than $\{1,\dots,n-j\}$.

\begin{definition}
\label{def:vec_crystal2}
We view the set $\VCrysR{j}$ as a $\sqgln$-crystal in the following way.
Let $\overline{\bfa} \in \VCrysR{j}$.
The crystal operators and statistics are given by the following cases:
\bei
\item[(a)] For $i \in [j-1]$, set $\varepsilon_i(\overline{\bfa}) =\varphi_i(\overline{\bfa})= -\infty$ and $e_i \overline{\bfa} = f_i \overline{\bfa} = \zero$.
\item[(b)] For $i \in [n-1] \setminus [j]$, define $e_i$, $f_i$, $\varepsilon_i$, and $\varphi_i$ by the same formulas as in Definition~\ref{def:vec_crystal}.

\item[(c)] For $i = j$ let $\varepsilon_j(\overline{\bfa}) = a_{j+1}$
and 
$\varphi_j(\overline{\bfa})  = -a_{j+1} - \sum_{m=j+2}^n \lfloor a_m \rfloor$ along with
\[
e_j \overline{\bfa}  = \begin{cases} \overline{\bfa} - \frac{1}{2} \e_{j+1} & \text{if } a_{j+1} > 0 \\
\zero & \text{if } a_{j+1} = 0
\end{cases}
\qquand
f_j \overline{\bfa}  = \overline{\bfa} + \tfrac{1}{2} \e_{j+1}.
\]
\eei
Finally, the weight map on $\VCrysR{j}$ is 
$\weight(\overline{\bfa})  
 = \sum_{i=j+1}^n \Bigl( \lfloor a_i \rfloor (\e_i - \e_j) - (\lfloor a_i \rfloor - \lceil a_i \rceil) (\e_i - \e_0) \Bigr).
$
\end{definition}

The character of this crystal has a simple product formula.

\begin{proposition}
\label{prop:infinite_row_char}
The set $\VCrysR{j}$ is a $\sqgln$-crystal whose character is
\[
\ch\left(\VCrysR{j}\right) = \prod_{i=j+1}^n \frac{1 + \beta x_i}{1 - x_i x_j^{-1}} \in \ZZ\llbracket\Lambda\rrbracket_<.
\]
\end{proposition}

\begin{proof}
Checking the $\sqgln$-crystal axioms is straightforward.
Since the components $a_{j+1},\dots,a_n$ of the vectors $\bfa \in \VCrysR{j}$ vary over all elements of $\HH_{\geq0}$ independently, we have 
\[
\ch\left(\VCrysR{j}\right) = \prod_{i=j+1}^n \sum_{a \in \HH_{\geq0}}  (x_ix_j^{-1})^{\lfloor a \rfloor} (\beta x_i)^{\lceil a \rceil - \lfloor a \rfloor} = \prod_{i=j+1}^n (1+\beta x_i)\sum_{a\in\NN}  (x_ix_j^{-1})^{a}.
\]
The value $\lfloor a_i \rfloor$ contributes to $x_i x_j^{-1}$ and gives the denominator (expanded as a formal power series).
Finally, we have $a_i \notin \ZZ$ if and only if $\lceil a_i \rceil - \lfloor a_i \rfloor = 1$ if and only if we take   $\beta x_i$ in the numerator.
\end{proof}

Our vector description is motivated by the direct limit.
We now formally state this connection.

\begin{proposition}\label{vcrys-prop}
For all $j \in [n]$, there are $\sqgln$-crystal isomorphisms
\[
\VCrysR{j} \iso \varinjlim_{k\to\infty}\Bigl( \cT_{-k\e_j}\otimes \restrictF{[j,n]}\bigl(\VCrys{k} \bigr) \Bigr)
\iso \varinjlim_{k\to\infty} \Bigl(\cT_{-k\e_j}\otimes \restrictF{[j,n]} \bigl(\SetTab_n(k\Lambda_1)\bigr)\Bigr).
\]
Moreover, 
the values of $\varepsilon_i$ on $\VCrysR{j}$ for $i \in [j,n-1]$ satisfy the upper regularity condition \eqref{upper-reg-eq}.
\end{proposition}

Informally, $\VCrysR{j}$ is an upper regular crystal when the index set is restricted to $[j,n-1]$.

\begin{proof}
From the preceding discussion, the elements of the first direct limit are naturally in bijection with $\VCrysR{j}$.
This bijection is a crystal isomorphism by the definition of direct limits and the tensor product for $\sqgln$-crystals.
The second isomorphism then holds by Proposition~\ref{prop:single_row_crystal}.
The proof of the last claim about (partial) upper regularity 
is similar to Corollary~\ref{cor:inf_upper_reg}. 
\end{proof}

Alternatively, we could derive Proposition~\ref{vcrys-prop} directly from the marginally large tableau model, where we can have any number of entries $i$ and each $i > j$ may or may not share its leftmost box with a smaller entry.

We now combine the previous result with our technique of restricting to entries at least $i$ in row $i$ as in the proof of Theorem~\ref{thm:directed_system}.
To avoid technical issues involving empty tensor products, we assume $n>1$ for the next theorem, which is one of our main results of this section.

\begin{theorem}[BZL word Lusztig-type parameterization]
\label{BZL-cor}
There is a crystal isomorphism
\be
\label{eq:Lusztig_param}
\SetTab_n(\infty) \iso \VCrysR{n-1} \otimes \cdots \otimes \VCrysR{2} \otimes \VCrysR{1}
\ee
that maps $u_{\infty} \mapsto (0) \otimes (0, 0) \otimes \cdots \otimes (0, \dotsc, 0)$.
\end{theorem}

\begin{proof}
We get an embedding of $\SetTab_n(\infty)$ into the right side  from Proposition~\ref{vcrys-prop} and the extension of Corollary~\ref{cor:row_restriction} to the case $\SetTab_n(\infty)$ by the direct limit construction.
This embedding is surjective since any tuple of limiting row parameters on the right is realized by a semistandard set-valued tableau after sufficiently many singleton entries $\{j\}$
are prepended to row $j$, and hence is the image of some equivalence class in $\SetTab_n(\infty)$.
\end{proof}

When $n=1$, the isomorphism \eqref{eq:Lusztig_param} still holds as long as we interpret the right side as $\one$.

The preceding theorem tells us that 
we can represent a marginally large tableau as a sequence
\[
[T] \mapsto (a_{n-1,n}) \otimes (a_{n-2,n-1}, a_{n-2,n}) \otimes \cdots \otimes (a_{1,2}, \dotsc, a_{1,n})
\longleftrightarrow
[a_{ij}]_{i,j=1}^n,
\]
where for each $1\leq i<j\leq n$ we define $a_{ij} := m'_{j}(R)$ where $R$ is the $i$th row of $T$.
We can convert this data to a strictly upper triangular matrix by setting the entries $a_{ij} = 0$ for all $i \geq j$.
For example, when $n = 4$, we have
\[
(f) \otimes (d,e) \otimes (a,b,c)
\longleftrightarrow
\begin{bmatrix}
0 & a & b & c \\
0 & 0 & d & e \\
0 & 0 & 0 & f \\
0 & 0 & 0 & 0
\end{bmatrix}.
\]
Compare Figure~\ref{fig:gl3mlt} with Figure~\ref{fig:gl3vec} for an example of this parameterization of $\SetTab_n(\infty)$.

Theorem~\ref{BZL-cor}   allows us to easily compute the character of $\SetTab_n(\infty)$ using Proposition~\ref{prop:infinite_row_char}.

\begin{figure}[t]
\[
\newcommand{\hf}[1]{\frac{#1}{2}}
    \begin{tikzpicture}[xscale=2.1, yscale=2.3,>=latex] 
\node at (0,0) (A1) {$(0, 0, 0)$};
\node at (-1,-1) (A2) {$(\hf1,0,0)$};
\node at (0,-1) (B2) {$(0,0,\hf1)$};
\node at (1,-1) (C2) {$(0,\hf1,0)$};
\node at (-2,-2) (A3)  {$(1,0,0)$};
\node at (-1,-2) (B3) {$(\hf1,\hf1,0)$};
\node at (0,-2) (C3) {$(\hf1,0,\hf1)$};
\node at (1,-2) (D3) {$(0,\hf1,\hf1)$};
\node at (2,-2) (E3) {$(0,1,0)$};
\node at (-2,-3) (A4) {$(1,\hf1,0)$};
\node at (-3,-3) (B4) {$(\hf3,0,0)$};
\node at (-1,-3) (C4) {$(\hf1,1,0)$};
\node at (0,-3) (D4) {$(\hf1,\hf1,\hf1)$};
\node at (1,-3) (H4) {$(1,0,\hf1)$};  
\node at (2,-3) (E4) {$(0,1,\hf1)$};
\node at (3,-3) (F4) {$(0,\hf3,0)$};
\node at (-3.8,-3.8) (C5) {$\vdots$};
\node at (-3,-3.8) (B5) {$\vdots$};
\node at (-2,-3.8) (A5) {$\vdots$};
\node at (-1.5,-3.8) (D5) {$\vdots$};
\node at (-0.5,-3.8) (E5) {$\vdots$};
\node at (0.5,-3.8) (F5) {$\vdots$};
\node at (1.5,-3.8) (J5) {$\vdots$};  
\node at (2,-3.8) (G5) {$\vdots$};
\node at (3,-3.8) (H5) {$\vdots$};
\node at (3.8,-3.8) (I5) {$\vdots$};
\draw[->,thick,red]  (A1) -- (A2) node[midway,above left,scale=0.75] {$2$};
\draw[->,thick,blue]  (A1) -- (C2) node[midway,above right,scale=0.75] {$1$};
\draw[->,thick,red]  (A2) -- (A3) node[midway,above left,scale=0.75] {$2$};
\draw[->,thick,blue]  (A2) -- (B3) node[midway,left,scale=0.75] {$1$};
\draw[->,thick,red]  (B2) -- (C3) node[midway,right,scale=0.75] {$2$};
\draw[->,thick,blue]  ([xshift=4pt]B2.south) -- ([xshift=-8pt]D3.north) node[midway,below left,scale=0.75] {$1$};
\draw[<-,thick,red]  ([xshift=8pt]B2.south) -- ([xshift=-4pt]D3.north) node[midway,above right,scale=0.75] {$2$};
\draw[->,thick,red]  (C2) -- (D3) node[midway,right,scale=0.75] {$2$};
\draw[->,thick,blue]  (C2) -- (E3) node[midway,above right,scale=0.75] {$1$};
\draw[->,thick,blue]  (A3) -- (A4) node[midway,above left,scale=0.75] {$1$};
\draw[->,thick,red]  (A3) -- (B4) node[midway,above left,scale=0.75] {$2$};
\draw[->,thick,blue]  (B3) -- (C4) node[midway,left,scale=0.75] {$1$};
\draw[->,thick,red]  (B3) -- (D4) node[midway,above right,scale=0.75] {$2$};
\draw[->,thick,blue]  (C3) -- (D4) node[midway,left,scale=0.75] {$1$};
\draw[->,thick,blue]  (D3) -- (E4) node[midway,above right,scale=0.75] {$1$};
\draw[->,thick,red]  (E3) -- (E4) node[midway,above left,scale=0.75] {$2$};
\draw[->,thick,blue]  (E3) -- (F4) node[midway,above right,scale=0.75] {$1$};
\draw[->,thick,red]  (C3) -- (H4) node[midway,above right,scale=0.75] {$2$};  
\draw[->,thick,blue]  (A4) -- (A5) node[midway,above left,scale=0.75] {$1$};
\draw[->,thick,blue]  (B4) -- (B5) node[midway,above left,scale=0.75] {$1$};
\draw[->,thick,blue]  (C4) -- (D5) node[midway,above left,scale=0.75] {$1$};
\draw[->,thick,blue]  (D4) -- (E5) node[midway,above left,scale=0.75] {$1$};
\draw[->,thick,blue]  (E4) -- (H5) node[midway,above right,scale=0.75] {$1$};
\draw[->,thick,blue]  (F4) -- (I5) node[midway,above right,scale=0.75] {$1$};
\draw[->,thick,blue]  (H4) -- (F5) node[midway,above left,scale=0.75] {$1$};  
\draw[->,thick,red]  (A4) -- (B5) node[midway,above left,scale=0.75] {$2$};
\draw[->,thick,red]  (B4) -- (C5) node[midway,above left,scale=0.75] {$2$};
\draw[->,thick,red]  (C4) -- (E5) node[midway,above right,scale=0.75] {$2$};
\draw[->,thick,red]  (D4) -- (F5) node[midway,above right,scale=0.75] {$2$};
\draw[->,thick,red]  (E4) -- (G5) node[midway,right,scale=0.75] {$2$};
\draw[->,thick,red]  (F4) -- (H5) node[midway,right,scale=0.75] {$2$};
\draw[->,thick,red]  (H4) -- (J5) node[midway,right,scale=0.75] {$2$};  
     \end{tikzpicture}
\]
\caption{The $\sqrt{\gl_3}$-crystal $\SetTab_3(\infty)$ up to depth at least 3 after applying the isomorphism~\eqref{eq:Lusztig_param} and writing $(a) \otimes (b, c) = (a,b,c)$ for simplicity.}
\label{fig:gl3vec}
\end{figure}

\begin{corollary}
\label{cor:infchar}
The character of the direct limit is
\[
\ch\bigl( \SetTab_n(\infty) \bigr) = \prod_{1\leq i < j\leq n} \frac{1}{1 - x_j x_i^{-1}} \prod_{i=1}^n (1 + \beta x_i)^{i-1} \in \ZZ\llbracket\Lambda\rrbracket_<.
\]
\end{corollary}

We know from \eqref{cB-SetTab-eq} that the character of the $\gl_n$-crystal $\cB(\infty) = \defectF{2}\bigl(\SetTab_n(\infty)\bigr)$ is obtained by setting $\beta = 0$ in $\ch\bigl(\SetTab_n(\infty)\bigr)$.
Thus, 
we have the simple factorization
\be
\ch\bigl(\SetTab_n(\infty)\bigr) = \ch\bigl(\cB(\infty)\bigr)  \prod_{i=1}^n (1 + \beta x_i)^{i-1}.
\ee

By contrast, the characters  $\ch\bigl(\SetTab_n(\lambda)\bigr) = G_\lambda(x;\beta)$ for $\lambda \in P^+$ are not generally divisible by the characters $s_{\lambda}$ of the Kashiwara crystals of semistandard Young tableaux given by $\defectF{2}(\SetTab_n(\lambda))$.
Likewise,
we do not know of any simple factorization for the characters of Demazure crystals. 
The following is a natural open problem:

\begin{problem}
Find a character formula for $\ch\bigl( \SetTab_n(\infty)_w \bigr)$ when $w \in S_n$ is arbitrary.
\end{problem}

\begin{example}
It might seem reasonable that the character of $ \SetTab_n(\infty)_w $ would have the form $ \prod_{\e_i - \e_j \in \Omega_w} \frac{1 + \beta x_j}{1 - x_j x_i^{-1}} $ for some subset  $\Omega_w \subseteq \Phi^+$, similar to negative quantum Schubert cells.
However, this type of formula can fail to exist even in the case for Kashiwara crystals (i.e., when $\beta = 0$).
Consider $\cB(\infty)_{s_2 s_1 s_3 s_2}$ when $n=4$, whose character contains the terms
\[
1 + a_1 + a_2 + a_3 + 2 a_1 a_2 + 2 a_2 a_3 + 2 a_1 a_2 a_3 + \cdots,
\]
where $a_i = x_i^{-1} x_{i+1}=x^{-\alpha_i}$, 
coming from the marginally large tableaux
\begin{gather*}
\ytableausetup{boxsize=1em}
\ytableaushort{111,22,3}\,,
\quad
\ytableaushort{1112,22,3}\,,
\quad
\ytableaushort{1111,223,3}\,,
\quad
\ytableaushort{1111,222,34}\,,
\quad
\ytableaushort{1113,22,3}\,,
\quad
\ytableaushort{11112,223,3}\,,
\\
\ytableaushort{1111,224,3}\,,
\quad
\ytableaushort{11111,2223,34}\,,
\quad
\ytableaushort{11113,222,34}\,,
\quad
\ytableaushort{11112,224,3}\,.
\end{gather*}
If the character could be factorized in the suggested form, then the set $\Omega_{s_2s_1s_3s_2}$ would have to contain at least the 5 roots  
$\alpha_1$, $\alpha_2$, $ \alpha_3$,  $\alpha_1+\alpha_2$, $ \alpha_2+\alpha_3$,
which contribute the terms
\[
\frac{1}{1 - a_1} \frac{1}{1 - a_2} \frac{1}{1 - a_3} \frac{1}{1 - a_1 a_2} \frac{1}{1 - a_2 a_3}
= 1 + a_1 + a_2 + a_3 + 2 a_1 a_2 + 2 a_2 a_3 + 3 a_1 a_2 a_3 + \cdots.
\]
However,   the Demazure crystal does not contain 3 elements of weight $-\alpha_1 - \alpha_2 - \alpha_3=\e_4-\e_1$ that could account for the last term  $3 a_1 a_2 a_3 $.
Specifically, there are exactly 4 marginally large tableaux of this weight  for $n=4$; these include the last two tableaux listed above as well as 
\[
\ytableaushort{1114,22,3}\,
\qquad\text{and}\qquad
\ytableaushort{111112,2223,34}\,,
\]
which do not appear in the Demazure crystal. Hence the proposed factorization is impossible.
\end{example}

\begin{remark}
Our way of parametrizing a one-row set-valued tableau $T$ by  $(m'_1,\dots,m'_n)$ is related by an invertible piecewise-linear change of coordinates to the so-called \defn{marked Gelfand--Tsetlin (GT) pattern} description of set-valued tableaux from, e.g.,~\cite{MPS21}.
Namely, we build the marked GT pattern diagonal-by-diagonal by considering some $T \in \SetTab_n(k\Lambda_1)$ with associated parameters $m_i$, $b_i$, and $m'_i$ and using the (invertible) change of coordinates
\begin{equation}
\label{eq:GT_pattern_map}
s_i := \sum_{j=1}^{i-1} \(m_j - b_j\) + m_i - \tfrac{1}{2} b_i = \sum_{j=1}^{i-1} \lfloor m'_j \rfloor + m'_i,
\end{equation}
where $\lfloor s_i \rfloor$ is the length of the row of $T$ if we remove all entries $> i$, and $s_i \in \ZZ$ if and only if the corresponding entry is not marked.

Also, one can obtain alternate versions of the crystal structures on vectors described in this section by replacing the boolean values $b_i$ in our encoding of one-row set-valued tableaux by
\[
b^{\dagger}_i := \begin{cases} 1 & \text{if there is a \emph{rightmost} entry $A$ with $i \in A$ and } \max A \neq i, \\ 0 & \text{otherwise.} \end{cases}
\]
This would lead to a natural definition of ``shifted'' marked GT patterns.
\end{remark}

\subsection{Looped path crystals}

Now we construct our $\sqgln$-crystal analog of the Kashiwara embedding.
We start with a variant of the $j$-th elementary crystal $\cE_j$ from Example~\ref{ex:elem_crystal}.

\begin{definition}[Looped path crystal]
\label{def:looped_path}
Fix $j \in [n-1]$.
The \defn{looped $j$-path crystal} $\sqrtE_j$ is the set of symbols $\{ \uj{j}{m} \mid m \in \frac{1}{2}\ZZ\}$ with the following $\sqgln$-crystal structure.
The weight function is 
\[
\weight(\uj{j}{m}) = \lfloor m \rfloor \alpha_j + (\lceil m \rceil - \lfloor m \rfloor) (\e_j - \e_0) \in \Lambda = \ZZ\oplus \ZZ^n.
\]
For $i \in [n-1] \setminus \{j, j+1\}$, we have $\varepsilon_i(\uj{j}{m}) = \varphi_i(\uj{j}{m}) = -\infty$ and $e_i\uj{j}{m} = f_i\uj{j}{m} = \zero$.
We let
\[ 
\ba
e_j \uj{j}{m} &=\uj{j}{m+\tfrac{1}{2}},
\\
f_j \uj{j}{m} &=\uj{j}{m-\tfrac{1}{2}},
\ea
\qquand
\ba
\varepsilon_j(\uj{j}{m}) &= -m,  \\
\varphi_j(\uj{j}{m}) &= m.
\ea
\]
Finally, when $j < n - 1$, we define
\[
\ba
e_{j+1} \uj{j}{m} &=\begin{cases}
\uj{j}{m-\tfrac{1}{2}} &\text{if }m \in \ZZ\\
\zero &\text{if }m \notin \ZZ,
\end{cases} 
\\
f_{j+1} \uj{j}{m} &=\begin{cases}
\uj{j}{m+\tfrac{1}{2}}&\text{if }m\notin\ZZ\\
\zero &\text{if }m \in \ZZ,
\end{cases}
\ea
\qquand
\ba
\varepsilon_{j+1}(\uj{j}{m}) &= m,\\
\varphi_{j+1}(  \uj{j}{m}) &=\lceil m \rceil - m \in \left\{0,\tfrac{1}{2}\right\}.
\ea
\]
\end{definition}

When $j \in[n-2]$,
the crystal graph of $\sqrtE_j$ is the infinite ``path-with-loops'' shown as
\[
    \begin{tikzpicture}[xscale=2, yscale=1.5,>=latex,baseline=(z.base)]
    \node at (0,0.0) (z) {};
      \node at (0,0) (left) {$\cdots$};
      \node at (1,0) (B) {$\uj{j}{1}$};
      \node at (2,0) (C) {$\uj{j}{\tfrac{1}{2}}$};
      \node at (3,0) (D) {$\uj{j}{0}$};
      \node at (4,0) (E) {$\uj{j}{-\tfrac{1}{2}}$};
      \node at (5,0) (right) {$\cdots$};
      \draw[->,thick,blue]  (left) -- (B) node[midway,above,scale=0.75] {$j$};
      
      \draw[->,thick,blue]  ([yshift=1.5pt]B.east) -- ([yshift=1.5pt]C.west) node[midway,above,scale=0.75] {$j$};
      \draw[<-,thick,red]  ([yshift=-1.5pt]B.east) -- ([yshift=-1.5pt]C.west) node[midway,below,scale=0.75] {$j+1$};
      
      \draw[->,thick,blue]  (C) -- (D) node[midway,above,scale=0.75] {$j$};

      \draw[->,thick,blue]  ([yshift=1.5pt]D.east) -- ([yshift=1.5pt]E.west) node[midway,above,scale=0.75] {$j$};
      \draw[<-,thick,red]  ([yshift=-1.5pt]D.east) -- ([yshift=-1.5pt]E.west) node[midway,below,scale=0.75] {$j+1$};
      
      \draw[->,thick,blue]  (E) -- (right) node[midway,above,scale=0.75] {$j$};
      
     \end{tikzpicture}\raisebox{-3pt}{.}
\]
When $j = n-1$, the crystal graph is a simple infinite path with only $\xrightarrow {\ n-1 \ }$ arrows.
Notice that 
\be\label{path-level-eq}
\defect(\uj{j}{m}) = \langle \weight(\uj{j}{m}), -\e_0\rangle
=
\begin{cases} 0 & \text{if }m \in \ZZ, \\ 1&\text{if }m \notin \ZZ.
\end{cases}
\ee
The following result is straightforward.

\begin{proposition}
Each $\sqrtE_j$ is a $\sqgln$-crystal in $\sK^+$ with character
\[
\ch(\sqrtE_j) =
 (1+\beta x_{j}) \sum_{m \in \ZZ} x_j^mx_{j+1}^{-m} \in \ZZ\llbracket\Lambda\rrbracket.
\]
\end{proposition}

\begin{remark}
\label{rem:square_func_elem}
Applying the functor $\defectF{2}$ to $\sqrtE_j$ yields a crystal that is closely related to the elementary crystal $\cE_j$ for $\gl_n$. 
The two crystals have the same crystal graph and weight map, but are only isomorphic when $j=n-1$.
When $j<n-1$,  crucially, they differ in the values of the statistics $\varepsilon_{j+1}$ and $\varphi_{j+1}$.
\end{remark}

We will use the following notation for brevity.
For any sequence $J = (j_k \in I)_{k=1}^{\ell}$, define
\be
\label{eq:path_seq_defn}
\sqrtE_J := \sqrtE_{j_1} \otimes \cdots \otimes \sqrtE_{j_{\ell}}
\quand
\uj{J}{m_1, \dotsc, m_{\ell}} := \uj{j_1}{m_1} \otimes \cdots \otimes \uj{j_{\ell}}{m_{\ell}} \in \sqrtE_J.
\ee
We also set $\sqrtE_{j,n} := \sqrtE_{(j, j+1, \dotsc, n-1)}$
and
$\uj{j,n}{m_1, \dotsc, m_{n-j}} := \uj{(j, j+1, \dotsc, n-1)}{m_1, \dotsc, m_{n-j}}$.

\begin{remark}
\label{rem:braid_rels}
For classical $\gl_n$-crystals, one has $\cE_{i_1}\otimes\cdots \otimes \cE_{i_k}\iso \cE_{j_1}\otimes\cdots \otimes \cE_{j_k}$ whenever $s_{i_1}\cdots s_{i_k} = s_{j_1}\cdots s_{j_k}$ are reduced expressions for the same element of the Weyl group~\cite[Props.~2.2.8 and~3.2.5]{Kashiwara93} (see also, e.g.,~\cite[Prop.~12.13]{BumpSchilling}).
By contrast, the only braid relations satisfied by tensor products of the crystals $\sqrtE_j$ are
\be
\sqrtE_j\otimes \sqrtE_k = \sqrtE_{(j,k)} \iso \sqrtE_{(k,j)} = \sqrtE_k \otimes \sqrtE_j
\qquad\text{ when } \abs{j - k} > 2.
\ee
In general, one has
$
 \sqrtE_{(j,j+2)} \not\iso \sqrtE_{(j+2,j)}
$
and
$
 \sqrtE_{(j,j+1,j)} \not\iso \sqrtE_{(j+1,j,j+1)}.
 $
The first non-isomorphism occurs because of discrepancies in the values of the statistics $\varepsilon_{j+1}$ and $\varphi_{j+1}$, despite there being a weight-preserving isomorphism between the respective crystal graphs.
\end{remark}

The following lemma is the main technical result of this section.
  
\begin{lemma}
\label{lemma:root_Kashiwara_embedding}
There exists a $\sqgln$-crystal embedding
\[
\VCrysR{k} \hookrightarrow \sqrtE_{k,n}
\qquad
\text{ sending }
\qquad
(0, \dotsc, 0) \mapsto \uj{k,n}{0,\dotsc,0}.
\]
\end{lemma}

To prove Lemma~\ref{lemma:root_Kashiwara_embedding}, 
we examine the $\sqgln$-crystal structure on the codomain of the desired embedding.
We identify $\uj{k,n}{v_k, \dotsc, v_{n-1}}$ with the vector $(v_k, \dotsc, v_{n-1})$ with $v_j \in \frac{1}{2}\ZZ$, so we can evaluate expressions of the form
$
\uj{k,n}{v_k, \dotsc, v_{n-1}} + C \e_i = \uj{k,n}{v_k,\dotsc, v_i + C, \dotsc, v_{n-1}}.
$

Denote by $\one_{[k,n]} := \restrictF{[k,n]}(\one)$.
Let $\PCrysR{k}$ be the connected component of
\[
\vec{0} := \uj{k,n}{0,\dotsc,0} \otimes \varnothing \in \sqrtE_{k,n} \otimes \one_{[k,n]}.
 \]
We prove Lemma~\ref{lemma:root_Kashiwara_embedding} by showing there exists a $\sqgln$-crystal isomorphism
$
\VCrysR{k} \iso \PCrysR{k}
$ 
and then noting that removing the tensor factor $\one_{[k,n]}$ results in a crystal embedding.
To this end, we now describe explicitly the $\sqgln$-crystal structure on $\PCrysR{k}$.

\begin{lemma}
\label{lemma:PCrys_structure}
Fix an element
\[
\vec{v} = \uj{k,n}{v_k, \dotsc, v_{n-1}}\otimes \varnothing \in \sqrtE_{k,n}\otimes \one_{[k,n]},
\] 
and set $v_n=0$ and $v_j = -\infty = \lfloor -\infty \rfloor = \lceil -\infty \rceil$ for all $j < k$.
Then the following properties hold:
\ben

\item[(a)] 
The weight of $\vec{v}$ is
$
\weight(\vec{v}) = \sum_{i=k}^{n-1}\Bigl( \lceil v_i \rceil (\e_i-\e_0) + \lfloor v_i \rfloor (\e_0-\e_{i+1})\Bigr).
$

\item[(b)] If $i \in [k-1]$ then
$e_i \vec{v}=f_i\vec{v}=\zero$
and $ \varepsilon_i(\vec{v})=\varphi_i(\vec{v})=-\infty$.

 \item[(c)] We have $\vec{v}\in\PCrysR{k}$ if and only if $v_i \leq \lceil v_{i+1}\rceil$ for all $i \in [k,n-1]$. 
 
\item[(d)] Assume $\vec{v}\in\PCrysR{k}$.
Then for each $i \in [k,n-1]$ we have
\begin{subequations}
\begin{align}
\label{vec-e-eq}
e_i \vec{v} & = \begin{cases}
\zero &\text{if $v_i= \lceil v_{i+1}\rceil$,} \\
\vec{v} - \frac{1}{2}\e_{i-1} &\text{if $v_{i-1} > v_i<\lceil v_{i+1}\rceil$,}  \\
\vec{v} + \frac{1}{2} \e_{i} &\text{if $v_{i-1} \leq v_i < \lceil v_{i+1}\rceil$,}
\end{cases}
\allowdisplaybreaks \\
\label{vec-f-eq}
f_i \vec{v} & = \begin{cases}
\zero  &\text{if $v_{i-1} =\lceil v_i\rceil $,} \\
\vec{v} + \frac{1}{2} \e_{i-1} &\text{if $v_{i-1} = v_i \notin \ZZ$,} \\ 
\vec{v} - \frac{1}{2} \e_{i} &\text{otherwise,} \\
\end{cases}
\allowdisplaybreaks \\
\label{vec-eps-eq}
\varepsilon_i(\vec{v}) &=  \begin{cases}
-v_i + \lceil v_{i+1}\rceil   & \text{if $v_{i-1} \leq v_i$}, \\
-v_i + \lceil v_{i+1}\rceil +  \frac{1}{2} &\text{if $v_{i-1} > v_i$}, \\
\end{cases}
\allowdisplaybreaks \\
\label{vec-phi-eq}
\varphi_i(\vec{v}) &=
 \begin{cases}
v_k & \text{if } i = k, \\
-\lfloor v_{i-1}\rfloor + v_i  & \text{if $i>k$ and $v_{i-1} \leq v_{i}$}, \\
-v_{i-1} + v_i+\tfrac{1}{2} & \text{if $i>k$ and $v_{i-1} >v_{i}$}.
\end{cases}
\end{align}
\end{subequations}

\een
\end{lemma}

\begin{proof}
Parts (a) and (b) are immediate from the definitions.
Fix some $i \in [k,n-1]$.
To compute the crystal operators $e_i$ and $f_i$, we only need to consider the tensor product $\sqrtE_{i-1} \otimes \sqrtE_i \otimes \sqrtE_{i+1} \otimes \one_{[k,n]}$ (or just $\sqrtE_k \otimes \sqrtE_{k+1} \otimes \one_{[k,n]}$ when $i=k$), where the result follows from analyzing the tensor product rule~\eqref{ef-tensor-eq}.
We present the details when $i>k$ below; the $i=k$ case is similar.

First, let us look at the statistics $\varepsilon_i$ and $\varphi_i$.
We can ignore the other factors as $\varepsilon_i(b) = -\infty$ for all $b \in \sqrtE_j$ with $j \notin\{ i-1, i\}$ and $\langle \weight(b), \alpha_i^{\vee} \rangle = 0$ with $\abs{i - j} > 1$.
The tensor product rule gives
\begin{subequations}
\label{eq:ep_phi_pair}
\begin{align}
\varepsilon_i(\uj{i-1}{v_{i-1}} \otimes \uj{i}{v_i}) & = \max \{ -v_i, v_{i-1} - 2v_i \},  \label{eq:ep_2pair}
\allowdisplaybreaks \\
\varphi_i(\uj{i-1}{v_{i-1}} \otimes \uj{i}{v_i}) & = \max \{ \lceil v_{i-1} \rceil - v_{i-1}, v_i - \lfloor v_{i-1} \rfloor  \},
\allowdisplaybreaks \\
\varepsilon_i(\uj{i-1}{v_{i-1}} \otimes \uj{i}{v_i} \otimes \uj{i+1}{v_{i+1}}) & = \max \{ -v_i, v_{i-1} - 2v_i \} - \lfloor -v_{i+1} \rfloor ,
\allowdisplaybreaks \\
\varphi_i(\uj{i-1}{v_{i-1}} \otimes \uj{i}{v_i} \otimes \uj{i+1}{v_{i+1}}) & = \max \{ \lceil v_{i-1} \rceil - v_{i-1}, v_i - \lfloor v_{i-1} \rfloor  \}.
\end{align}
\end{subequations}
Next, let us discuss the crystal operators.

Let us consider the subcrystal $\cC \subseteq \PCrysR{k}$ that is lower generated by $\vec{0}$, which we claim is the entire set defined by (c).
The fact that $f_i \uj{i-1}{0} = \zero$ implies that $v_i \leq 0$ for all $i = 1, \dotsc, n-1$ for any $\vec{v} \in \cC$.
Let us restrict our attention to $\sqrtE_{i-1} \otimes \sqrtE_i$ as we have $e_i \sqrtE_j = f_i \sqrtE_j = \zero$ for all $j \neq i, i-1$.
The tensor product rule says that $e_i$ (respectively $f_i$) acts on the right factor $\sqrtE_i$ if and only if $v_{i-1} \leq v_i$ (respectively $v_{i-1} < v_i$).
Recall that the action of $e_i$ and $f_i$ on an element of $\sqrtE_i$ is always nonzero.
In order to have $f_i (\uj{i-1}{v_{i-1}} \otimes \uj{i}{v_i}) \neq \zero$, we must have $v_{i-1} < \lceil v_i \rceil$, and so the only time $f_i$ acts on the left factor is when $v_{i-1} = v_i \notin \ZZ$.
Hence, we have $v_{i-1} \leq \lceil v_i \rceil$ for all $\vec{v} \in \cC \subseteq \PCrysR{k}$.
That $\cC$ is precisely the entire set given by (c) follows by induction on $\sum_{i=k}^{n-1} (k-i-1)v_i$, with base case $\vec{0} \in \cC$, using the formulas in (d).
That $\cC = \PCrysR{k}$ follows from the fact that $\cC$ is closed under the $e_i$ crystal operators; this proves (c).

To finish the proof of the formulas for (d) for an element $\vec{v} \in \cC$, let $\delta_i = \max\{0, v_{i-1}- v_i\} $.
We observe from~\eqref{eq:ep_2pair} that $e_i$ acts on the right factor in $\uj{i+1}{v_{i+1}} \otimes \varnothing \in \sqrtE_{i+1} \otimes \one_{[k,n]}$ if and only if
\begin{equation}
\label{eq:ei_one_action}
\max \{ -v_i, v_{i-1} - 2v_i \} = -v_i +\delta_i \leq 
\lfloor -v_{i+1} \rfloor = -\lceil v_{i+1}\rceil.
\end{equation}
Similarly $f_i$ acts on the right factor if and only if $-v_i + \delta_i < -\lceil v_{i+1} \rceil$.
As $\delta_i \in \{0,\frac{1}{2}\}$ and  $ \lceil v_{i+1} \rceil-v_i \geq 0$, equality holds in \eqref{eq:ei_one_action} exactly when $v_i =  \lceil v_{i+1} \rceil \in \ZZ$, in which case we have $v_{i-1}\leq \lceil v_i \rceil = v_i$ and $\delta_i=0$ so
$
e_i (\uj{i+1}{v_{i+1}} \otimes \varnothing) = f_i (\uj{i+1}{v_{i+1}} \otimes \varnothing) = \zero.
$
Hence, we obtain~\eqref{vec-e-eq} and~\eqref{vec-f-eq} as claimed.

The formulas for $\varepsilon_i$ in~\eqref{vec-eps-eq} and $\varphi_i$ in~\eqref{vec-phi-eq} follow from~\eqref{eq:ep_phi_pair}.
\end{proof}

\begin{lemma}\label{lemma:PCrys_structure2}
There is a $\sqgln$-crystal isomorphism
$
\Phi \colon \VCrysR{k} \xrightarrow{\sim} \PCrysR{k}
$
given by the map
\[
\Phi(\overline{\bfa}) = \uj{k,n}{v_k, v_{k+1}, \dotsc, v_{n-1}} \otimes \varnothing
\qquad \text{where }
v_i := -a_{i+1} - \sum_{j=i+2}^n \lfloor a_j \rfloor.
\]
\end{lemma}

\begin{proof}
Note that $\lfloor -x \rfloor = -\lceil x \rceil$ and $\lfloor \lfloor x \rfloor \rfloor = \lceil \lfloor x \rfloor \rceil = \lfloor x \rfloor$.
It is easy to see that $\Phi$ is a bijection since the inverse map is given by $a_{i+1} = \lceil v_{i+1} \rceil - v_i$ and we have $a_i \geq 0$ for all $i$ if and only if $v_i \leq \lceil v_{i+1} \rceil$ for all $i$. 
Next, we see that $\Phi$ preserves weights as starting with the weight map formula in 
Lemma~\ref{lemma:PCrys_structure}
  gives
\begin{align*}
\weight\bigl(\Phi(\overline{\bfa})\bigr) 
& = \sum_{i=k}^{n-1}\left( \bigl( -\lfloor a_{i+1} \rfloor - \sum_{j=i+2}^n \lfloor a_j \rfloor \bigr) (\e_i-\e_0) + \left( -\lceil a_{i+1} \rceil - \sum_{j=i+2}^n \lfloor a_j \rfloor \right) (\e_0-\e_{i+1})\right) \\
& = - \sum_{i=k}^{n-1} \sum_{j=i+1}^n \lfloor a_j \rfloor (\e_i-\e_0) + \sum_{i=k}^{n-1} \left( \lceil a_{i+1} \rceil + \sum_{j=i+2}^n \lfloor a_j \rfloor \right) (\e_{i+1}-\e_0) \\
& = -\sum_{j=k+1}^n \lfloor a_j \rfloor (\e_k-\e_0) + \sum_{i=k+1}^{n} \lceil a_i \rceil (\e_i-\e_0),
\end{align*}
which equals $\weight(\overline{a})$ from  Definition~\ref{def:vec_crystal2}. 
The map $\Phi$ is clearly compatible with all crystal operators and statistics indexed by $i \in [k-1]$. 
On the other hand, both  objects are generated by a single element of weight $0$ and are upper regular when regarded as crystals with index set $I_{\geq k} := [k,n-1]$.
Thus, by Proposition~\ref{prop:morphism_by_ops} we just need to show $\Phi$ commutes with 
$f_i$ for each $i \in I_{\geq k}$.

Let $\vec{v} = \Phi(\overline{\bfa})$.
Clearly $f_k \overline{\bfa} = \overline{\bfa} + \frac{1}{2} \e_{k+1}$ for all $\overline{\bfa} \in \VCrysR{k}$ and $f_k \vec{v} = \vec{v} - \frac{1}{2} \e_k$ for all $\vec{v} \in \PCrysR{k}$, and so $\Phi(f_k \overline{\bfa}) = f_k \Phi(\overline{\bfa})$ as desired.
Now suppose $i > k$, and we proceed using case-by-case analysis.

Note that $f_i \overline{\bfa} = \zero$ if and only if $a_i = 0$ if and only if $\lceil v_i \rceil = v_{i-1}$ if and only if $f_i \vec{v} = \zero$ as desired.
 

Next, we have $a_{i+1} \notin \ZZ$ and $a_i = \frac{1}{2}$ if and only if $v_{i-1} = v_i \notin \ZZ$.
Thus, $f_i \overline{\bfa} = \overline{\bfa} - \frac{1}{2} \e_{i}$ if and only if $f_i \vec{v} = \vec{v} + \frac{1}{2} \e_{i-1}$ as desired.

If we instead assume $a_{i+1} \in \ZZ$ and $a_i > 0$, then we have $v_i - v_{i-1} = a_i > 0$.
Thus $v_i > v_{i-1}$ and $f_i \vec{v} = \vec{v} - \frac{1}{2} \e_i = \Phi(f_i \overline{\bfa}) = \Phi(\overline{\bfa} + \frac{1}{2}\e_{i+1})$ as desired.

Finally, we consider the case $a_{i+1} \notin \ZZ$ and $a_i > \frac{1}{2}$. Then we again have $v_i > v_{i-1}$ and $f_i \vec{v} = \vec{v} - \frac{1}{2} \e_i = \Phi(f_i \overline{\bfa}) = \Phi(\overline{\bfa} - \e_i + \frac{1}{2}\e_{i+1})$. This proves our claim that $\Phi$ is an isomorphism.
\end{proof}

\begin{proof}[Proof of Lemma~\ref{lemma:root_Kashiwara_embedding}]
One can see from the computations in the proof of Lemma~\ref{lemma:PCrys_structure} that the crystal operators of $\PCrysR{k}$ only act on the factor $\one_{[k,n]}$ in cases when
\[ (e_i\uj{k,n}{v_k, \dotsc, v_{n-1}})\otimes \varnothing \notin \PCrysR{k}
\quad\text{or}\quad
(f_i\uj{k,n}{v_k, \dotsc, v_{n-1}})\otimes \varnothing \notin \PCrysR{k}
\] for some $\vec{v} = \uj{k,n}{v_k, \dotsc, v_{n-1}}\otimes \varnothing \in \PCrysR{k}$ and $i \in [n-1]$.
Therefore, the map $ \vec{v} \mapsto  \uj{k,n}{v_k, \dotsc, v_{n-1}}$ is an embedding $\PCrysR{k} \hookrightarrow \sqrtE_{k,n}$ and the result follows from Lemma~\ref{lemma:PCrys_structure2} and Proposition~\ref{prop:trivial_crystal} (after applying the restriction functor).
\end{proof}

The characterization of the elements $\uj{k,n}{v_k, \dotsc, v_{n-1}}\otimes \varnothing \in \PCrysR{k}$ in Lemma~\ref{lemma:PCrys_structure} makes the values $(-2v_k, \dotsc, -2v_{n-1})$ a subset of the so-called pseudo-partitions that have appeared in, e.g.,~\cite{AGS20}.
Additionally, the map $\Phi$ above is a sort of ``dual'' version of the transformation~\eqref{eq:GT_pattern_map}  to marked GT patterns.

Using the embeddings from Lemma~\ref{lemma:root_Kashiwara_embedding} applied to the tensor factors from the isomorphism~\eqref{eq:Lusztig_param}, we obtain our analog of the polyhedral realization for the BZL word.

\begin{theorem}[BZL word polyhedral-type model]
\label{bzl-thm}
There exists an embedding of $\sqgln$-crystals
\begin{equation}
\label{eq:Kashiwara_BZL_embedding}
\Xi \colon \SetTab_n(\infty) \hookrightarrow \sqrtE_{BZL}
\end{equation}
defined by $u_{\infty} \mapsto \uj{BZL}{0,\dotsc,0}$.
The image $\sqrtB(\infty) := \Xi\bigl(\SetTab_n(\infty)\bigr)$ 
 is the set of elements
\[
\uj{\BZL}{v_{n-1,n-1},\  v_{n-2,n-2}, v_{n-2,n-1},\ \dotsc,\  v_{1,1},  v_{1,2},  \dotsc,  v_{1,n-1}} \in  \sqrtE_{\BZL}
\]
that have $v_{j,i} \leq \lceil v_{j,i+1} \rceil$ for all $1\leq j \leq i \leq n-1$, where we set $v_{j,n} = 0$.
\end{theorem}

\begin{proof}
The embedding $ \SetTab_n(\infty) \hookrightarrow \sqrtE_{\BZL}$ is obtained by composing the isomorphism~\eqref{eq:Lusztig_param} with the tensor product of the embeddings from Lemma~\ref{lemma:root_Kashiwara_embedding}.
The description of $\sqrtB(\infty)$ follows from Lemma~\ref{lemma:PCrys_structure}.
\end{proof}

\begin{figure}[t]
\newcommand{\tikzElem}[8]{{\scriptsize\begin{tabular}{|c|}  \hline $#1 #2 #3$ \\  \hline    $ #4$ \\[-10pt]\\  $\ba \varepsilon &= {#5} ,&\hspace{-3mm} {#7} \\ \varphi &={#6},&\hspace{-3mm} {#8}\ea$ \\[-10pt]  \\ \hline \end{tabular}}}
\[
    \begin{tikzpicture}[xscale=2.1, yscale=2.9,>=latex]
\node at (0,0) (A1) {$\tikzElem{0}{0}{0}{1}{0}{0}{0}{0}$};
\node at (-1,-1) (A2) {$\tikzElem{1}{0}{0}{\beta x_3}{0}{0}{\tfrac{1}{2}}{-\tfrac{1}{2}}$};
\node at (0,-1) (B2) {$\tikzElem{0}{0}{1}{\beta x_3}{0}{0}{1}{0}$};
\node at (1,-1) (C2) {$\tikzElem{0}{1}{0}{\beta x_2}{\phantom{-}\tfrac{1}{2}}{-\tfrac{1}{2}}{0}{1}$};
\node at (-2,-2) (A3) {$\tikzElem{2}{0}{0}{x_2^{-1}x_3}{0}{1}{1}{-1}$};
\node at (-1,-2) (B3) {$\tikzElem{1}{1}{0}{\beta^2x_2x_3}{\phantom{-}\tfrac{1}{2}}{-\tfrac{1}{2}}{0}{0}$};
\node at (0,-2) (C3) {$\tikzElem{1}{0}{1}{\beta^2x_3^2}{0}{0}{\tfrac{3}{2}}{-\tfrac{1}{2}}$};
\node at (1,-2) (D3) {$\tikzElem{0}{1}{1}{\beta^2x_2x_3}{\phantom{-}\tfrac{1}{2}}{-\tfrac{1}{2}}{\tfrac{1}{2}}{\tfrac{1}{2}}$};
\node at (2,-2) (E3) {$\tikzElem{0}{2}{0}{x_1^{-1}x_2}{\phantom{-}1}{-1}{0}{1}$};
\node at (-2,-3) (A4) {$\tikzElem{2}{1}{0}{\beta x_3}{\tfrac{1}{2}}{\tfrac{1}{2}}{0}{-1}$};
\node at (-3,-3) (B4) {$\tikzElem{3}{0}{0}{\beta x_2^{-1}x_3^2}{0}{1}{\tfrac{3}{2}}{-\tfrac{3}{2}}$};
\node at (-1,-3) (C4) {$\tikzElem{1}{2}{0}{\beta x_1^{-1}x_2x_3}{\phantom{-}1}{-1}{0}{0}$};
\node at (-0.02,-3) (D4) {$\tikzElem{1}{1}{1}{\beta^3 x_2x_3^2}{\phantom{-}\tfrac{1}{2}}{-\tfrac{1}{2}}{\tfrac{1}{2}}{-\tfrac{1}{2}}$};
\node at (1,-3) (H4) {$\tikzElem{2}{0}{1}{\beta x_2^{-1}x_3^2}{0}{1}{2}{-1}$};
\node at (2,-3) (E4) {$\tikzElem{0}{2}{1}{\beta x_1^{-1}x_2x_3}{\phantom{-}1}{-1}{\tfrac{1}{2}}{\tfrac{1}{2}}$};
\node at (3,-3) (F4) {$\tikzElem{0}{3}{0}{\beta x_1^{-1}x_2^2}{\phantom{-}\tfrac{3}{2}}{-\tfrac{3}{2}}{0}{2}$};
\node at (-3.8,-3.8) (C5) {$\vdots$};
\node at (-3,-3.8) (B5) {$\vdots$};
\node at (-2,-3.8) (A5) {$\vdots$};
\node at (-1.5,-3.8) (D5) {$\vdots$};
\node at (-0.5,-3.8) (E5) {$\vdots$};
\node at (0.5,-3.8) (F5) {$\vdots$};
\node at (1.5,-3.8) (J5) {$\vdots$};  
\node at (2,-3.8) (G5) {$\vdots$};
\node at (3,-3.8) (H5) {$\vdots$};
\node at (3.8,-3.8) (I5) {$\vdots$};
\draw[->,thick,red]  (A1) -- (A2) node[midway,above left,scale=0.75] {$2$};
\draw[->,thick,blue]  (A1) -- (C2) node[midway,above right,scale=0.75] {$1$};
\draw[->,thick,red]  (A2) -- (A3) node[midway,above left,scale=0.75] {$2$};
\draw[->,thick,blue]  (A2) -- (B3) node[midway,left,scale=0.75] {$1$};
\draw[->,thick,red]  (B2) -- (C3) node[midway,right,scale=0.75] {$2$};
\draw[->,thick,blue]  ([xshift=4pt]B2.south) -- ([xshift=-8pt]D3.north) node[midway,below left,scale=0.75] {$1$};
\draw[<-,thick,red]  ([xshift=8pt]B2.south) -- ([xshift=-4pt]D3.north) node[midway,above right,scale=0.75] {$2$};
\draw[->,thick,red]  (C2) -- (D3) node[midway,right,scale=0.75] {$2$};
\draw[->,thick,blue]  (C2) -- (E3) node[midway,above right,scale=0.75] {$1$};
\draw[->,thick,blue]  (A3) -- (A4) node[midway,above left,scale=0.75] {$1$};
\draw[->,thick,red]  (A3) -- (B4) node[midway,above left,scale=0.75] {$2$};
\draw[->,thick,blue]  (B3) -- (C4) node[midway,left,scale=0.75] {$1$};
\draw[->,thick,red]  (B3) -- (D4) node[midway,above right,scale=0.75] {$2$};
\draw[->,thick,blue]  (C3) -- (D4) node[midway,left,scale=0.75] {$1$};
\draw[->,thick,blue]  (D3) -- (E4) node[midway,above right,scale=0.75] {$1$};
\draw[->,thick,red]  (E3) -- (E4) node[midway,above left,scale=0.75] {$2$};
\draw[->,thick,blue]  (E3) -- (F4) node[midway,above right,scale=0.75] {$1$};
\draw[->,thick,red]  (C3) -- (H4) node[midway,above right,scale=0.75] {$2$};  
\draw[->,thick,blue]  (A4) -- (A5) node[midway,above left,scale=0.75] {$1$};
\draw[->,thick,blue]  (B4) -- (B5) node[midway,above left,scale=0.75] {$1$};
\draw[->,thick,blue]  (C4) -- (D5) node[midway,above left,scale=0.75] {$1$};
\draw[->,thick,blue]  (D4) -- (E5) node[midway,above left,scale=0.75] {$1$};
\draw[->,thick,blue]  (E4) -- (H5) node[midway,above right,scale=0.75] {$1$};
\draw[->,thick,blue]  (F4) -- (I5) node[midway,above right,scale=0.75] {$1$};
\draw[->,thick,blue]  (H4) -- (F5) node[midway,above left,scale=0.75] {$1$};  
\draw[->,thick,red]  (A4) -- (B5) node[midway,above left,scale=0.75] {$2$};
\draw[->,thick,red]  (B4) -- (C5) node[midway,above left,scale=0.75] {$2$};
\draw[->,thick,red]  (C4) -- (E5) node[midway,above right,scale=0.75] {$2$};
\draw[->,thick,red]  (D4) -- (F5) node[midway,above right,scale=0.75] {$2$};
\draw[->,thick,red]  (E4) -- (G5) node[midway,right,scale=0.75] {$2$};
\draw[->,thick,red]  (F4) -- (H5) node[midway,right,scale=0.75] {$2$};
\draw[->,thick,red]  (H4) -- (J5) node[midway,right,scale=0.75] {$2$};  
     \end{tikzpicture}
\]
\caption{Crystal graph of $\sqrtB(\infty)$ for $n=3$.
In each boxed vertex,
the word ``$ABC$'' stands for 
$b= \uj{2}{-\tfrac{A}{2}} \otimes \uj{1}{-\tfrac{B}{2}}\otimes \uj{2}{-\tfrac{C}{2}} = \uj{\BZL}{-\tfrac{A}{2},-\tfrac{B}{2},-\tfrac{C}{2}} \in \sqrtB(\infty)$,
the monomial on the second line is $x^{\weight(b)}$,
and the last two lines show the respective values of $\varepsilon_1(b)$, $\varepsilon_2(b)$ and $\varphi_1(b)$, $\varphi_2(b)$.
}
\label{binfty-fig}
\end{figure}

Figure~\ref{binfty-fig} shows part of the crystal graph of $\sqrtB(\infty)$ for $n=3$.
Compare this with Figure~\ref{fig:gl3mlt} for an example of the isomorphism $\SetTab_n(\infty) \iso \sqrtB(\infty)$ from Theorem~\ref{bzl-thm}.

Although $\sqrtB(\infty) $ is not a full subcrystal of $\sqrtE_{\BZL} $, we may now prove the following.

\begin{corollary}
\label{cor:cutting_out}
The $\sqgln$-crystal $\sqrtB(\infty)$ is upper regular and isomorphic to the connected component of $\sqrtE_{\BZL} \otimes \one$ containing $\uj{\BZL}{0}\otimes \varnothing$ via the map $b \mapsto b \otimes \varnothing$.
\end{corollary}

\begin{proof}
Suppose $\cB$ is an abstract $\sqgln$-crystal such that for all $b \in \cB$ and $i \in I$ we have $\varepsilon_i(b) \geq0$.
Then the definition of $\otimes$ implies that in $\cB\otimes \one$, we have
\[
\weight(b\otimes \varnothing) = \weight(b),
\qquad
f_i(b\otimes\varnothing) = (f_i b) \otimes \varnothing,
\qquad
\varepsilon_i(b\otimes \varnothing) = \varepsilon_i(b),
\qquad
\varphi_i(b\otimes \varnothing) = \varphi_i(b).
\]
Furthermore, it holds that $e_i(b\otimes \varnothing) = (e_i b) \otimes \varnothing$ when $\varepsilon_i(b)>0$ or $e_i(b\otimes \varnothing)=\zero$ when $\varepsilon_i(b)=0$. 
Consequently, if $\cB$ is upper regular, then $\cB \otimes \one \iso \cB$ by the isomorphism $b \otimes \varnothing \mapsto b$.
(Equivalently, $\one$ is a right unit object in the category of upper regular abstract crystals.)

It follows that if $\cB\subseteq \cC$ is an upper regular (but not necessarily full) subcrystal such that $\cB\sqcup \{\zero\}$ is closed under all $f_i$ crystal operators for $\cC$, then $\cB \otimes \one$ is a full subcrystal of $\cC\otimes \one$ that is isomorphic to $\cB$ via the map $b\mapsto b\otimes \varnothing$ (Proposition~\ref{prop:trivial_crystal}).
This applies to the connected crystal  $\cB = \sqrtB(\infty) \subseteq \cC = \sqrtE_{\BZL}$ since $\sqrtB(\infty) \iso \SetTab_n(\infty) $ is upper regular by Corollary~\ref{cor:inf_upper_reg}
and closed under all $f_i$ operators by Proposition~\ref{mala-lem}.
\end{proof}

\begin{example}
\label{ex:not_polyhedral}
Unlike in the classical $\gl_n$ case, the set of elements in $\PCrysR{k}$, and hence the image of $\SetTab_n(\infty)$ in $\sqrtE_{\BZL}$, does not always correspond to the half-integer points inside some polyhedron.
Consider the $\sqrt{\gl_3}$-crystal $\mathsf{LP}_{1,3}^{\halfpower}$, which contains $\uj{1,3}{0,-\tfrac{1}{2}}$ and $\uj{1,3}{-1,-\tfrac{3}{2}}$ but not $\uj{1,3}{-\tfrac{1}{2},-1}$.
If we plot these points $\uj{1,3}{x,y}$, we obtain
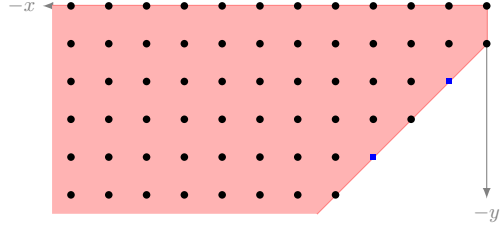
\[
\begin{tikzpicture}[scale=.5]
\draw[->,gray,>=latex] (0,0) -- (-11.75,0) node[anchor=east,scale=.7] {$-x$};
\draw[->,gray,>=latex] (0,0) -- (0,-5.1) node[anchor=north,scale=.7] {$-y$};
\draw[-,red!50,thick] (-11.5,0) -- (0,0) -- (0,-1) -- (-4.5,-5.5);
\fill[red!30] (-11.5,0) -- (0,0) -- (0,-1) -- (-4.5,-5.5) -- (-11.5,-5.5) -- (-11.5,0);
\foreach \y in {-2,...,0} {
  \foreach \x in {-5,...,\y} {
    \fill (2*\x,2*\y-1) circle (.1);
    \fill (2*\x,2*\y) circle (.1);
    \fill (2*\x-1,2*\y-1) circle (.1);
    \fill (2*\x-1,2*\y) circle (.1);
  }
}
\fill[blue] ++(-1,-2) + (.07,.07) rectangle ++(-.07,-.07);
\fill[blue] ++(-3,-4) + (.07,.07) rectangle ++(-.07,-.07);
\end{tikzpicture}
\]
where the small squares are the missing lattice points from the polyhedron.
\end{example}

\begin{definition}
For any sequence $J = (j_k \in I)_{k=1}^{\ell}$ and $\lambda \in P^+$ (respectively $\lambda = \infty$), let $\sqrtB_J(\lambda)$ be the subcrystal of $\cR_{\lambda} \otimes \sqrtE_J$ (respectively $\sqrtE_J$) that is lower generated by $\{ r_{\lambda} \otimes \uj{J}{0, \dotsc, 0} \}$ (respectively $\{ \uj{J}{0,\dotsc,0} \}$).
As a special case, set
\be
\sqrtB(\lambda) := \sqrtB_{\BZL}(\lambda)
\ee
where $\BZL=(n-1,n-2,n-1,\dots,1,2,\dots,n-1)$ is the reduced word for $w_0 \in S_n$ from~\eqref{bzl-eq}.
\end{definition}

The following is immediate from Theorem~\ref{bzl-thm} and Corollary~\ref{cor:finite_recovery}.

\begin{corollary}\label{rb-B-cor}
For all $\lambda \in P^+$ it holds that $\cB^{\halfpower}(\lambda) \iso \SetTab_n(\lambda)$.
\end{corollary}

See Figure~\ref{rb-fig} for an example of the isomorphism in this corollary. 

\begin{figure}[t]
\def\ytabd{\ytabc{0.4cm}}
\[
    \begin{tikzpicture}[xscale=2.3, yscale=1.75,>=latex,baseline=(D0.base)]
\node at (1,0) (A1) {$\ytabd{1 \\ 2}$};
\node at (2,0) (A2) {$\ytabd{1 \\ 23}$};
\node at (1,-1) (B1) {$\ytabd{1 \\ \scalebox{.8}{234}}$};
\node at (2,-1) (B2) {$\ytabd{1 \\ 3}$};
\node at (0,-2) (C0) {$\ytabd{1 \\ 24}$};
\node at (1,-2) (C1) {$\ytabd{1 \\ 34}$};
\node at (2,-2) (C2) {$\ytabd{12 \\ 3}$};
\node at (0,-3) (D0) {$\ytabd{1 \\ 4}$};
\node at (1,-3) (D1) {$\ytabd{12 \\ 34}$};
\node at (2,-3) (D2) {$\ytabd{2 \\ 3}$};
\node at (0,-4) (E0) {$\ytabd{13 \\ 4}$};
\node at (1,-4) (E1) {$\ytabd{12 \\ 4}$};
\node at (2,-4) (E2) {$\ytabd{2 \\ 34}$};
\node at (1,-5) (F1) {$\ytabd{\scalebox{.8}{123} \\ 4}$};
\node at (2,-5) (F2) {$\ytabd{2 \\ 4}$};
\node at (1,-6) (G1) {$\ytabd{3 \\ 4}$};
\node at (2,-6) (G2) {$\ytabd{23 \\ 4}$};
\draw[->,thick,red]  (A1) -- (A2) node[midway,above,scale=0.75] {$2$};
\draw[->,thick,red]  (A2) -- (B2) node[midway,right,scale=0.75] {$2$};
\draw[->,thick,gray]  (A2) -- (B1) node[midway,below right,scale=0.75] {$3$};
\draw[->,thick,gray]  (B1) -- (C0) node[midway,below right,scale=0.75] {$3$};
\draw[->,thick,red]  (B1) -- (C1) node[midway,right,scale=0.75] {$2$};
\draw[->,thick,gray]  (B2) -- (C1) node[midway,below right,scale=0.75] {$3$};
\draw[->,thick,blue]  (B2) -- (C2) node[midway,right,scale=0.75] {$1$};
\draw[->,thick,red]  ([xshift=4pt]C0.north) -- ([xshift=0pt]B1.west) node[midway,above left,scale=0.75] {$2$};
\draw[->,thick,gray]  (C1) -- (D0) node[midway,below right,scale=0.75] {$3$};
\draw[->,thick,gray]  (C2) -- (D1) node[midway,below right,scale=0.75] {$3$};
\draw[->,thick,blue]  (C1) -- (D1) node[midway,right,scale=0.75] {$1$};
\draw[->,thick,blue]  (C2) -- (D2) node[midway,right,scale=0.75] {$1$};
\draw[->,thick,gray]  (D1) -- (E1) node[midway,right,scale=0.75] {$3$};
\draw[->,thick,gray]  (D2) -- (E2) node[midway,right,scale=0.75] {$3$};
\draw[->,thick,blue]  (D0) -- (E1) node[midway,above right,scale=0.75] {$1$};
\draw[->,thick,blue]  (D1) -- (E2) node[midway,above right,scale=0.75] {$1$};
\draw[->,thick,blue]  (E0) -- (F1) node[midway,above right,scale=0.75] {$1$};
\draw[->,thick,blue]  (E1) -- (F2) node[midway,above right,scale=0.75] {$1$};
\draw[->,thick,red]  (E1) -- (F1) node[midway,right,scale=0.75] {$2$};
\draw[->,thick,gray]  (E2) -- (F2) node[midway,right,scale=0.75] {$3$};
\draw[->,thick,red]  ([xshift=0pt]F1.west) -- ([xshift=4pt]E0.south) node[midway,below left,scale=0.75] {$2$};
\draw[->,thick,blue]  (F1) -- (G2) node[midway,above right,scale=0.75] {$1$};
\draw[->,thick,red]  (F2) -- (G2) node[midway,right,scale=0.75] {$2$};
\draw[->,thick,red]  (G2) -- (G1) node[midway,above,scale=0.75] {$2$};
     \end{tikzpicture}
\hspace{60pt}
    \begin{tikzpicture}[xscale=2.3, yscale=1.75,>=latex,baseline=(D0.base)]
\node at (1,0) (A1) {$000000$};
\node at (2,0) (A2) {$010000$};
\node at (1,-1) (B1) {$011000$};
\node at (2,-1) (B2) {$020000$};
\node at (0,-2) (C0) {$001000$};
\node at (1,-2) (C1) {$021000$};
\node at (2,-2) (C2) {$020100$};
\node at (0,-3) (D0) {$022000$};
\node at (1,-3) (D1) {$021100$};
\node at (2,-3) (D2) {$020200$};
\node at (0,-4) (E0) {$022010$};
\node at (1,-4) (E1) {$022100$};
\node at (2,-4) (E2) {$021200$};
\node at (1,-5) (F1) {$022110$};
\node at (2,-5) (F2) {$022200$};
\node at (1,-6) (G1) {$022220$};
\node at (2,-6) (G2) {$022210$};
\draw[->,thick,red]  (A1) -- (A2) node[midway,above,scale=0.75] {$2$};
\draw[->,thick,red]  (A2) -- (B2) node[midway,right,scale=0.75] {$2$};
\draw[->,thick,gray]  (A2) -- (B1) node[midway,below right,scale=0.75] {$3$};
\draw[->,thick,gray]  (B1) -- (C0) node[midway,below right,scale=0.75] {$3$};
\draw[->,thick,red]  (B1) -- (C1) node[midway,right,scale=0.75] {$2$};
\draw[->,thick,gray]  (B2) -- (C1) node[midway,below right,scale=0.75] {$3$};
\draw[->,thick,blue]  (B2) -- (C2) node[midway,right,scale=0.75] {$1$};
\draw[->,thick,red]  ([xshift=0pt]C0.north) -- ([xshift=-8pt]B1.south) node[midway,above left,scale=0.75] {$2$};
\draw[->,thick,gray]  (C1) -- (D0) node[midway,below right,scale=0.75] {$3$};
\draw[->,thick,gray]  (C2) -- (D1) node[midway,below right,scale=0.75] {$3$};
\draw[->,thick,blue]  (C1) -- (D1) node[midway,right,scale=0.75] {$1$};
\draw[->,thick,blue]  (C2) -- (D2) node[midway,right,scale=0.75] {$1$};
\draw[->,thick,gray]  (D1) -- (E1) node[midway,right,scale=0.75] {$3$};
\draw[->,thick,gray]  (D2) -- (E2) node[midway,right,scale=0.75] {$3$};
\draw[->,thick,blue]  (D0) -- (E1) node[midway,above right,scale=0.75] {$1$};
\draw[->,thick,blue]  (D1) -- (E2) node[midway,above right,scale=0.75] {$1$};
\draw[->,thick,blue]  (E0) -- (F1) node[midway,above right,scale=0.75] {$1$};
\draw[->,thick,blue]  (E1) -- (F2) node[midway,above right,scale=0.75] {$1$};
\draw[->,thick,red]  (E1) -- (F1) node[midway,right,scale=0.75] {$2$};
\draw[->,thick,gray]  (E2) -- (F2) node[midway,right,scale=0.75] {$3$};
\draw[->,thick,red]  ([xshift=-8pt]F1.north) -- ([xshift=0pt]E0.south) node[midway,below left,scale=0.75] {$2$};
\draw[->,thick,blue]  (F1) -- (G2) node[midway,above right,scale=0.75] {$1$};
\draw[->,thick,red]  (F2) -- (G2) node[midway,right,scale=0.75] {$2$};
\draw[->,thick,red]  (G2) -- (G1) node[midway,above,scale=0.75] {$2$};
     \end{tikzpicture}
\]
\caption{Crystal graphs of $\SetTab_4(\Lambda_2)$ (left) and $\sqrtB(\Lambda_2)$ for $n = 4$ (right), where the word $ABCDEF$ represents the element $r_{\Lambda_2} \otimes \uj{\BZL}{-\tfrac{A}{2},-\tfrac{B}{2},-\tfrac{C}{2},-\tfrac{D}{2},-\tfrac{E}{2},-\tfrac{F}{2}}$.}
\label{rb-fig}
\end{figure}

Now we make precise how to obtain a sequence of Kashiwara embeddings corresponding to an initial segment of the BZL word   for $w_0 \in S_n$ defined in~\eqref{bzl-eq}.
This comes from splitting the embedding~\eqref{eq:Kashiwara_BZL_embedding} at some point in the middle of the BZL word.

\begin{theorem}[BZL Kashiwara-type embedding]
\label{thm:Kashiwara_embedding}
Let $J$ be an initial prefix of $\BZL$ from~\eqref{bzl-eq}, and let $\overline{J}$ be such that $J \overline{J} = \BZL$.
Define $\overline{w} \in S_n$ to be the inverse of the element with reduced word $\overline{J}$.
Then there exists a $\sqgln$-crystal embedding
$
\SetTab_n(\infty) \hookrightarrow \sqrtE_J \otimes \SetTab_n(\infty)_{\overline{w}}
$
sending $u_{\infty} \mapsto \uj{J}{0, \dotsc, 0} \otimes u_{\infty}$.
\end{theorem}

\begin{proof}
Fix $1\leq j \leq k \leq n$.
By Proposition~\ref{vcrys-prop},
the map $a\mapsto a\otimes \varnothing$ is an isomorphism $\VCrysR{j} \xrightarrow{\sim} \VCrysR{j}\otimes \one_{[k,n]}$.
Tensoring with the embedding in Lemma~\ref{lemma:root_Kashiwara_embedding}
gives an embedding
\[
\VCrysR{j} \hookrightarrow \sqrtE_{j,n} \otimes\one_{[k,n]}=
\sqrtE_{(j,\dots,k-1)} \otimes  \sqrtE_{k,n}\otimes \one_{[k,n]}
\]
which factors through $\sqrtE_{(j,\dots,k-1)} \otimes \PCrysR{k}$ by Lemma~\ref{lemma:PCrys_structure}.
Applying the inverse of the isomorphism in Lemma~\ref{lemma:PCrys_structure2} to the right factor 
gives an embedding
$\VCrysR{j}\hookrightarrow \sqrtE_{(j,\dotsc,k-1)} \otimes \VCrysR{k}$.
This is the claimed Kashiwara-type embedding for a single row, and the theorem now follows from~\eqref{eq:Lusztig_param}.
\end{proof}

\subsection{String parameterization}

We now give our last description of $\SetTab_n(\infty)$ using the string data parameterization of a crystal studied in~\cite{BZ93,BZ01,Kashiwara93,Littelmann98}.
The construction here actually describes the elements of $\SetTab_n(\lambda)$ for any $\lambda \in P^+\sqcup\{\infty\}$.
This parameterization depends on a reduced word for $w_0 \in S_n$, but 
unlike in previous sections this word can be arbitrary (rather than just the BZL word).

\begin{definition}[String data]
\label{string-data-def}
Let $\cB$ be an upper regular $\sqgln$-crystal.
Fix a reduced expression $\ii = (i_1, \dotsc, i_N)$ for $w_0 \in S_n$.
The \defn{$\ii$-string data} $\Sigma = (\Sigma_1, \dotsc, \Sigma_N)$ for $b \in \cB$ is defined recursively by
\[
\Sigma_{j-1}:= \varepsilon_{i_{j-1}}(e_{i_j}^{2\Sigma_j} e_{i_{j+1}}^{2\Sigma_{j+1}} \cdots e_{i_N}^{2\Sigma_N} b)
\qquad\text{with }\Sigma_N := \varepsilon_{i_N}(b).
\]
When $n=1$, the reduced expression $\ii$ is empty  and we define $\Sigma$ also to be the empty sequence.
\end{definition}

This definition is based on similar constructions in \cite{BZ93,Kashiwara93}.
We briefly remark that it is sometimes interesting to consider string data corresponding to an infinite index sequence $\ii$ in which each $i \in [n-1]$ occurs infinitely often~\cite{BZ01,Littelmann98}.

Our main result concerning Definition~\ref{string-data-def} is to show that for $\SetTab_n(\lambda)$ for $\lambda \in P^+ \sqcup \{\infty\}$  each element can be recovered from its string data, and that applying the sequence of raising operators indexed by $\Sigma$ always returns to the highest weight element.
The latter property is not 
guaranteed (and actually fails for connected polynomial $\sqgln$-crystals in general), as 2-cycles can occur in the relevant crystals as a consequence of 
our choice of $\bal$ data~\eqref{eq:sqrt_roots}.

\begin{proposition}
\label{prop:string_data_exists}
Fix a reduced expression $\ii = (i_1, \dotsc, i_N)$ for $w_0 \in S_n$ and $\lambda \in P^+ \sqcup \{\infty\}$.
Then each $T \in \SetTab_n(\lambda)$ is uniquely determined by its $\ii$-string data $\Sigma = (\Sigma_1, \dotsc, \Sigma_N)$, and it holds that
$
e_{i_1}^{2\Sigma_1} \cdots e_{i_N}^{2\Sigma_N} T = u_{\lambda}.
$
\end{proposition}

\begin{proof}
Suppose $T \in \SetTab_n(\lambda)$ has $\ii$-string data $\Sigma = (\Sigma_1, \dotsc, \Sigma_N)$.
Then, in view of Corollary~\ref{cor:string_decomp}, it is clear from Theorems~\ref{thm:crystal_SVT} and \ref{dem-thm} that $e_{i_1}^{2\Sigma_1}  \cdots e_{i_N}^{2\Sigma_N} T = u_{\lambda}$.
This identity shows that $T$ is uniquely determined by $\Sigma$ as the raising and lowering operators are partial inverses.
\end{proof}

For examples of the string data parameterization of $\SetTab_n(\lambda)$, see Figure~\ref{fig:gl3string}.

\begin{remark}
Proposition~\ref{prop:string_data_exists} does not hold for arbitrary connected polynomial $\sqgln$-crystals, as such crystals can have multiple highest weight elements, and even when there is a unique highest weight element $u_{\lambda}$
the identity $e_{i_1}^{2\Sigma_1} e_{i_2}^{2\Sigma_2} \cdots e_{i_N}^{2\Sigma_N}T=u_\lambda $ may fail to hold.

For certain choices of reduced words $\ii$, if we  also include the highest weight element obtained after applying $e_{i_1}^{2\Sigma_1} e_{i_2}^{2\Sigma_2} \cdots e_{i_N}^{2\Sigma_N}$ as part of the $\ii$-string data, then we can get a unique parameterization of any connected crystal $\cB \in \sqrt{\sC}_{\poly}$. 
\cite[Thm.~2.21]{MTY} shows that this works for the reversed dual BZL word $\ii = (1,2,\dots,n-1,\dots,1,2,3,1,2,1)$.
Just recording the weight of the highest weight element reached does not lead to a unique parametrization, as connected polynomial $\sqgln$-crystals can have multiple highest weight elements with the same weight.
\end{remark}

We can explicitly compute the string data for $\SetTab_n(\lambda)$ when $\ii = \BZL$ and $\lambda$ is a one-row partition.
In this case, only the rightmost part $(1, 2, 3,\dotsc,  n-1)$ of the BZL word contributes nonzero terms to $\Sigma$.
Recall from Proposition~\ref{prop:single_row_crystal} that  $\VCrys{m} \iso \SetTab_n(m\Lambda_1)$ when $m$ is finite.

\begin{proposition}
Fix $\bfa=(a_1,a_2,\dots,a_n) \in \VCrys{m}$ for $m \in \ZZ_{>0} \sqcup \{\infty\}$, taking $a_1=\infty$ when $m = \infty$.
For $i \in [n-1]$, let $b_i = \frac{1}{2}$ if there exists $k \in [i+2,n]$  such that $a_{i+1} = \cdots = a_{k-1} = 0$ and $a_k \notin \ZZ$; otherwise set $b_i = 0$.
Also let $c_i = \frac{1}{2}$ if  $a_i = 0$ 
and either $b_i = \frac{1}{2}$ or $a_{i+1} \notin \ZZ$;
otherwise set $c_i = 0$.
Then the BZL string data of $\bfa$ has the form $(0, \dotsc, 0, \Sigma_1, \Sigma_2, \dotsc, \Sigma_{n-1})$, where
\[
\Sigma_i = a_{i+1} + \sum_{j=i+2}^n \lfloor a_j \rfloor + b_i + c_i 
\quad\text{for each }i \in[n-1].
\]
\end{proposition}

\begin{proof}
This follows by downward induction on $i$  using~\eqref{eq:vec_e} and~\eqref{eq:vec_ep}.
Indeed, for any $\bfa \in \VCrys{m}$, we have
$
\varepsilon_{n-1}(\bfa) =  a_{n}  + c_{n-1}
$
and
$
e_{n-1}^{2\varepsilon_{n-1}(\bfa)} \bfa = \bfa + (\lfloor a_{n} \rfloor + c_{n-1}) \e_{n-1} - a_{n} \e_{n},
$
where the last expression is an element in $\restrictF{[1,n-1]}\bigl(\VCrys{m}\bigr)$.
The claim follows by noting that the role of the $b_i$ is to carry forward the $c_i$ contribution.
\end{proof}

For more general partitions $\lambda$, it seems difficult to find closed formulas for the string data of the elements of $\SetTab_n(\lambda)$.

\begin{figure}[t]
\[
\newcommand{\st}[2]{\substack{\displaystyle #1 \\[1pt] \displaystyle #2}}
    \begin{tikzpicture}[xscale=2.1, yscale=2.3,>=latex] 
\node at (0,0) (A1) {$\st{000}{000}$};
\node at (-1,-1) (A2) {$\st{001}{010}$};
\node at (0,-1) (B2) {$\st{012}{120}$};
\node at (1,-1) (C2) {$\st{010}{001}$};
\node at (-2,-2) (A3)  {$\st{002}{020}$};
\node at (-1,-2) (B3) {$\st{110}{011}$};
\node at (0,-2) (C3) {$\st{013}{130}$};
\node at (1,-2) (D3) {$\st{011}{121}$};
\node at (2,-2) (E3) {$\st{020}{002}$};
\node at (-2,-3) (A4) {$\st{210}{021}$};
\node at (-3,-3) (B4) {$\st{003}{030}$};
\node at (-1,-3) (C4) {$\st{120}{012}$};
\node at (0,-3) (D4) {$\st{111}{131}$};
\node at (1,-3) (H4) {$\st{014}{140}$};  
\node at (2,-3) (E4) {$\st{021}{122}$};
\node at (3,-3) (F4) {$\st{030}{003}$};
\node at (-3.8,-3.8) (C5) {$\vdots$};
\node at (-3,-3.8) (B5) {$\vdots$};
\node at (-2,-3.8) (A5) {$\vdots$};
\node at (-1.5,-3.8) (D5) {$\vdots$};
\node at (-0.5,-3.8) (E5) {$\vdots$};
\node at (0.5,-3.8) (F5) {$\vdots$};
\node at (1.5,-3.8) (J5) {$\vdots$};  
\node at (2,-3.8) (G5) {$\vdots$};
\node at (3,-3.8) (H5) {$\vdots$};
\node at (3.8,-3.8) (I5) {$\vdots$};
\draw[->,thick,red]  (A1) -- (A2) node[midway,above left,scale=0.75] {$2$};
\draw[->,thick,blue]  (A1) -- (C2) node[midway,above right,scale=0.75] {$1$};
\draw[->,thick,red]  (A2) -- (A3) node[midway,above left,scale=0.75] {$2$};
\draw[->,thick,blue]  (A2) -- (B3) node[midway,left,scale=0.75] {$1$};
\draw[->,thick,red]  (B2) -- (C3) node[midway,right,scale=0.75] {$2$};
\draw[->,thick,blue]  ([xshift=4pt]B2.south) -- ([xshift=-8pt]D3.north) node[midway,below left,scale=0.75] {$1$};
\draw[<-,thick,red]  ([xshift=8pt]B2.south) -- ([xshift=-4pt]D3.north) node[midway,above right,scale=0.75] {$2$};
\draw[->,thick,red]  (C2) -- (D3) node[midway,right,scale=0.75] {$2$};
\draw[->,thick,blue]  (C2) -- (E3) node[midway,above right,scale=0.75] {$1$};
\draw[->,thick,blue]  (A3) -- (A4) node[midway,above left,scale=0.75] {$1$};
\draw[->,thick,red]  (A3) -- (B4) node[midway,above left,scale=0.75] {$2$};
\draw[->,thick,blue]  (B3) -- (C4) node[midway,left,scale=0.75] {$1$};
\draw[->,thick,red]  (B3) -- (D4) node[midway,above right,scale=0.75] {$2$};
\draw[->,thick,blue]  (C3) -- (D4) node[midway,left,scale=0.75] {$1$};
\draw[->,thick,blue]  (D3) -- (E4) node[midway,above right,scale=0.75] {$1$};
\draw[->,thick,red]  (E3) -- (E4) node[midway,above left,scale=0.75] {$2$};
\draw[->,thick,blue]  (E3) -- (F4) node[midway,above right,scale=0.75] {$1$};
\draw[->,thick,red]  (C3) -- (H4) node[midway,above right,scale=0.75] {$2$};  
\draw[->,thick,blue]  (A4) -- (A5) node[midway,above left,scale=0.75] {$1$};
\draw[->,thick,blue]  (B4) -- (B5) node[midway,above left,scale=0.75] {$1$};
\draw[->,thick,blue]  (C4) -- (D5) node[midway,above left,scale=0.75] {$1$};
\draw[->,thick,blue]  (D4) -- (E5) node[midway,above left,scale=0.75] {$1$};
\draw[->,thick,blue]  (E4) -- (H5) node[midway,above right,scale=0.75] {$1$};
\draw[->,thick,blue]  (F4) -- (I5) node[midway,above right,scale=0.75] {$1$};
\draw[->,thick,blue]  (H4) -- (F5) node[midway,above left,scale=0.75] {$1$};  
\draw[->,thick,red]  (A4) -- (B5) node[midway,above left,scale=0.75] {$2$};
\draw[->,thick,red]  (B4) -- (C5) node[midway,above left,scale=0.75] {$2$};
\draw[->,thick,red]  (C4) -- (E5) node[midway,above right,scale=0.75] {$2$};
\draw[->,thick,red]  (D4) -- (F5) node[midway,above right,scale=0.75] {$2$};
\draw[->,thick,red]  (E4) -- (G5) node[midway,right,scale=0.75] {$2$};
\draw[->,thick,red]  (F4) -- (H5) node[midway,right,scale=0.75] {$2$};
\draw[->,thick,red]  (H4) -- (J5) node[midway,right,scale=0.75] {$2$};  
     \end{tikzpicture}
\]
\caption{String data $(\Sigma_1, \Sigma_2, \Sigma_3)$ with respect to the BZL word $w_0 = s_2 s_1 s_2$ (above) and dual BZL word $w_0 = s_1 s_2 s_1$ (below) for the $\sqrt{\gl_3}$-crystal $\SetTab_3(\infty)$ up to depth at least 3. We write the string data as the word $A_1 A_2 A_3$ with $A_i = 2\Sigma_i$.}
\label{fig:gl3string}
\end{figure}

\section{Idiosyncrasies}
\label{sec:oddities}

We conclude with some additional noteworthy  observations regarding $\sqgln$-crystals.

\subsection{Constructions using non-BZL words}
\label{sec:nonBZL}

As mentioned in Remark~\ref{rem:braid_rels}, in the classical version of Theorem~\ref{bzl-thm}, the BZL word can be replaced by any reduced word for $w_0 \in S_n$.
This is no longer the case in the square root setting as our construction of $\sqrtB(\infty) \iso \SetTab_n(\infty) $ is quite sensitive to the choice of the BZL word for $w_0 \in S_n$.
One can change the BZL word in Theorem~\ref{bzl-thm} to any word in its \defn{restricted commutation class}, which is generated by the transformations $(\dots,i,j,\dots) \leftrightarrow (\dots,j,i,\dots)$ when $\abs{i-j} > 2$ as $\sqrtE_i\otimes \sqrtE_j \iso \sqrtE_j\otimes \sqrtE_i$ as noted in Remark~\ref{rem:braid_rels}.
As per Remark~\ref{rem:square_func_elem}, this difference arises as our looped-path crystals $\sqrtE_{j}$ are not exact analogs of the elementary $\gl_n$-crystals $\cE_j$ as $\defectF{2}(\sqrtE_j) \not\iso \cE_j$ when $j<n-1$.
This already hints that one should not expect the classical braid relation isomorphisms. 

Nevertheless, one might ask if there are versions of Theorem~\ref{bzl-thm} with the BZL word replaced by other nice reduced expressions for $w_0 \in S_n$.
For example, the \defn{dual BZL word} is given by applying the Dynkin diagram automorphism of $S_n$
to the BZL word:
\begin{align*}
w_0 & = s_1 (s_2 s_1) \cdots (s_{n-2} \cdots s_2 s_1) (s_{n-1} \cdots s_2 s_1) \in S_n,
\\
\BZL^* & := (1, 2,1, \dotsc, n-2, \dotsc, 2, 1, n-1, \dotsc, 2, 1).
\end{align*}
However, we will argue there is no reasonable analog of the Kashiwara embedding for $i = 1$ for the $\sqrt{\gl_3}$-crystal $\SetTab_3(\infty)$.
This suggests that \emph{only} the (restricted commutation class of) the BZL word yields a Kashiwara-type embedding and corresponds to a nice construction for $\sqgln$-crystals.

Before proceeding, let us motivate why this is actually not unexpected for $\SetTab_n(\infty)$ (and subsequently all polynomial $\sqgln$-crystals).
For $n = 3$, the BZL word and the dual BZL are related by the $\gl_3$ root system symmetry interchanging $\alpha_1 \leftrightarrow \alpha_2$.
On the other hand, there is an asymmetry between the $\sqrt{\gl_3}$-crystal statistics in $\SetTab_3(\infty)$, in the sense that the root system symmetry does not induce a twisted crystal automorphism as it would interchange $\varepsilon_1 \leftrightarrow \varepsilon_2$ and $\varphi_1 \leftrightarrow \varphi_2$.
This is effectively a consequence of the asymmetry from splitting $\alpha_i$ into $\alpha_i^{(0)}$ and $\alpha_i^{(1)}$.

\begin{example}
\label{ex:non_twisted_auto}
In $\SetTab_3(\infty)$, consider the elements $b_{12} = f_1 f_2 u_{\infty}$ and $b_{21} = f_2 f_1 u_{\infty}$.
We can see from Figure~\ref{binfty-fig} that
$
\varepsilon_1(b_{12}) = \frac{1}{2}
\neq
\varepsilon_2(b_{12}) = 0
$
and
$
\varepsilon_1(b_{21}) = \varepsilon_2(b_{21}) = \frac{1}{2}.
$
Clearly the automorphism interchanging $1 \leftrightarrow 2$ does \emph{not} induce a twisted automorphism on the crystal $\SetTab_3(\infty)$ as it would need to interchange $b_{12} \leftrightarrow b_{21}$.
Similarly, we have $f_1 f_2^2 f_1 u_{\infty} = f_2 f_1 u_{\infty}$ but $f_2 f_1^2 f_2 u_{\infty} \neq f_1 f_2 u_{\infty}$.
\end{example}

\begin{proposition}
\label{prop:no_kashiwara_here}
Suppose $n\geq 3$ and $\cB_1$ is a $\sqgln$-crystal with a $\sqgln$-crystal embedding 
\[\Upsilon_1 \colon \SetTab_n(\infty) \hookrightarrow \cB_1 \otimes \SetTab_n(\infty).\] 
Then there is no element $b_0 \in \cB_1$ satisfying $\Upsilon_1(f_1^k u_{\infty}) = (f_1^k b_0) \otimes u_{\infty}$ for all $k \in \NN$ and
\be\label{b1-umption}
\cB_1 \subseteq \{ f_1^k b_0 \mid k \in \NN \} \cup \{e_1^k b_0 \mid k \in \NN \}.
\ee
\end{proposition}

\begin{proof}
Suppose there exists such an 
element $b_0 \in \cB_1$.
Since the embedding $\Upsilon_1$ preserves weights and all statistics,
we must have $\weight(b_0) = 0$ and $\varepsilon_i(b_0) = \varphi_i(b_0)\leq 0$ for all $i \in [n-1]$ in order for $\Upsilon_1$ to be defined.
Notice that $f_1$ acts on the left factor of $(f_1^k b_0) \otimes u_{\infty}$
since 
\[ f_1\((f_1^k b_0) \otimes u_{\infty}\) = f_1 \Upsilon_1(f_1^k u_\infty) = \Upsilon_1(f_1^{k+1} u_\infty)  = (f_1^{k+1}b_0)\otimes  u_\infty,\] so it also follows that $\varphi_1(b_0) = \varepsilon_1(b_0) \geq \varphi_1(u_{\infty}) = 0$.
Thus $\varphi_1(b_0) = \varepsilon_1(b_0) = 0$.

Since $f_1 f_2 u_{\infty} \neq f_2 f_1 u_{\infty}$, which are both distinct from $u_{\infty}$, we claim that we must have
\[
f_2 f_1(b_0 \otimes u_{\infty}) = f_2\bigl((f_1 b_0) \otimes u_{\infty}\bigr) = (f_1 b_0) \otimes (f_2 u_{\infty})
\quand
f_1 f_2(b_0 \otimes u_{\infty}) = b_0 \otimes (f_1 f_2 u_{\infty}).
\]
For the first equation, we have $f_1$ acting on the left factor by the definition of $\Upsilon_1$, and if $f_2$ acts on the first factor, then $f_2 f_1 b_0 = b_0$ by our assumption \eqref{b1-umption} and weight considerations.
But $f_2f_1(b_0\otimes u_\infty) =\Upsilon_1(f_2f_1u_\infty)$ cannot be equal to
$ b_0\otimes u_\infty=\Upsilon_1(u_\infty)$
since $\Upsilon_1$ is an embedding.

For the second equation, note that 
if $\varepsilon_2(b_0) = 0$ then $f_2(b_0 \otimes u_{\infty}) = (f_2b_0)\otimes u_{\infty}$ must be zero since otherwise its weight would be $\weight(f_2b_0) = e_3-e_0$ which is not the weight of any element in the right side of \eqref{b1-umption}. 
But as $f_2u_\infty\neq \zero$, this would lead to the contradiction \[\zero \neq \Upsilon_1(f_2u_\infty) =  f_2\Upsilon_1(u_\infty) = f_2(b_0 \otimes u_{\infty})=\zero.\]
Hence $\varepsilon_2(b_0) <0$ so 
$f_2(b_0 \otimes u_{\infty}) = b_0\otimes f_2u_\infty$,
and
$
f_1f_2(b_0\otimes u_\infty) $ is either $(f_1b_0)\otimes(f_2u_\infty)$ or $b_0\otimes(f_1f_2u_\infty).
$
Yet the first case cannot occur as
\[
f_1f_2(b_0\otimes u_\infty) = \Upsilon_1(f_1f_2u_\infty) \neq \Upsilon_1(f_2f_1u_\infty) = (f_1b_0)\otimes(f_2u_\infty).
\]
Finally, for the second equation to hold, we need to have $\varepsilon_1(b_0) < \varphi_1(f_2 u_{\infty}) = 0$ by the tensor product rule, which is 
impossible.
Therefore, we have a contradiction and the result holds.
\end{proof}

If we wanted to extend the potential construction in the proof of Proposition~\ref{prop:no_kashiwara_here} a bit further, we would see that $f_1 f_2^2 f_1 u_{\infty} = f_2 f_1 u_{\infty}$ implies $f_1 f_2 f_1 b_0 = f_1 b_0$.
Hence, in order to construct a Kashiwara-like embedding for $\SetTab_n(\infty)$, any reasonable analog of $\cE_i$ must include $2$-cycles.

\begin{figure}[t]
\[
\begin{tikzpicture}[xscale=2, yscale=1.75,>=latex,baseline=(D0.base)]
\node at (1,0) (A1) {$000000$};
\node at (2,0) (A2) {$100000$};
\node at (1,-1) (B1) {$101000$};
\node at (2,-1) (B2) {$200000$};
\node at (0,-2) (C0) {$001000$};
\node at (1,-2) (C1) {$201000$};
\node at (2,-2) (C2) {$210000$};
\node at (0,-3) (D0) {$202000$};
\node at (1,-3) (D1) {$211000$};
\node at (2,-3) (D2) {$220000$};
\node at (0,-4) (E0) {$212200$};
\node at (1,-4) (E1) {$212000$};
\node at (2,-4) (E2) {$221000$};
\node at (0,-5) (F0) {$212210$};
\node at (1,-5) (F1) {$212100$};
\node at (2,-5) (F2) {$222000$};
\node at (0,-6) (G0) {$222210$};
\node at (1,-6) (G1) {$222200$};
\node at (2,-6) (G2) {$222100$};
\draw[->,thick,red]  (A1) -- (A2) node[midway,above,scale=0.75] {$2$};
\draw[->,thick,red]  (A2) -- (B2) node[midway,right,scale=0.75] {$2$};
\draw[->,thick,gray]  (A2) -- (B1) node[midway,right,scale=0.75] {$3$};
\draw[->,thick,gray]  (B1) -- (C0) node[midway,right,scale=0.75] {$3$};
\draw[->,thick,red]  (B1) -- (C1) node[midway,right,scale=0.75] {$2$};
\draw[->,thick,gray]  (B2) -- (C1) node[midway,right,scale=0.75] {$3$};
\draw[->,thick,blue]  (B2) -- (C2) node[midway,right,scale=0.75] {$1$};
\draw[->,thick,red]  ([xshift=0pt]C0.north) -- ([xshift=-8pt]B1.south) node[midway,left,scale=0.75] {$2$};
\draw[->,thick,gray]  (C1) -- (D0) node[midway,right,scale=0.75] {$3$};
\draw[->,thick,gray]  (C2) -- (D1) node[midway,right,scale=0.75] {$3$};
\draw[->,thick,blue]  (C1) -- (D1) node[midway,right,scale=0.75] {$1$};
\draw[->,thick,blue]  (C2) -- (D2) node[midway,right,scale=0.75] {$1$};
\draw[->,thick,gray]  (D1) -- (E1) node[midway,right,scale=0.75] {$3$};
\draw[->,thick,gray]  (D2) -- (E2) node[midway,right,scale=0.75] {$3$};
\draw[->,thick,blue]  (D0) -- (E1) node[midway,right,scale=0.75] {$1$};
\draw[->,thick,blue]  (D1) -- (E2) node[midway,right,scale=0.75] {$1$};
\draw[->,thick,blue]  (E1) -- (F2) node[midway,right,scale=0.75] {$1$};
\draw[->,thick,red]  (E1) -- (F1) node[midway,right,scale=0.75] {$2$};
\draw[->,thick,gray]  (E2) -- (F2) node[midway,right,scale=0.75] {$3$};
\draw[->,thick,red]  (F1) -- (E0) node[midway,left,scale=0.75] {$2$};
\draw[->,thick,blue]  (F1) -- (G2) node[midway,right,scale=0.75] {$1$};
\draw[->,thick,red]  (F2) -- (G2) node[midway,right,scale=0.75] {$2$};
\draw[->,thick,red]  (G2) -- (G1) node[midway,above,scale=0.75] {$2$};
\draw[->,thick,blue]  (E0) -- (F0) node[midway,right,scale=0.75] {$1$};
\draw[->,thick,blue]  (F0) -- (G0) node[midway,right,scale=0.75] {$1$};
     \end{tikzpicture}
\hspace{60pt}
    \begin{tikzpicture}[xscale=2, yscale=1.75,>=latex,baseline=(D0.base)]
\node at (1,0) (A1) {$000000$};
\node at (2,0) (A2) {$010000$};
\node at (1,-1) (B1) {$011000$};
\node at (2,-1) (B2) {$020000$};
\node at (0,-2) (C0) {$012000$};
\node at (1,-2) (C1) {$021000$};
\node at (2,-2) (C2) {$020100$};
\node at (-1,-3) (D0) {$012010$};
\node at (0,-3) (D1) {$022000$};
\node at (1,-3) (D2) {$021100$};
\node at (2,-3) (D3) {$020200$};
\node at (0,-4) (E0) {$022010$};
\node at (1,-4) (E1) {$022100$};
\node at (2,-4) (E2) {$021200$};
\node at (1,-5) (F1) {$022110$};
\node at (2,-5) (F2) {$022200$};
\node at (0,-5) (G1) {$022120$};
\node at (2,-6) (G2) {$022210$};
\node at (1,-6) (H1) {$022220$};
\draw[->,thick,red]  (A1) -- (A2) node[midway,above,scale=0.75] {$2$};
\draw[->,thick,red]  (A2) -- (B2) node[midway,right,scale=0.75] {$2$};
\draw[->,thick,gray]  (A2) -- (B1) node[midway,right,scale=0.75] {$3$};
\draw[->,thick,gray]  (B1) -- (C0) node[midway,right,scale=0.75] {$3$};
\draw[->,thick,red]  (B1) -- (C1) node[midway,right,scale=0.75] {$2$};
\draw[->,thick,gray]  (B2) -- (C1) node[midway,right,scale=0.75] {$3$};
\draw[->,thick,blue]  (B2) -- (C2) node[midway,right,scale=0.75] {$1$};
\draw[->,thick,red]  (C0) -- (D0) node[midway,right,scale=0.75] {$2$};
\draw[->,thick,gray]  (C1) -- (D1) node[midway,right,scale=0.75] {$3$};
\draw[->,thick,blue]  (C1) -- (D2) node[midway,right,scale=0.75] {$1$};
\draw[->,thick,gray]  (C2) -- (D2) node[midway,right,scale=0.75] {$3$};
\draw[->,thick,blue]  (C2) -- (D3) node[midway,right,scale=0.75] {$1$};
\draw[->,thick,red]  (D0) -- (E0) node[midway,right,scale=0.75] {$2$};
\draw[->,thick,blue]  (D2) -- (E2) node[midway,right,scale=0.75] {$1$};
\draw[->,thick,gray]  (D3) -- (E2) node[midway,right,scale=0.75] {$3$};
\draw[->,thick,blue]  (D1) -- (E1) node[midway,right,scale=0.75] {$1$};
\draw[->,thick,gray]  (D2) -- (E1) node[midway,right,scale=0.75] {$3$};
\draw[->,thick,blue]  (E0) -- (F1) node[midway,right,scale=0.75] {$1$};
\draw[->,thick,red]  (E1) -- (F1) node[midway,right,scale=0.75] {$2$};
\draw[->,thick,blue]  (E1) -- (F2) node[midway,right,scale=0.75] {$1$};
\draw[->,thick,gray]  (E2) -- (F2) node[midway,right,scale=0.75] {$3$};
\draw[->,thick,red]  (F1) -- (G1) node[midway,above,scale=0.75] {$2$};
\draw[->,thick,blue]  (F1) -- (G2) node[midway,right,scale=0.75] {$1$};
\draw[->,thick,red]  (F2) -- (G2) node[midway,right,scale=0.75] {$2$};
\draw[->,thick,blue]  (G1) -- (H1) node[midway,above,scale=0.75] {$1$};
\draw[->,thick,red]  (G2) -- (H1) node[midway,above,scale=0.75] {$2$};
     \end{tikzpicture}
\]
\caption{For $n = 4$ and $J = (2,1,3,2,1,3)$, the crystal graphs of $\sqrtB_J(\Lambda_2)$ (left) and
$\overline{\cB}^{\halfpower}_\BZL(\Lambda_2)$ 
 (right).
The word $ABCDEF$ stands for the element $r_{\Lambda_2} \otimes \uj{J}{-\frac{A}{2},-\frac{B}{2},-\frac{C}{2},-\frac{D}{2},-\frac{E}{2},-\frac{F}{2}}$ 
on the left and $r_{\Lambda_2} \otimes \elem{\BZL}{-\frac{A}{2},-\frac{B}{2},-\frac{C}{2},-\frac{D}{2},-\frac{E}{2},-\frac{F}{2}}$ on the right.
}
\label{rb-non-fig}
\end{figure}

\begin{remark}
\label{rem:non_BZL_failure}
One might hope that for reduced words $J$ for $w_0 \in S_n$ outside the restricted commutation class of $\BZL$, the crystals $\cB_J^{\halfpower}(\lambda)$ for $\lambda \in P^+$ are still interesting objects.
However, we do not have any version of Theorem~\ref{thm:standard_characters} for these crystals.
Indeed, for such words $J$ we typically have $\cB_{J}^{\halfpower}(\lambda) \not\iso \SetTab_n(\lambda)$ and while the crystal
$\cB_{J}^{\halfpower}(\lambda)$ is still finite, it can fail to be regular, and its character may not even be symmetric; for example, compare Figure~\ref{rb-fig} with Figure~\ref{rb-non-fig} (left).
\end{remark}

\subsection{Other definitions of square root crystals}

From the Kashiwara $\gl_n$-crystal construction of $\cB(\lambda)$ using $\cB(\infty)$ and the polyhedral model, it would be natural to guess that we could obtain $\sqrtB(\infty)$ by using the $\sqgln$-crystal $\cE_i$ from Example~\ref{ex:elem_crystal} and performing the analogous tensor product construction.
Indeed, the looped path crystal $\sqrtE_j$ is a somewhat unusual variant of the classical elementary crystals $\cE_j$.

However, for general $n$, no embedding as in Theorem~\ref{bzl-thm} holds if $\sqrtE_j$ is replaced by these crystals, and for sequences $J = (j_k \in I)_{k=1}^{\ell}$, the crystals $\overline \cB_{J}^{\halfpower}(\lambda)$, which we form by swapping out $\sqrtE_j$ for the crystal $\cE_j$ in the definition of $\cB_{J}^{\halfpower}(\lambda)$, exhibit similar pathologies.

\begin{remark}
To make this precise, 
choose a reduced word  $J = (j_1, \dotsc, j_{\ell})$ for  $w_0 \in S_n$
and define $\elem{J}{m_1,\dotsc,m_\ell} = \elem{j_1}{m_1}\otimes \cdots \otimes \elem{j_\ell}{m_\ell}$.
Then define $\overline{\cB}^{\halfpower}_J(\infty) \subseteq \cE_J := \cE_{j_1} \otimes \cdots \otimes \cE_{j_\ell}$ as the subcrystal lower generated by $\{\elem{J}{0,\dotsc,0}\}$, and 
 for  $\lambda \in P^+$ let $\overline{\cB}^{\halfpower}_J(\lambda) $
be the connected component of $
r_\lambda \otimes \elem{J}{0,\dotsc,0} \in \cR_{\lambda} \otimes \overline{\cB}^{\halfpower}_J(\infty).
$
The characters of these crystals are not always symmetric, even when using the BZL word, as can be checked using Figure~\ref{rb-non-fig} (contrast this with Figure~\ref{rb-fig}).
\end{remark}

\begin{example}
The $\sqgln$-crystals $\cE_i$ from Example~\ref{ex:elem_crystal} have $\cE_i \otimes \cE_j \otimes \cE_i \not\iso \cE_j \otimes \cE_i \otimes \cE_j$ for $\abs{i - j} = 1$, so do not generally satisfy a braid relation isomorphism.
We can see this by noting in $\cE_{(1,2,1)}$ that  
\[
f_1 f_2 \elem{(1,2,1)}{0,0,0} = \elem{(1,2,1)}{-\tfrac{1}{2},-\tfrac{1}{2},0} = f_2 f_1 \elem{(1,2,1)}{0,0,0},
\]
whereas in $\cE_{(2,1,2)}$, we have distinct elements
\begin{align*}
f_1 f_2 \elem{(2,1,2)}{0,0,0} & = \elem{(2,1,2)}{-\tfrac{1}{2},-\tfrac{1}{2},0},
& 
f_2 f_1 \elem{(2,1,2)}{0,0,0} & = \elem{(2,1,2)}{0,-\tfrac{1}{2},-\tfrac{1}{2}}.
\end{align*}
Since the zero-coordinate triple is the unique element 
in either tensor product with weight $0$ and 
whose values under $\varepsilon_1$ and $\varepsilon_2$ are both zero,
the non-isomorphism follows.

However, it does hold that $\cE_i \otimes \cE_j \iso \cE_j \otimes \cE_i$ whenever $\abs{i - j} > 1$.
\end{example}

These observations are again reflecting the asymmetry between $\alpha_i^{(0)}$ and $\alpha_i^{(1)}$.

\subsection{Dualities}
\label{sec:dualities}

For $\gl_n$-crystals, there are three primary dualities that come from (anti)involutions of an associated quantum group and/or root system:
\begin{itemize}
\item[$\dagger$:] \defn{Lusztig duality}, which is given by the action of $w_0$. For $\gl_n$ this corresponds to interchanging the crystal operators $e_i \leftrightarrow f_{n-i}$ and twisting the weight map by $w_0$ (see, e.g.,~\cite{Lusztig93,Lenart07}).

\item[$*$:] The \defn{$*$-involution}, which is only defined on $\cB(\infty)$ and which was introduced in~\cite{Kashiwara93}.
  This duality is closely related to the Kashiwara embedding and can be used to characterize $\cB(\infty)$.
  
  \item[$\vee$:] \defn{Contragredient duality} given by interchanging the crystal operators $e_i \leftrightarrow f_i$ and negating the weight map (first introduced in~\cite{Kashiwara93}).
\end{itemize}

Both the Lusztig dual and the contragredient dual operations send highest weight elements to lowest weight elements.

The definition of $\dagger$ naturally extends to all $\sqgln$-crystals. Specifically, this is defined just as for $\gl_n$-crystals, except that $w_0$ twists the weight map using the action
\[
w_0 \colon\lambda_0 \e_0 + \lambda_1\e_1 + \lambda_2 \e_2 + \dots +\lambda_n \e_n 
\mapsto 
\lambda_0 \e_0 + \lambda_n \e_1 + \dots + \lambda_2 \e_{n-1}+ \lambda_1\e_n.
\]
The operation $\dagger$ acts functorially as for Kashiwara crystals (see, e.g.,~\cite{Kashiwara95,Lenart07})
and one has $(\cB \otimes \cC)^{\dagger} \iso \cC^{\dagger} \otimes \cB^{\dagger}$ by \cite[Prop.~4.41]{MT2023}.
The subcategory $\sqrt{\sC}_{\poly}$ of polynomial $\sqgln$-crystals is closed under $\dagger$
and 
 if $\lambda=k_n \Lambda_n + k_i \Lambda_i$ for $k_i, k_n\in \NN$, then $\SetTab_n(\lambda)^{\dagger} \iso \SetTab_n(\lambda)$ \cite[Props.~4.41 and 4.42]{MT2023}.

The isomorphism $\SetTab_n(\lambda)^{\dagger} \iso \SetTab_n(\lambda)$ does not hold for general $\lambda\in P^+$, in contrast to Kashiwara crystals of tableaux.
For example, a direct computation shows that if $\lambda = \Lambda_1+\Lambda_2$ 
then
 $\SetTab_3(\lambda)^{\dagger} \not\iso \SetTab_3(\lambda)$; cf. \cite[Fig.~4.1]{Yu23}.
However, $\SetTab_3(\Lambda_1+\Lambda_2)^{\dagger}$ is isomorphic to the connected component of $\SS_3^{\otimes 3}$ containing $\{1\} \otimes \{2\} \otimes \{1\}$.

As our final result, we mention
 dual analogs to Theorems~\ref{thm:crystal_SVT} and \ref{thm:standard_characters} using $\dagger$:

\begin{corollary}
The following properties hold:

\begin{enumerate}
\item[(a)] Suppose $\cB \in \sqrt{\sC}_{\poly}$ has lowest weight elements of weight $\lambda^{(1)},\lambda^{(2)}, \dotsc, \lambda^{(k)}$. Then
\[
\ch(\cB) = \sum_{m=1}^k G_{w_0\lambda^{(m)}}(x; \beta).
\]
\item[(b)] There exists a filtration of $\SetTab_n(\lambda)^{\dagger}$ by $w \in S_n$ relative to Bruhat order given by
\[
\SetTab_n(\lambda)^{\dagger}_w := \fkD_{i_1}^{\dagger} \cdots \fkD_{i_N}^{\dagger} \{u_{\lambda}^{\ddagger}\},
\]
where $(i_1, \dotsc, i_N)$ is any Hecke word for $w$, where the \defn{dual crystal Demazure operator} $\fkD_i^{\dagger}$ is
\begin{subequations}
\begin{align}
\fkD^{\dagger}_i X & := \{ b \in \cB \mid f_i^k b\in X\text{ for some }k \geq 0\}
\\ & = \{ e_i^k x \mid k\geq 0 \text{ and } x \in X\} \setminus\{\zero\},
\end{align}
\end{subequations}
and where $u_{\lambda}^{\ddagger}$ is the unique lowest weight element of $\SetTab_n(\lambda)^{\dagger}$.
\end{enumerate}
\end{corollary}

We do not currently know how to extend the $*$-involution to $\sqgln$-crystals, because we do not have a Kashiwara embedding analog for arbitrary $i$ (only for the initial prefix of the BZL word~\eqref{bzl-eq}).
The category of $\sqgln$-crystals is also not closed under the natural extension of contragredient duality $\vee$ as this would need to interchange the weights $\alpha_i^{(0)}$ and $\alpha_i^{(1)}$.
This suggests that there may not exist a good analog of the Lie algebra $\gl_n$ or the quantum group $U_q(\gl_n)$ for $\sqgln$-crystals.

On the other hand, the asymmetry between $\alpha_i^{(0)}$ and $\alpha_i^{(1)}$ 
suggests that we can build an interesting category of \defn{dual $\sqgln$-crystals} by using the other natural choice of $\bfN$-root data with the simple roots split by $\overline{\alpha}_i^{(0)} = \e_0 - \e_{i+1}$ and $\overline{\alpha}_i^{(1)} = \e_i-\e_0 $.
We expect that the analogs of looped $j$-path crystals  in this category will  involve $\{j-1, j\}$ instead of $\{j, j+1\}$.

The contragredient duality $\vee$ for $\gl_n$-crystals, rather than preserving $\sqrt{\sC}$, extends to a functor between $\sqrt{\sC}$ and the category of these dual objects. This is illustrated in the following example.

\begin{example}
\label{ex:dual_nonclosure}
Consider the standard $\sqrt{\gl_3}$-crystal $\SS_3$, which has a unique highest weight element of weight $\Lambda_1$.
The contragredient dual crystal $\SS_3^{\vee}$ also has a unique highest weight element, but now of weight $-w_0 \Lambda_1 = \Lambda_2 - \Lambda_3$.
Its crystal graph is given by reversing all arrows in   $\SS_3$. 
\end{example}

We leave  the following as an open problem.

\begin{problem}
Determine if the category of ``dual $\sqgln$-crystals'' has  
a polynomial subcategory
analogous to the constructions given in this paper,
for which there are ``dual'' versions of the crystals $\SetTab_n(\lambda)$ 
and of Theorems~\ref{thm:crystal_SVT} and \ref{thm:standard_characters}.
\end{problem}

We mention one other future direction. 
This paper has focused on  $\bfN$-root crystals corresponding to the root system and weight lattice of type $\gl_n$. 
There appear to be interesting variants of this category for both of the Lie superalgebras $\q_n$ and $\gl_{m|n}$ in type A.
Some results on (regular) square root crystals for the queer Lie superalgebra $\q_n$ already appear in \cite{MT2023}.
On the other hand, \cite{SalScr} studies a limiting object $\cB(-\infty)$ for polynomial $\q_n$-crystals as developed in \cite{GJKKK10,GJKKK}.

\begin{problem}
Extend the $\cB(-\infty)$ construction in \cite{SalScr} to square root crystals. More generally, develop the basic properties of the Lie superalgebra versions of square root crystals in type A and investigate positivity phenomena exhibited by these objects, analogous to Theorem~\ref{thm:standard_characters}. 
\end{problem}

\bibliographystyle{alpha}
\bibliography{crystals}

\end{document}